\documentclass[11pt]{article}
\usepackage[a4paper,margin=28mm]{geometry}
\usepackage{amsmath,amssymb,amsthm,mathtools}
\usepackage{mathrsfs}
\usepackage{microtype}
\usepackage{xcolor}
\usepackage[colorlinks=true,linkcolor=blue!55!black,citecolor=blue!55!black,urlcolor=blue!55!black]{hyperref}
\usepackage[nameinlink,capitalize]{cleveref}

\newtheorem{theorem}{Theorem}[section]
\newtheorem{proposition}[theorem]{Proposition}
\newtheorem{lemma}[theorem]{Lemma}
\newtheorem{corollary}[theorem]{Corollary}
\theoremstyle{remark}
\newtheorem{remark}[theorem]{Remark}

\newcommand{\B}{\mathbb B}
\newcommand{\Sph}{\mathbb S}
\newcommand{\N}{\mathcal N}
\newcommand{\E}{\mathcal E}

\newcommand{\dd}{\,\mathrm d}

\title{Sharp observation time for wave equations with an inverse-square boundary potential}
\author{Longben Wei}
\date{}

\begin{document}
\maketitle

\begin{abstract}
We determine the sharp observation-time threshold for the autonomous, pure-potential wave equation with an inverse-square boundary singularity on the unit ball.
For every fixed $\kappa<0$, full-boundary observability holds on $(-T,T)$ whenever $T>1$ and fails when $T<1$. This resolves, for this model, the optimal-time question raised by Enciso, Shao, and Vergara [J. Eur. Math. Soc. 23 (2021), 3459--3495]
and extends their parameter range $-\frac12<\kappa<0$
to all $\kappa<0$, including the limit-point regime
$\kappa\leq-\frac12$.

Our approach is based on the Ingham--Beurling theory of
nonharmonic Fourier series.
Spherical harmonic decomposition reduces the problem to
one-dimensional singular operators, for which we establish
frequency separation, local counting estimates, and lower
bounds for the boundary norming constants, uniformly in the
angular parameter.
A quantitative Ingham--Beurling inequality then yields the
sharp-time observability estimate.

As a by-product, this spectral approach provides an independent
proof of boundary observability and an explicit functional
characterization of its maximal $L^2$ observation domain.
The latter is described by a weighted $\ell^2$ condition on
the spectral coefficients and consists precisely of energy
solutions whose distributional Neumann trace belongs to $L^2$.
It contains the boundary-admissible energy solutions considered
in Enciso--Shao--Vergara's original work.
\end{abstract}
\tableofcontents
\clearpage
\section{Introduction}

Our aim is to sharpen the boundary observability inequality
of \cite{ESV2021} by establishing the optimal observation
time for the unperturbed singular wave equation on the
unit ball by a spectral approach.

Let $n\geq2$, $\B:=B(0,1)\subset\mathbb R^n$,
$r=|x|$, and $y=1-r$. We consider the initial-boundary
value problem
\begin{equation}\label{eq:wave}
    \begin{cases}
        \Box_\kappa u
        :=-\partial_t^2u+\Delta u
        +\dfrac{\kappa(1-\kappa)}{y^2}u=0,
        & \text{in }\mathbb R\times\B,\\[1mm]
        \mathcal D_\kappa u=0,
        & \text{on }\mathbb R\times\partial\B,\\[1mm]
        u(0,\cdot)=u_0,\qquad
        \partial_tu(0,\cdot)=u_1,
        & \text{in }\B,
    \end{cases}
\end{equation}
where $\kappa$ measures the strength of the potential,
following the notation of \cite{ESV2021}.
We use the same weighted boundary traces as in
\cite[equation~(2.11)]{ESV2021}:
\begin{equation}\label{eq:weighted-boundary-traces}
    \mathcal D_\kappa u
    :=
    \lim_{r\uparrow1}y^{-\kappa}u,
    \qquad
    \mathcal N_\kappa u
    :=
    \lim_{r\uparrow1}
    y^{2\kappa}\partial_r(y^{-\kappa}u),
\end{equation}
initially defined for functions for which the interior
expressions and the indicated limits are well defined.
In the boundary-admissible setting of
\cite[Definition~2.2]{ESV2021}, the boundary limits
are taken strongly in
$L^2((-T,T)\times\Sph^{n-1})$.
Using Carleman estimates adapted to the inverse-square
boundary singularity, Enciso, Shao, and Vergara
established boundary observability for sufficiently
large observation times. For the unperturbed equation
\eqref{eq:wave}, their result takes the following form.

\begin{theorem}[Enciso--Shao--Vergara
{\cite[Theorem~5.1]{ESV2021}}]
\label{thm:esv-observability}
Fix $n\geq3$ and $-\frac12<\kappa<0$.
There exists $T_0=T_0(n,\kappa)>0$ such that,
for every $T>T_0$, there is a constant
$C_{T,n,\kappa}>0$ with the following property.

Let $u$ be a real-valued solution of
\eqref{eq:wave} on $(-T,T)\times\B$ that is
boundary admissible in the sense of
\cite[Definition~2.2]{ESV2021}.
Assume that its twisted energy
\[
    E_{\mathrm{ESV}}[u](t)
    :=
    \int_\B
    \left(
        |\partial_tu|^2
        +\bigl|y^\kappa\nabla(y^{-\kappa}u)\bigr|^2
        +|u|^2
    \right)\dd x
\]
is finite for every $t\in(-T,T)$.
Then
\begin{equation}\label{eq:esv-observability}
    E_{\mathrm{ESV}}[u](0)
    \leq
    C_{T,n,\kappa}
    \int_{-T}^{T}\int_{\Sph^{n-1}}
        |\mathcal N_\kappa u|^2
        \,\dd\vartheta\,\dd t.
\end{equation}
\end{theorem}

The optimal observation time was left open in
\cite[Remark~1.11]{ESV2021}.
Our aim is to establish the sharp threshold $T>1$
for the unperturbed model, extend the parameter range,
and formulate the estimate on a larger class of
energy solutions.
More recently, the large-time boundary observability
result has also been extended to dimension two;
see \cite[Theorem~1.13]{JenaShao2026}.

Exploiting the spherical geometry and the radial
symmetry of the potential, we adopt a spectral
approach based on Ingham--Beurling inequalities,
rather than Carleman estimates.
The use of Ingham-type inequalities in control theory
has a long history, including the works of
Haraux \cite{Haraux1989},
Jaffard \cite{Jaffard1988,Jaffard1990},
and Komornik \cite{Komornik1992};
see \cite{KomornikLoreti2005} for a comprehensive
account of the classical theory. Subsequent applications include Gueye's spectral
approach to boundary observability for
one-dimensional degenerate wave equations
\cite{Gueye2014}.
More recently, T\"aufer \cite{Taeufer2023} combined
Floquet--Bloch decomposition with multidimensional
Ingham-type inequalities to establish internal
observability for the free Schr\"odinger equation
on $\mathbb R^d$ from any nonempty periodic open set
in any positive time.
In a recent preprint, Haak, Jaming, Wang, and Wang
\cite{HaakJamingWangWang2026} established curved
Ingham inequalities for dispersive trigonometric
sums restricted to curved trajectories, with
applications to observability for fractional
Schr\"odinger equations on the torus.

Our approach is directly inspired by
\cite[Chapter~7]{KomornikLoreti2005}, where spherical
harmonic expansions and nonharmonic Fourier estimates
are combined to establish boundary observability
on balls. We follow this strategy to obtain the sharp
observation-time threshold for the present singular
model. The main difficulty is to obtain spectral
estimates uniform in the angular parameter: fixed-order
eigenvalue asymptotics do not provide the uniform
constants needed to sum over all spherical harmonic
channels. Moreover, the additional inverse-square
singularity at the observed boundary prevents a direct
use of the explicit Bessel formulas. We therefore
establish uniform frequency separation, a quantitative
counting bound with the sharp leading coefficient,
and a lower bound for the boundary observation
coefficients.

We use the abstract wave-equation framework of
\cite[Chapter~3]{KomornikLoreti2005} and
\cite[Sections~3.7--3.8]{TucsnakWeiss2009}.
Throughout, let $\kappa<0$ and set
\[
    \mu:=\frac12-\kappa>\frac12,
    \qquad
    a:=-\kappa(1-\kappa)=\mu^2-\frac14>0.
\]
The parameter $\mu$ puts the boundary singularity
in the standard Bessel form
$(\mu^2-\frac14)y^{-2}$, with characteristic exponents
$\frac12\pm\mu$, and will be convenient for describing
the principal solution and its endpoint asymptotics.
It corresponds to $l=\mu-\frac12$ in
\cite{KST2010, KostenkoTeschl2011}.
The symbol $a$ is simply an abbreviation for the
positive potential coefficient; these parameters
are all determined by $\kappa$.

On the unit ball, the boundary Hardy inequality reads
\[
    \int_\B|\nabla v|^2\,\dd x
    \geq
    \frac14\int_\B\frac{|v|^2}{y^2}\,\dd x,
    \qquad v\in H_0^1(\B),
\]
where the constant $1/4$ is optimal and independent
of the dimension; see
\cite[Proposition~2.1]{BiccariZuazua2016}.
Consequently, for every $\mu\geq0$, the symmetric operator
\[
    -\Delta+\frac{\mu^2-\frac14}{y^2},
    \qquad D=C_c^\infty(\B),
\]
is nonnegative and therefore admits a unique
Friedrichs extension. For $\mu>0$, the form domain of the Friedrichs extension is $H_0^1(\B)$, and its form norm is equivalent to the $H_0^1$ norm.
At the critical value $\mu=0$, corresponding to the
inverse-square coefficient $-1/4$, the Friedrichs
form domain is strictly larger than $H_0^1(\B)$;
see the discussion following equation~(2.7) in
\cite{BiccariZuazua2016}.
We denote this realization by
\[
    A_\kappa=-\Delta+\frac{a}{y^2},
    \qquad
    D(A_\kappa^{1/2})=H_0^1(\B).
\]

With this realization, the wave equation takes the
abstract form
\[
    \partial_t^2u+A_\kappa u=0.
\]
Its conserved form energy is (see \cite[Proposition~3.8.7]{TucsnakWeiss2009})
\begin{equation}\label{eq:energy}
    \E_\kappa[u](t)
    :=
    \frac12\int_{\B}
    \left(
        |\partial_tu|^2+|\nabla u|^2
        +\frac{a}{y^2}|u|^2
    \right)\dd x.
\end{equation}
We will show in \cref{app:energy-equivalence} in the common parameter range that $\E_\kappa[u](t)$ is equivalent to the twisted energy $E_{\mathrm{ESV}}[u](t)$. For every initial state
$(u_0,u_1)\in H_0^1(\B)\times L^2(\B)$,
the functional calculus for $A_\kappa$ yields a unique
solution
\[
    u\in C(\mathbb R;H_0^1(\B))
    \cap C^1(\mathbb R;L^2(\B))
\]
of $\partial_t^2u+A_\kappa u=0$ in the variational
sense
\footnote{That is, for every $v\in D(A_\kappa^{1/2})$,
\[
    \frac{d^2}{dt^2}(u(t),v)_{L^2(\B)}
    +
    (A_\kappa^{1/2}u(t),A_\kappa^{1/2}v)_{L^2(\B)}
    =0
\]
in $\mathcal D'(\mathbb R)$;
cf.~\cite[Section~7.2]{Evans2010}.}, with $u(0)=u_0$ and $\partial_tu(0)=u_1$. Here and throughout,
we refer to solutions in this class as Friedrichs energy solutions (or simply energy solutions).

The homogeneous boundary condition in \eqref{eq:wave}
is imposed through the Friedrichs realization
$A_\kappa$. Every such energy solution satisfies
$\mathcal D_\kappa u=0$ in the strong $L^2$ sense
on bounded time intervals, as shown in \cref{remark1.1}.

For general energy solutions, the weighted Neumann
trace is understood distributionally, following the
construction in \cref{prop:extended-weak-trace}.
With this interpretation, our main result is as follows.
\begin{theorem}\label{thm:target}
Fix $n\geq2$ and $\kappa<0$.
For every $T>1$, there exists a constant
$C_{T,n,\kappa}>0$ such that every Friedrichs energy
solution $u$ of \eqref{eq:wave} whose weak Neumann trace
satisfies
\[
    \mathcal N_\kappa u
    \in L^2\bigl((-T,T)\times\Sph^{n-1}\bigr)
\]
obeys
\begin{equation}\label{eq:obs}
    \E_\kappa[u](0)
    \leq C_{T,n,\kappa}
    \int_{-T}^{T}\int_{\Sph^{n-1}}
        |\mathcal N_\kappa u|^2
        \,\dd\vartheta\,\dd t.
\end{equation}

Conversely, for every $0<T<1$, no constant
$C_{T,n,\kappa}>0$ makes \eqref{eq:obs} valid for all
finite spectral-sum solutions.
Thus the observation-time threshold $T_*=1$ is sharp.
\end{theorem}
The following remarks highlight the contributions of
our main theorem and clarify its relation to, and
differences from, the results of \cite{ESV2021}.
\begin{remark}\label{remark1.1}
Every Friedrichs energy solution satisfies
$\mathcal D_\kappa u=0$ strongly in
$L^2((-T,T)\times\Sph^{n-1})$ for every $T>0$.
Indeed, the map $r\mapsto u(\,\cdot\,,r\,\cdot\,)$
belongs to
$H^1((1/2,1);L^2((-T,T)\times\Sph^{n-1}))$
and has zero trace at $r=1$.
The fundamental theorem of calculus and
Cauchy--Schwarz give
\[
    \bigl\|
        y^{-\kappa}u(\,\cdot\,,(1-y)\,\cdot\,)
    \bigr\|_{L^2((-T,T)\times\Sph^{n-1})}^2
    \leq
    C_n y^{1-2\kappa}
    \int_{-T}^{T}\|\nabla u(t)\|_{L^2(\B)}^2\,\dd t
    \longrightarrow0.
\]
Thus the Friedrichs formulation imposes the homogeneous
weighted Dirichlet condition used in \cite{ESV2021}.

The weak Neumann trace agrees with the classical
weighted trace under the boundary-admissibility
assumptions of \cite{ESV2021}; see
\cref{prop:app-inclusion}.
The initial states for which this weak trace belongs
to $L^2((-T,T)\times\Sph^{n-1})$ are characterized
in \cref{prop:observation-domain}.

For $n\geq3$ and $-\frac12<\kappa<0$,
\cref{prop:app-energy-comparison,prop:app-inclusion}
show that our estimate applies to every
boundary-admissible energy solution of the unperturbed
homogeneous Dirichlet problem considered in
\cite[Theorem~5.1]{ESV2021}, with the same boundary
observation and an equivalent energy.
Consequently, \cref{thm:target} sharpens the
observation-time range for that model, extends the
parameter range to every fixed $\kappa<0$, and
includes dimension two.
Moreover, \cref{rem:app-strict-inclusion} shows that,
in the common parameter range, our observation class
strictly contains the boundary-admissible energy
class of \cite{ESV2021}.
\end{remark}
\begin{remark}
Our parameter range includes both
$-\frac12<\kappa<0$, considered in \cite{ESV2021},
and $\kappa\leq-\frac12$, corresponding respectively
to $\frac12<\mu<1$ and $\mu\geq1$.
The radial endpoint $r=1$ is limit-circle in the
former case and limit-point in the latter;
see \cite[Theorem~2.4]{KST2010}.
Thus no additional self-adjoint boundary condition
is required at $r=1$ in the latter range.
For $0<\kappa<\frac12$, corresponding to
$0<\mu<\frac12$, the endpoint remains limit-circle,
but the modal counterexample in \cref{rem:positive-kappa-failure} shows that
the boundary observability inequality considered here
fails for every finite observation time. The critical Hardy case $\kappa=\frac12$ ($\mu=0$)
lies outside the scope of this paper:
the Friedrichs condition selects the logarithm-free
branch and no longer implies $\mathcal D_\kappa u=0$
with the trace convention
\eqref{eq:weighted-boundary-traces}.
\end{remark}
\begin{remark}
The sharp threshold $T_*=1$ coincides with that for
the free Dirichlet wave equation on the unit ball;
see \cite[Section~2.3, Theorem~2.2]{Zuazua2024} and the references therein.
Here $T$ denotes the half-length of the observation
interval $(-T,T)$.

For the free wave equation, the lower time obstruction
reflects finite propagation at unit speed: for $T<1$,
waves with initial data supported sufficiently close
to the centre remain away from the boundary throughout
$(-T,T)$. The same finite-propagation argument applies to the present wave equation with an inverse-square boundary potential. Together with approximation by finite spectral sums, it shows that \eqref{eq:obs} fails for $T<1$, even when restricted to finite spectral-sum solutions. Combined with the estimate for $T>1$, this establishes
the sharp threshold $T_*=1$. The critical case $T=1$ remains open.
\end{remark}

\subsection{Related results}
Evolution equations with inverse-square potentials have
attracted considerable attention because of the analytical
difficulties associated with their critical singularities.
Dispersive estimates have been established for potentials
singular at an isolated point
\cite{BarceloRuizVega2006,BurqEtAl2004}, and related
questions have been studied for singularities along
timelike hypersurfaces \cite{Bachelot2011}.
In the closely related setting of asymptotically
anti-de Sitter spacetimes, Warnick \cite{Warnick2013}
developed a well-posedness theory for wave equations
with Dirichlet, Neumann, and Robin boundary conditions
in an appropriate range of the mass parameter.

Observability and controllability for equations with
inverse-square potentials on bounded domains have also
been extensively investigated.
Vancostenoble and Zuazua \cite{VancostenobleZuazua2009}
established boundary observability and exact boundary
controllability for wave and Schr\"odinger equations
with an interior point singularity, under suitable
geometric conditions on the observed boundary portion
and restrictions on the potential strength.
Their analysis combines multiplier identities with
sharp Hardy inequalities.
Cazacu \cite{Cazacu2012} treated the corresponding
problems when the singular point lies on the boundary,
using boundary Hardy inequalities and a Pohozaev identity
to obtain observability and controllability results. 
For heat equations, null controllability and the
associated observability estimates have been studied
for interior point singularities
\cite{VancostenobleZuazua2008,Ervedoza2008},
boundary point singularities \cite{Cazacu2014},
and potentials singular along the entire boundary
\cite{BiccariZuazua2016}.
In particular, Biccari and Zuazua \cite{BiccariZuazua2016}
considered an inverse-square potential defined in terms
of the distance to the boundary and established null
controllability by an interior distributed control.

In earlier joint work \cite{WeiDuanXu2026}, we studied
exact observability for Schr\"odinger equations with
inverse-square potentials on the half-line using
harmonic analysis, particularly a Logvinenko--Sereda
type theorem for the Hankel transform.
This study motivates our use of spherical harmonic
decomposition to reduce the present problem to a
family of one-dimensional singular equations.
\paragraph{Organization of the paper.}
Section~2 gives the exact spherical reduction,
the spectral representation of energy solutions,
and the definition of the observation domain.
Section~3 reduces observability to uniform spectral
estimates through a quantitative Ingham--Beurling
inequality and deduces the main theorem from these
estimates and the results of Appendix~A.
Section~4 proves the uniform density and frequency
gap estimates.
Appendix~A constructs the distributional Neumann
trace, characterizes the maximal $L^2$ observation
domain, proves the lower time obstruction, and
establishes an independent lower bound for the
observation coefficients.
Appendix~B compares the energies and solution
classes with those of \cite{ESV2021} and identifies
the corresponding boundary traces.
\paragraph{AI disclosure.}
The author's interest in this problem and the main ideas
underlying the present approach originated during
the author's earlier joint work \cite{WeiDuanXu2026}.
AI(OpenAI’s GPT-5.6 Sol and GPT-6 Astra) tools were subsequently used to assist with proof development, drafting, language editing, and LaTeX
preparation. The author directed the mathematical development,
identified gaps in preliminary arguments, and determined
the revisions needed to obtain the stated results.
In particular, the author guided the extension of
arguments for finite spectral sums to general energy
solutions through the construction of a distributional
boundary trace and the characterization of its maximal
$L^2$ observation domain. The author also directed the reformulation of the spectral-gap argument using globally defined principal
endpoint solutions from the theory of perturbed Bessel
operators developed in \cite{KST2010, KostenkoTeschl2011}.
The author checked and rewrote all AI-generated material
incorporated into the manuscript and takes full
responsibility for the mathematical content of
this article.
\paragraph{Notation.}
We write $\mathbb N=\{1,2,\ldots\}$,
$\mathbb N_0=\mathbb N\cup\{0\}$, and
$s_+:=\max\{s,0\}$.
For $T>0$, set
\[
    I_T:=(-T,T),
    \qquad
    \Gamma_T:=I_T\times\Sph^{n-1},
\]
where $\dd\vartheta$ denotes the standard surface
measure on $\Sph^{n-1}$.
Inner products are linear in the first argument.
We use the corresponding convention for
distributional pairings, so that
$\langle f,\psi\rangle=\int f\overline{\psi}$
for locally integrable $f$.

The indices $\ell\in\mathbb N_0$, $1\leq m\leq d_\ell$,
and $k\in\mathbb N$ denote the spherical harmonic
degree, its multiplicity index, and the radial
eigenvalue index, respectively; sums over
$\ell,m,k$ run over these ranges.
The fixed parameter $\mu$ is suppressed in
$\omega_{\nu,k}$, $\phi_{\nu,k}$, $c_{\nu,k}$,
and $\gamma_{\nu,k}$.
We denote by $j_{\rho,k}$ the $k$th positive zero
of the Bessel function $J_\rho$.
Positive constants $c$ and $C$ may change from line
to line, and their relevant dependence is indicated
in the statements.
In estimates stated to be uniform in the modes,
these constants are independent of $\nu$ and $k$.
For nonnegative quantities, $X\lesssim Y$ means
$X\leq CY$, and $X\asymp Y$ means that both
$X\lesssim Y$ and $Y\lesssim X$ hold.
\section{Exact spherical reduction}

In this section, we decompose $A_\kappa$ into radial
Friedrichs operators, derive the corresponding
one-dimensional wave equations, and describe the
endpoint asymptotics and weighted Neumann traces
of the associated eigenfunctions.
We follow the spherical decomposition used in
\cite[Section~10.4]{TeschlQM2009}, identifying the
self-adjoint realizations through their closed
quadratic forms.

\paragraph{The orthogonal decomposition of the Hilbert space.}

Let
\[
    \{Y_{\ell,m}:1\leq m\leq d_\ell\}
\]
be an orthonormal basis of the degree-\(\ell\) spherical
harmonics, so that
\[
    -\Delta_{\Sph^{n-1}}Y_{\ell,m}
    =
    \ell(\ell+n-2)Y_{\ell,m},
    \qquad \ell\in\mathbb N_0.
\]
The union of these bases is an orthonormal basis of
\(L^2(\Sph^{n-1})\).

Set
\[
    \mathcal R
    :=
    L^2\bigl((0,1),r^{n-1}\dd r\bigr)
\]
and define the closed subspaces
\[
    \mathcal K_{\ell,m}
    :=
    \left\{
        R(r)Y_{\ell,m}(\vartheta):
        R\in\mathcal R
    \right\}
    \subset L^2(\B).
\]
The map \(R\mapsto RY_{\ell,m}\) is an isometry from
\(\mathcal R\) onto \(\mathcal K_{\ell,m}\).
Completeness and orthogonality of the spherical harmonics
give the Hilbert-space decomposition
\begin{equation}\label{eq:angular-hilbert-decomposition}
    L^2(\B)
    =
    \bigoplus_{\ell\geq0}
    \bigoplus_{m=1}^{d_\ell}\mathcal K_{\ell,m}.
\end{equation}
The orthogonal projection onto \(\mathcal K_{\ell,m}\) is
\[
    (P_{\ell,m}f)(r,\vartheta)
    =
    \left(
        \int_{\Sph^{n-1}}
        f(r,\eta)\overline{Y_{\ell,m}(\eta)}
        \,\dd\eta
    \right)Y_{\ell,m}(\vartheta).
\]
In particular, for every \(f\in L^2(\B)\),
\[
    f=\sum_{\ell,m}P_{\ell,m}f
    \quad\text{in }L^2(\B),
    \qquad
    \|f\|_{L^2(\B)}^2
    =
    \sum_{\ell,m}\|P_{\ell,m}f\|_{L^2(\B)}^2.
\]
This decomposition applies to arbitrary \(L^2\) functions.
\paragraph{Reduction to radial Friedrichs operators.}

The Friedrichs operator $A_\kappa$ is associated with
the closed form
\[
    \mathfrak a_\kappa[f,g]
    =
    \int_\B
    \left(
        \nabla f\cdot\overline{\nabla g}
        +\frac{a}{(1-r)^2}f\overline g
    \right)\dd x,
    \qquad
    D(\mathfrak a_\kappa)=H_0^1(\B).
\]
The space $\mathcal C=C_c^\infty(\B\setminus\{0\})$
is a form core: removing the origin does not change
the $H_0^1$ closure for $n\geq2$, using logarithmic
cutoffs when $n=2$.

Polar integration and spherical harmonic orthogonality give
\[
    P_{\ell,m}\mathcal C\subset\mathcal C,
    \qquad
    \mathfrak a_\kappa[P_{\ell,m}f,g]
    =\mathfrak a_\kappa[f,P_{\ell,m}g],
    \qquad
    \mathfrak a_\kappa[P_{\ell,m}f]
    \leq\mathfrak a_\kappa[f].
\]
By form-core approximation, these identities extend
to the form domain. The first representation theorem
\cite[Chapter~VI, Theorem~2.1(iii)]{Kato1995}
then yields
\[
    P_{\ell,m}D(A_\kappa)\subset D(A_\kappa),
    \qquad
    A_\kappa P_{\ell,m}=P_{\ell,m}A_\kappa
    \quad\text{on }D(A_\kappa).
\]
Thus $\mathcal K_{\ell,m}$ reduces $A_\kappa$.
Denoting its self-adjoint restriction by $A_{\ell,m}$,
we obtain
\[
    A_\kappa=\bigoplus_{\ell,m}A_{\ell,m}.
\]

Put
\[
    \alpha:=\frac{n-1}{2},
    \qquad
    \nu_\ell:=\ell+\frac{n-2}{2},
    \qquad
    \mu:=\frac12-\kappa,
    \qquad
    a=\mu^2-\frac14.
\]
The unitary map
\[
    U_{\ell,m}:\mathcal K_{\ell,m}\to L^2(0,1),
    \qquad
    U_{\ell,m}(RY_{\ell,m})=r^\alpha R,
\]
removes the radial measure; cf.
\cite[Section~10.4]{TeschlQM2009}.
For $v\in C_c^\infty(0,1)$, direct computation gives
\begin{equation}\label{eq:radial-conjugation}
    -\Delta(r^{-\alpha}vY_{\ell,m})
    =
    r^{-\alpha}
    \left(
        -v''+\frac{\nu_\ell^2-\frac14}{r^2}v
    \right)Y_{\ell,m}.
\end{equation}
Consequently,
\[
    \mathfrak a_\kappa[U_{\ell,m}^{-1}v]
    =
    \int_0^1
    \left(
        |v'|^2
        +\frac{\nu_\ell^2-\frac14}{r^2}|v|^2
        +\frac{\mu^2-\frac14}{(1-r)^2}|v|^2
    \right)\dd r.
\]
Since $P_{\ell,m}\mathcal C$ is a form core in
$\mathcal K_{\ell,m}$ and
$U_{\ell,m}(P_{\ell,m}\mathcal C)=C_c^\infty(0,1)$,
the operator $U_{\ell,m}A_{\ell,m}U_{\ell,m}^{-1}$
is the Friedrichs realization of
\begin{equation}\label{eq:radial-op}
    H_{\nu_\ell,\mu}
    =
    -\frac{\dd^2}{\dd r^2}
    +\frac{\nu_\ell^2-\frac14}{r^2}
    +\frac{\mu^2-\frac14}{(1-r)^2},
    \qquad 0<r<1,
\end{equation}
initially defined on $C_c^\infty(0,1)$.
This includes the critical sector $\nu_\ell=0$,
where the radial form is obtained by closing the combined quadratic expression and the logarithm-free endpoint solution is selected.
At limit-point endpoints, no additional boundary
condition is imposed.

Thus the unitary map
\[
    (\mathcal Uf)_{\ell,m}:=U_{\ell,m}P_{\ell,m}f
\]
gives
\begin{equation}\label{eq:operator-direct-sum}
    \mathcal U A_\kappa\mathcal U^{-1}
    =
    \bigoplus_{\ell,m}H_{\nu_\ell,\mu}.
\end{equation}
This is an identity of self-adjoint operators,
with direct sum domain
\[
    \left\{
        (v_{\ell,m})\in\bigoplus_{\ell,m}L^2(0,1):
        v_{\ell,m}\in D(H_{\nu_\ell,\mu}),\
        \sum_{\ell,m}
        \|H_{\nu_\ell,\mu}v_{\ell,m}\|_{L^2(0,1)}^2
        <\infty
    \right\}.
\]

\paragraph{Application to the wave equation.}

For an energy solution $u$, define
\begin{equation}\label{eq:angular-projection}
    \widetilde u_{\ell,m}(t,r)
    :=
    \int_{\Sph^{n-1}}
    u(t,r,\vartheta)\overline{Y_{\ell,m}(\vartheta)}
    \,\dd\vartheta,
\end{equation}
and set
\begin{equation}\label{eq:radial-unitary-conjugation}
    v_{\ell,m}(t,r):=r^\alpha\widetilde u_{\ell,m}(t,r).
\end{equation}
Then
\begin{equation}\label{eq:spherical-expansion}
    u(t,r,\vartheta)
    =
    \sum_{\ell,m}
    r^{-\alpha}v_{\ell,m}(t,r)Y_{\ell,m}(\vartheta),
\end{equation}
with convergence in the form norm, and
\begin{equation}\label{eq:polar-parseval}
    \|u(t)\|_{L^2(\B)}^2
    =
    \sum_{\ell,m}\|v_{\ell,m}(t)\|_{L^2(0,1)}^2,
    \qquad
    \|A_\kappa^{1/2}u(t)\|_{L^2(\B)}^2
    =
    \sum_{\ell,m}
    \|H_{\nu_\ell,\mu}^{1/2}v_{\ell,m}(t)\|_{L^2(0,1)}^2.
\end{equation}
By \eqref{eq:operator-direct-sum} and functional calculus,
each radial component is the energy solution of
\begin{equation}\label{eq:radial-wave-equation}
    \partial_t^2v_{\ell,m}+H_{\nu_\ell,\mu}v_{\ell,m}=0,
\end{equation}
with initial data
$U_{\ell,m}P_{\ell,m}u_0$ and
$U_{\ell,m}P_{\ell,m}u_1$.
\paragraph{Endpoint asymptotics and weighted Neumann traces.}
Let $\omega_{\nu_\ell,k}^2$, $k\geq1$, be the simple
eigenvalues of $H_{\nu_\ell,\mu}$, with corresponding
real, $L^2(0,1)$-normalized eigenfunctions
$\phi_{\nu_\ell,k}$.
The indicial roots at $r=0$ and $r=1$ are
$1/2\pm\nu_\ell$ and $1/2\pm\mu$, respectively.
The Friedrichs realization selects the principal
solution at each endpoint
\cite[Theorem~11.19]{EGNT2013}: the branch with
exponent $1/2+\rho$ for $\rho>0$, and the
logarithm-free branch for $\rho=0$;
see \cref{Principal solution}. Hence
\begin{equation}\label{eq:frob-ends}
    \phi_{\nu_\ell,k}(r)=O(r^{\nu_\ell+1/2})
    \quad(r\downarrow0),
    \qquad
    \phi_{\nu_\ell,k}(r)=O((1-r)^{\mu+1/2})
    \quad(r\uparrow1).
\end{equation}
In particular, when $n=2$ and $\ell=0$, the
$r^{1/2}\log r$ branch is excluded, so the first
estimate remains valid.

Since the remaining potential
$(\nu_\ell^2-\frac14)/r^2$ is smooth near $r=1$,
the principal-solution estimates give, with $y=1-r$,
\begin{equation}\label{eq:norming}
    \begin{aligned}
        \phi_{\nu_\ell,k}(r)
        &=c_{\nu_\ell,k}y^{\mu+1/2}(1+O(y^2)),\\
        \phi_{\nu_\ell,k}'(r)
        &=-\left(\mu+\frac12\right)c_{\nu_\ell,k}
          y^{\mu-1/2}+O(y^{\mu+3/2}),
    \end{aligned}
    \qquad y\downarrow0,
\end{equation}
where $c_{\nu_\ell,k}\neq0$ and the remainder
constants may depend on the mode.

Define
\begin{equation}\label{eq:full-eigenfunctions}
    \Phi_{\ell,m,k}(r,\vartheta)
    :=r^{-(n-1)/2}\phi_{\nu_\ell,k}(r)
      Y_{\ell,m}(\vartheta).
\end{equation}
By \eqref{eq:operator-direct-sum} and radial
completeness, these functions form an orthonormal
eigenbasis of $A_\kappa$ in $L^2(\B)$, with
$A_\kappa\Phi_{\ell,m,k}
=\omega_{\nu_\ell,k}^2\Phi_{\ell,m,k}$.

Put $R_{\ell,k}=r^{-(n-1)/2}\phi_{\nu_\ell,k}$.
Since $\mu+\frac12=1-\kappa$ and
$r^{-(n-1)/2}=1+O(y)$, \eqref{eq:norming} yields
\[
    R_{\ell,k}
    =c_{\nu_\ell,k}y^{1-\kappa}+O(y^{2-\kappa}),
    \qquad
    R_{\ell,k}'
    =-(1-\kappa)c_{\nu_\ell,k}y^{-\kappa}
      +O(y^{1-\kappa}).
\]
Thus $\mathcal D_\kappa\Phi_{\ell,m,k}=0$.
Define the boundary observation coefficient by
\begin{equation}\label{eq:gamma}
    \gamma_{\nu_\ell,k}
    :=-(1-2\kappa)c_{\nu_\ell,k}
    =-2\mu c_{\nu_\ell,k}.
\end{equation}
Using \eqref{eq:weighted-boundary-traces}, we obtain
\begin{equation}\label{eq:trace}
    \begin{aligned}
        \mathcal N_\kappa\Phi_{\ell,m,k}
        &=
        \lim_{r\uparrow1}
        \left[
            y^\kappa R_{\ell,k}'
            +\kappa y^{\kappa-1}R_{\ell,k}
        \right]Y_{\ell,m}\\
        &=\gamma_{\nu_\ell,k}Y_{\ell,m}.
    \end{aligned}
\end{equation}

\subsection{Spectral representation and the observation domain}
We first give the spectral representation of energy
solutions and characterize the observation domain,
deferring the proofs concerning this domain to
\cref{sec:observation-domain}. We then derive a necessary condition for
boundary observability. The main task of this paper
is to prove that this condition, together with uniform
spectral estimates, yields boundary observability
in the sharp time range.

Following the standard abstract wave equation framework
of \cite[Proposition~3.7.6]{TucsnakWeiss2009}, let
\[
    \mathcal H_\kappa
    :=
    D(A_\kappa^{1/2})\times L^2(\B)
\]
be the energy space, equipped with the norm
\[
    \|(u_0,u_1)\|_{\mathcal H_\kappa}^2
    :=
    \|A_\kappa^{1/2}u_0\|_{L^2(\B)}^2
    +
    \|u_1\|_{L^2(\B)}^2
    =
    2\E_\kappa[u](0).
\]
Since $A_\kappa\geq cI$ for some $c>0$,
$\mathcal H_\kappa$ is a Hilbert space under this norm. For initial data
\[
    (u_0,u_1)\in\mathcal H_\kappa,
    \qquad
    u(0)=u_0,\quad \partial_tu(0)=u_1,
\]
write
\begin{equation}\label{eq:initial-modal-coefficients}
    u_0
    =
    \sum_{\ell,m,k}U_{\ell,m,k}\Phi_{\ell,m,k},
    \qquad
    u_1
    =
    \sum_{\ell,m,k}V_{\ell,m,k}\Phi_{\ell,m,k}.
\end{equation}
The energy-space condition is equivalent to
\[
    \sum_{\ell,m,k}
    \left(
        \omega_{\nu_\ell,k}^2|U_{\ell,m,k}|^2
        +|V_{\ell,m,k}|^2
    \right)<\infty.
\]

The spectral theorem gives the following representation
of the wave evolution constructed in
\cite[Proposition~3.8.7]{TucsnakWeiss2009}:
\begin{equation}\label{eq:abstract-wave-solution}
    u(t)
    =
    \cos(t\sqrt{A_\kappa})u_0
    +
    \frac{\sin(t\sqrt{A_\kappa})}{\sqrt{A_\kappa}}u_1.
\end{equation}
It defines the unique energy solution
\[
    u\in
    C\bigl(\mathbb R;D(A_\kappa^{1/2})\bigr)
    \cap C^1\bigl(\mathbb R;L^2(\B)\bigr)
\]
of \eqref{eq:wave}, with the boundary condition incorporated
in the Friedrichs realization. We denote this class of
energy solutions by $\mathcal U_\kappa$. 

Consequently,
\begin{equation}\label{eq:cos-sin-full-wave}
    u(t,r,\vartheta)
    =
    \sum_{\ell,m,k}
    \left[
        U_{\ell,m,k}\cos(\omega_{\nu_\ell,k}t)
        +
        \frac{V_{\ell,m,k}}{\omega_{\nu_\ell,k}}
        \sin(\omega_{\nu_\ell,k}t)
    \right]
    \Phi_{\ell,m,k}(r,\vartheta).
\end{equation}
The series converges in the energy topology, locally
uniformly in time.

For the nonharmonic Fourier argument, it is convenient
to introduce the positive- and negative-frequency
coefficients
\begin{equation}\label{eq:apm-from-initial-data}
    a^+_{\ell,m,k}
    :=
    \frac12\left(
        U_{\ell,m,k}
        -\frac{iV_{\ell,m,k}}{\omega_{\nu_\ell,k}}
    \right),
    \qquad
    a^-_{\ell,m,k}
    :=
    \frac12\left(
        U_{\ell,m,k}
        +\frac{iV_{\ell,m,k}}{\omega_{\nu_\ell,k}}
    \right).
\end{equation}
Thus
\begin{equation}\label{eq:initial-data-from-apm}
    U_{\ell,m,k}
    =
    a^+_{\ell,m,k}+a^-_{\ell,m,k},
    \qquad
    V_{\ell,m,k}
    =
    i\omega_{\nu_\ell,k}
    \bigl(a^+_{\ell,m,k}-a^-_{\ell,m,k}\bigr).
\end{equation}
For real initial data and a real eigenbasis,
$a^-_{\ell,m,k}=\overline{a^+_{\ell,m,k}}$.
These coefficients are determined entirely by the
initial state.

Returning to the radial notation in
\eqref{eq:spherical-expansion}, we obtain
\begin{equation}\label{eq:modal-wave}
    v_{\ell,m}(t,r)
    =
    \sum_{k\geq1}
    \left(
        a^+_{\ell,m,k}e^{i\omega_{\nu_\ell,k}t}
        +
        a^-_{\ell,m,k}e^{-i\omega_{\nu_\ell,k}t}
    \right)
    \phi_{\nu_\ell,k}(r).
\end{equation}
Orthogonality gives
\begin{equation}\label{eq:modal-energy-initial-data}
    \E_\kappa[u](0)
    =
    \frac12
    \sum_{\ell,m,k}
    \left(
        |V_{\ell,m,k}|^2
        +
        \omega_{\nu_\ell,k}^2|U_{\ell,m,k}|^2
    \right).
\end{equation}
The parallelogram identity and
\eqref{eq:apm-from-initial-data} yield the equivalent
expression
\begin{equation}\label{eq:modal-energy}
    \E_\kappa[u](0)
    =
    \sum_{\ell,m,k}\omega_{\nu_\ell,k}^2
    \left(
        |a^+_{\ell,m,k}|^2+|a^-_{\ell,m,k}|^2
    \right),
\end{equation}
where the factor $1/2$ is already included in the
definition \eqref{eq:energy}.

We next specify the class on which the boundary
observation is square integrable. Define
\begin{align}
    Q(u_0,u_1)
    &:=
    \sum_{\ell,m,k}|\gamma_{\nu_\ell,k}|^2
    \left(
        |a^+_{\ell,m,k}|^2+|a^-_{\ell,m,k}|^2
    \right),
    \label{eq:observation-coefficient-norm}\\
    \mathcal D_{\rm obs}
    &:=
    \left\{
        (u_0,u_1)\in\mathcal H_\kappa:
        Q(u_0,u_1)<\infty
    \right\}.
    \label{eq:observation-domain}
\end{align}
Equivalently, in terms of the initial modal coefficients,
\[
    Q(u_0,u_1)
    =
    \frac12\sum_{\ell,m,k}|\gamma_{\nu_\ell,k}|^2
    \left(
        |U_{\ell,m,k}|^2
        +
        \frac{|V_{\ell,m,k}|^2}{\omega_{\nu_\ell,k}^2}
    \right).
\]
Thus $Q$ is a weighted coefficient norm squared of
the initial state. The corresponding solution class is
\[
    \mathcal U_{\rm obs}
    :=
    \left\{
        u\in\mathcal U_\kappa:
        \bigl(u(0),\partial_tu(0)\bigr)
        \in\mathcal D_{\rm obs}
    \right\}.
\]
For a solution $u$, we also write
\[
    Q[u]:=Q\bigl(u(0),\partial_tu(0)\bigr).
\]

For $u\in\mathcal U_{\rm obs}$, define
\[
    g_{\ell,m}(t)
    :=
    \sum_{k\geq1}\gamma_{\nu_\ell,k}
    \left(
        a^+_{\ell,m,k}e^{i\omega_{\nu_\ell,k}t}
        +
        a^-_{\ell,m,k}e^{-i\omega_{\nu_\ell,k}t}
    \right),
\]
and
\[
    g(t,\vartheta)
    :=
    \sum_{\ell,m}g_{\ell,m}(t)Y_{\ell,m}(\vartheta).
\]

For a finite spectral-sum solution, the weighted Neumann
trace is defined by the classical boundary limit \eqref{eq:weighted-boundary-traces}. The
endpoint expansion of each eigenfunction shows that this
limit equals the corresponding finite sum defining \(g\).
More generally, whenever a sufficiently regular solution
has a classical weighted Neumann trace, that trace agrees
with \(g\).

As proved in \cref{sec:observation-domain}, the trace map
on finite spectral sums extends continuously from the
energy space into distributions. We denote this extension
by \(\N_\kappa u\) for an arbitrary energy solution \(u\).
Equivalently, it is the distributional limit of the extended
weighted traces on interior spheres; this interpretation
does not require restricting the radial derivative of a
general energy solution to a sphere.

For \(u\in\mathcal U_{\rm obs}\), the extended trace is
represented by the distribution \(g\) constructed above.
Thus
\[
    \N_\kappa u=g
    \quad\text{in }
    \mathcal D'\bigl((-T,T)\times\Sph^{n-1}\bigr),
    \qquad T>0.
\]
In particular, no pointwise boundary limit is taken
term by term in the infinite spectral expansion. 

Finally, \cref{sec:observation-domain} proves that,
for every $T>1$,
\[
    \mathcal U_{\rm obs}
    =
    \left\{
        u\in\mathcal U_\kappa:
        \N_\kappa u
        \in L^2\bigl((-T,T)\times\Sph^{n-1}\bigr)
    \right\}.
\]
Hence the spectral observation domain describes exactly
the energy solutions with square-integrable weak Neumann
trace on the observation interval. No bounded observation
map on the entire energy space is assumed.

For $\N_\kappa u
        \in L^2\bigl((-T,T)\times\Sph^{n-1}\bigr)$, Parseval's identity on the sphere and Tonelli's theorem
now give
\begin{align}
    \int_{-T}^{T}\int_{\Sph^{n-1}}
        |\N_\kappa u|^2\,\dd\vartheta\,\dd t
    &=
    \sum_{\ell,m}\int_{-T}^{T}|g_{\ell,m}(t)|^2\,\dd t
    \notag\\
    &=
    \sum_{\ell,m}\int_{-T}^{T}
    \left|
        \sum_{k\geq1}\gamma_{\nu_\ell,k}
        \left(
            a^+_{\ell,m,k}e^{i\omega_{\nu_\ell,k}t}
            +
            a^-_{\ell,m,k}e^{-i\omega_{\nu_\ell,k}t}
        \right)
    \right|^2\,\dd t.
    \label{eq:boundary-parseval}
\end{align}
Thus distinct spherical harmonic modes are orthogonal
in the full-boundary observation. Within each angular
mode, the remaining temporal interactions are handled
by the nonharmonic Fourier estimates.

We now derive a necessary condition for boundary
observability from \eqref{eq:boundary-parseval} and
\eqref{eq:modal-energy}. Testing the observation inequality
on a single eigenmode shows that the boundary coefficients
must satisfy a uniform lower bound
\[
    |\gamma_{\nu_\ell,k}|^2
    \geq c\,\omega_{\nu_\ell,k}^2,
\]
with \(c>0\) independent of both indices.
Conversely, a central part of our argument is to combine
this lower bound with uniform frequency separation and
local counting estimates, and then apply an
Ingham--Beurling type theorem to obtain boundary
observability for every \(T>1\).
Together with the failure of observability for \(T<1\),
this identifies the sharp observation-time threshold.
The following corollary establishes the necessary lower bound
whenever a large-time observability inequality is available.

\begin{corollary}[Uniform lower norming estimate]
\label{lem:norming-lower}
Fix $\kappa<0$. Suppose that there exist constants
$T_0,C_0>0$ such that observation inequality \eqref{eq:obs} holds
for every finite spectral-sum solution $w$.
Then
\begin{equation}\label{eq:gamma-lower}
    |\gamma_{\nu_\ell,k}|^2
    \geq c_0\,\omega_{\nu_\ell,k}^2,
    \qquad \ell\geq0,\quad k\geq1,
\end{equation}
where $c_0=(4C_0T_0)^{-1}$ is independent of $\ell$ and $k$.
\end{corollary}

\begin{proof}
It suffices to test \eqref{eq:obs} on a single eigenmode. Fix a real
\(L^2(\mathbb B)\)-normalized eigenfunction
\(\Phi_{\ell,m,k}\) from \eqref{eq:full-eigenfunctions} and consider the
corresponding solution
\[
    w(t)
    =
    \cos(\omega_{\nu_\ell,k}t)\Phi_{\ell,m,k}.
\]
Applying \eqref{eq:obs} to this solution gives the desired estimate.
\end{proof}

Since the finite spectral-sum solutions used above have the
required interior regularity and finite energy. Their Friedrichs endpoint expansions, together with
the corresponding differentiated expansions, give
\[
    \lim_{y\downarrow0}y^{-\kappa}u=0,
    \qquad
    (1-2\kappa)
    \lim_{y\downarrow0}y^{-1+\kappa}u
    =-\N_\kappa u,
    \qquad
    \lim_{y\downarrow0}y^\kappa\partial_tu=0.
\]
These limits hold strongly in
$L^2((-T_0,T_0)\times\Sph^{n-1})$.
Together with the existence of the Neumann trace,
they verify the boundary admissibility conditions
required by \cite[Theorem~5.1]{ESV2021} and \cite[Theorem~1.14]{JenaShao2026}. 

Thus, for \(-\frac12<\kappa<0\) and \(n\geq3\),
\eqref{eq:obs} follows from \cite[Theorem~5.1]{ESV2021}.
For \(n=2\), the corresponding result is provided by
\cite[Theorem~1.14]{JenaShao2026}, which applies in all
dimensions under a convexity assumption that is satisfied by the
unit ball. Hence, in this parameter range, the lower bound
\eqref{eq:gamma-lower} is already available.

To extend the argument to the full range \(\kappa<0\), including the
limit-point regime \(\kappa\leq-\frac12\), we shall also establish
\eqref{eq:gamma-lower} directly from the radial equation in
\cref{Section6}. This direct approach provides a uniform lower
norming estimate throughout the entire parameter range and renders
the resulting boundary observability argument independent of the
existing large-time observability results.
\section{Reduction to uniform spectral estimates via Ingham--Beurling inequalities}

We recall two classical inequalities for nonharmonic Fourier
series. Let $\Lambda=\{\lambda_j:j\in\mathbb Z\}$ be a strictly
increasing sequence of real frequencies. For
$a=(a_j)\in\ell^2(\mathbb Z)$, write
\[
    F_a(t):=\sum_{j\in\mathbb Z}a_j e^{i\lambda_jt},
\]
where the series is understood in $L^2$ on bounded intervals.
The convergence asserted below follows from the corresponding
upper estimates.

Ingham's theorem gives a two-sided estimate in terms of the
minimum separation of the frequencies.

\begin{theorem}[Ingham {\cite{Ingham1936}}]
\label{thm:classical-ingham}
Suppose that
\[
    \lambda_{j+1}-\lambda_j\geq\delta>0,
    \qquad j\in\mathbb Z.
\]
For every bounded interval $I$ satisfying
\[
    |I|>\frac{2\pi}{\delta},
\]
there exist constants $A,B>0$, depending only on $|I|$ and
$\delta$, such that, for every $a\in\ell^2(\mathbb Z)$, the
series defining $F_a$ converges in $L^2(I)$ and satisfies
\[
    A\sum_{j\in\mathbb Z}|a_j|^2
    \leq
    \int_I |F_a(t)|^2\,\dd t
    \leq
    B\sum_{j\in\mathbb Z}|a_j|^2.
\]
\end{theorem}

Beurling improved the sufficient interval length by replacing
the reciprocal of the minimum separation with the upper
density of the frequency set. Define
\[
    n_\Lambda^+(R)
    :=
    \sup_{s\in\mathbb R}
    \#\bigl(\Lambda\cap[s,s+R]\bigr),
    \qquad R>0,
\]
and
\[
    D^+(\Lambda)
    :=
    \limsup_{R\to\infty}\frac{n_\Lambda^+(R)}{R}.
\]

\begin{theorem}[Beurling {\cite{Beurling1989}}]
\label{thm:classical-beurling}
Suppose that $\Lambda$ is uniformly separated:
\[
    \inf_{j\neq k}|\lambda_j-\lambda_k|>0.
\]
For every bounded interval $I$ satisfying
\[
    |I|>2\pi D^+(\Lambda),
\]
there exist constants $A,B>0$ such that, for every
$a\in\ell^2(\mathbb Z)$, the series defining $F_a$ converges
in $L^2(I)$ and satisfies
\[
    A\sum_{j\in\mathbb Z}|a_j|^2
    \leq
    \int_I |F_a(t)|^2\,\dd t
    \leq
    B\sum_{j\in\mathbb Z}|a_j|^2.
\]
The constants are independent of $a$, but may depend on
$I$ and $\Lambda$.
\end{theorem}

If the frequencies have separation at least $\delta$, then
\[
    n_\Lambda^+(R)\leq R/\delta+1,
    \qquad D^+(\Lambda)\leq1/\delta.
\]
Thus Beurling's density condition may improve Ingham's
threshold $|I|>2\pi/\delta$; neither theorem settles
its critical interval length.

Uniform separation also yields an upper estimate on
every bounded interval $I$, depending only on $|I|$
and $\delta$. Hence the exponential series converges
in $L^2(I)$ for every square-summable coefficient
sequence, even when a lower estimate is unavailable.

For families with a common separation bound, Ingham's
lower constant is uniform when $|I|>2\pi/\delta$.
Beurling's density formulation alone does not provide
a lower constant controlled solely by $|I|$, $\delta$,
and the upper density. To obtain uniform estimates
at the improved threshold, we therefore impose
a quantitative counting bound.

\begin{lemma}[Uniform-family Beurling--Ingham lemma]
\label{lem:uniform-family-ingham}
Let $\mathfrak L$ be a family of strictly increasing
two-sided sequences
\[
    \Lambda=\{\lambda_j:j\in\mathbb Z\}\subset\mathbb R.
\]
Suppose that there are constants $\delta,D>0$ and $C\geq1$,
common to the family, such that
\begin{align}
    \lambda_{j+1}-\lambda_j
    &\geq\delta,
    \label{eq:family-separation}\tag{UF1}\\
    \#\bigl(\Lambda\cap[s,s+R]\bigr)
    &\leq DR+C,
    \label{eq:family-count}\tag{UF2}
\end{align}
for every $\Lambda\in\mathfrak L$, $j\in\mathbb Z$,
$s\in\mathbb R$, and $R>0$.

If $I$ is a bounded interval with $|I|>2\pi D$, then there
exists a constant
\[
    c_I=c(|I|,\delta,D,C)>0
\]
such that, for every $\Lambda\in\mathfrak L$ and every
$d=(d_\lambda)_{\lambda\in\Lambda}\in\ell^2(\Lambda)$,
the series
\[
    \sum_{\lambda\in\Lambda}d_\lambda e^{i\lambda t}
\]
converges in $L^2(I)$ and satisfies
\begin{equation}\label{eq:uniform-family-ingham}
    \int_I
    \left|
        \sum_{\lambda\in\Lambda}
        d_\lambda e^{i\lambda t}
    \right|^2\,\dd t
    \geq
    c_I\sum_{\lambda\in\Lambda}|d_\lambda|^2.
\end{equation}
The lower constant is independent of
$\Lambda\in\mathfrak L$.
\end{lemma}

\begin{proof}
The common separation bound $\delta$ ensures that
the exponential series converges in $L^2(I)$ for every
$d\in\ell^2(\Lambda)$.
The uniform lower estimate can be obtained from
\cite[Theorem~1.5]{BKL2002}.
Here we use the explicit quantitative estimate in
\cite[Theorem~1.1]{Komornik2008}, which provides
a more direct proof.

Specifically, write $|I|=2R$, by translation invariance, it suffices to
consider $I=(-R,R)$. Since $R>\pi D$ and
\[
    \frac{\pi DM}{M+1-C}\longrightarrow\pi D
    \qquad\text{as }M\longrightarrow\infty,
\]
we may choose an integer $M>C-1$, depending only on
$R,D,C$, such that
\begin{equation}\label{eq:choose-M}
    \frac{\pi DM}{M+1-C}<R.
\end{equation}
Set
\[
    \Gamma_M:=\frac{M+1-C}{D}>0.
\]

Fix $\Lambda=\{\lambda_j:j\in\mathbb Z\}\in\mathfrak L$.
Applying \eqref{eq:family-count} to
$[\lambda_j,\lambda_{j+M}]$ gives
\[
    M+1
    \leq D(\lambda_{j+M}-\lambda_j)+C.
\]
Consequently,
\begin{equation}\label{eq:M-step-gap}
    \lambda_{j+M}-\lambda_j\geq\Gamma_M,
    \qquad j\in\mathbb Z.
\end{equation}
Partition $\Lambda$ into $M$ subsequences:
\[
    \Lambda=\bigcup_{q=0}^{M-1}\Lambda_q,
    \qquad
    \Lambda_q
    :=
    \{\lambda_{q+kM}:k\in\mathbb Z\}.
\]
Each $\Lambda_q$ has separation at least $\Gamma_M$,
while the full sequence has separation at least $\delta$.
Both bounds are independent of $\Lambda$.

To make the quantitative application explicit, the one-dimensional
Dirichlet eigenvalue of $-\partial_t^2$ on $(-1,1)$ is
$\pi^2/4$. In the notation of
\cite[Theorem~1.1]{Komornik2008}, the common lower bound
$\Gamma_M$ for the gaps in each of the $M$ subsequences
gives the sufficient half-length
\[
    R_*:=\sum_{q=0}^{M-1}\frac{\pi}{\Gamma_M}
         =\frac{M\pi}{\Gamma_M}
         =\frac{\pi DM}{M+1-C}<R.
\]
Choose the fixed number
\[
    R_1:=R_*+\frac12\min\{R-R_*,R_*\}.
\]
Then $R_*<R_1<2R_*$ and $R_1<R$. The proof of the cited
quantitative theorem applies with the displayed lower bounds for
the gaps (its auxiliary Fourier-support kernels are constructed using
these bounds); the resulting constants depend only on
$M,\delta,\Gamma_M$ and $R_1-R_*$.
Applying the theorem on $(-R_1,R_1)$ and enlarging the
integration interval to $(-R,R)$ therefore gives a lower
constant independent of $\Lambda$. The uniform upper
estimate from the common gap permits passage from finite
coefficient families to all of $\ell^2(\Lambda)$.
This proves the lemma.
\end{proof}

For $\nu\in\mathcal V:=\{0\}\cup[1/2,\infty)$, define the signed radial frequency set
\begin{equation}\label{eq:freq-set}
    \Lambda_\nu
    :=
    \{-\omega_{\nu,k}:k\geq1\}
    \cup
    \{\omega_{\nu,k}:k\geq1\}.
\end{equation}

\begin{proposition}[Uniform radial Ingham estimate]
\label{prop:uniform-ingham}
Fix $\mu>1/2$. For every $T>1$, there exists a constant
$c_{T,\mu}>0$ such that, for every
\[
    \nu\in
    \left\{
        \ell+\frac{n-2}{2}:\ell\in\mathbb N_0
    \right\}
\]
and every pair
\[
    b^+=(b_k^+)_{k\geq1}\in\ell^2(\mathbb N),
    \qquad
    b^-=(b_k^-)_{k\geq1}\in\ell^2(\mathbb N),
\]
the series
\[
    \sum_{k\geq1}
    \left(
        b_k^+e^{i\omega_{\nu,k}t}
        +
        b_k^-e^{-i\omega_{\nu,k}t}
    \right)
\]
converges in $L^2(-T,T)$ and satisfies
\begin{equation}\label{eq:ingham}
    \int_{-T}^{T}
    \left|
        \sum_{k\geq1}
        \left(
            b_k^+e^{i\omega_{\nu,k}t}
            +
            b_k^-e^{-i\omega_{\nu,k}t}
        \right)
    \right|^2\,\dd t
    \geq
    c_{T,\mu}
    \sum_{k\geq1}
    \left(|b_k^+|^2+|b_k^-|^2\right).
\end{equation}
The constant is uniform in the angular parameter $\nu$.
\end{proposition}

By the uniform-family
Beurling--Ingham lemma \cref{lem:uniform-family-ingham}, a convenient sufficient form of the spectral statement behind
\cref{prop:uniform-ingham} is the following uniform spacing and density lemma.

\begin{proposition}\label{prop:spectral-core}
For fixed $\mu>1/2$, the family $\{\Lambda_\nu\}_{\nu\geq1/2}$
satisfies
\begin{align}
 &\inf_{\nu\geq1/2}\inf_{\lambda\neq\lambda'\in\Lambda_\nu}
       |\lambda-\lambda'|>0, \label{eq:uniform-separation}\\
 &\lim_{R\to\infty}\ \sup_{\nu\geq1/2}\sup_{s\in\mathbb R}
   \frac{\#(\Lambda_\nu\cap[s,s+R])}{R}
   \leq\frac1\pi. \label{eq:uniform-density}
\end{align}
The isolated two-dimensional sector $\nu=0$ satisfies the corresponding
claims separately.
\end{proposition}

More quantitatively, the proof below gives constants $c_\mu,C_\mu>0$ such
that
\begin{align}
 \omega_{\nu,k+1}-\omega_{\nu,k}&\geq c_\mu,
       \label{eq:quant-gap}\\
 \#\bigl(\Lambda_\nu\cap[s,s+R]\bigr)
     &\leq \frac{R}{\pi}+C_\mu
       \label{eq:quant-count}
\end{align}
for every $\nu\in\mathcal V$, $k\geq1$, $s\in\mathbb R$, and $R>0$.
Once these two estimates are established, the uniform-family
Beurling--Ingham lemma proved in \cref{lem:uniform-family-ingham}, with
$D=1/\pi$ and $I=(-T,T)$, yields \eqref{eq:ingham}.  Indeed,
\[
 2\pi D=2<2T=|I|.
\]
The word ``uniform'' is essential: the fixed-$\nu$
Weyl law $\omega_{\nu,k}=\pi k+O_\nu(1)$, obtained
from the perturbed-Bessel asymptotics in
\cite[Theorem~2.5]{KST2010} together with Dirichlet
bracketing, does not by itself control windows whose
centres and angular parameters tend to infinity
simultaneously.

With \cref{prop:uniform-ingham} in hand, we now prove
the sufficient time condition in \cref{thm:target}
and outline the necessity argument.

\begin{proof}[Proof of \cref{thm:target}]
\noindent\emph{\textbf{Sufficiency: $T>1$.}}
Let $u$ be a Friedrichs energy solution whose weak
Neumann trace belongs to $L^2(\Gamma_T)$.
By \cref{prop:observation-domain},
\[
    Q\bigl(u(0),\partial_tu(0)\bigr)<\infty.
\]
Thus, for each $(\ell,m)$, the sequences
\[
    b_k^\pm:=\gamma_{\nu_\ell,k}a^\pm_{\ell,m,k}
\]
are square summable.
Applying \cref{prop:uniform-ingham} in each angular
channel of \eqref{eq:boundary-parseval} and summing
over $(\ell,m)$ gives
\begin{align*}
    \int_{-T}^{T}\int_{\Sph^{n-1}}
        |\N_\kappa u|^2\,\dd\vartheta\,\dd t
    &\geq
    c_{T,\mu}\sum_{\ell,m,k}
        |\gamma_{\nu_\ell,k}|^2
        \left(
            |a^+_{\ell,m,k}|^2
            +|a^-_{\ell,m,k}|^2
        \right)\\
    &\geq
    4\mu^2c_{T,\mu}c_\mu
    \sum_{\ell,m,k}
        \omega_{\nu_\ell,k}^2
        \left(
            |a^+_{\ell,m,k}|^2
            +|a^-_{\ell,m,k}|^2
        \right)\\
    &=4\mu^2c_{T,\mu}c_\mu\,\E_\kappa[u](0).
\end{align*}
Here the second inequality follows from
\cref{lem:norming-lower1} and \eqref{eq:gamma},
and the last equality is \eqref{eq:modal-energy}.
This proves \eqref{eq:obs} for $T>1$.

\medskip
\noindent\emph{\textbf{Necessity: failure for $0<T<1$.}}
The argument consists in constructing a nonzero solution
$u\in\mathcal U_{\rm obs}$ with smooth initial data
supported sufficiently close to the centre of the ball.
Finite propagation at unit speed then ensures that
$u$ vanishes near the boundary throughout $(-T,T)$.
Consequently, its weighted Neumann trace vanishes on
$\Gamma_T$, while its initial energy is positive,
contradicting \eqref{eq:obs}.
Approximation in the observation graph norm also
yields failure uniformly over finite spectral sums.

Verifying that the constructed solution belongs to
$\mathcal U_{\rm obs}$ and justifying this approximation
use the characterization of the observation domain
developed in the appendix.
We therefore defer the complete necessity proof to
\cref{prop:lower-time-obstruction}.
\end{proof}

\begin{remark}For the original parameter range, \cref{lem:norming-lower}
provides another route from the earlier large-time theorem to the
norming estimate. The argument above uses the direct bound proved below,
so the observability result for all $\kappa<0$ is independent of
the previously known large-time inequality. The time threshold itself
comes from the uniform frequency density.
\end{remark}
\begin{remark}[Failure for $0<\kappa<\frac12$]
\label{rem:positive-kappa-failure}
The conclusion of \cref{thm:target} does not extend to
$0<\kappa<\frac12$ with the same boundary observation.
Indeed, no inequality of the form \eqref{eq:obs}
holds uniformly even for single eigenmode solutions,
regardless of the fixed observation time $T>0$.

As noted by Enciso, Shao, and Vergara in
\cite[Remark~1.2(iii)]{ESV2021}, the one-dimensional
spectral results of Gueye \cite{Gueye2014} suggest
that the analogue of their Carleman estimate fails
for positive $\kappa$.
More precisely, Gueye's boundary coefficient estimates
already imply, by testing on single eigenmodes, the
failure of observability in the natural energy norm
for the corresponding one-dimensional model;
see \cite[Theorem~3.1 and Section~5.1]{Gueye2014}.
The argument below establishes this failure directly
for the present unit-ball problem in every dimension
$n\geq2$.

To see this, we consider the Friedrichs realization in the
positive-\(\kappa\) regime. The Friedrichs extension selects the
principal solution at the singular endpoint, and hence the
principal-solution analysis from \cref{Principal solution} applies to
the first radial eigenfunction. We shall use the resulting endpoint
asymptotics to prove the failure of the boundary observability estimate for the corresponding eigenmodes. Set
\[
    p=1-\kappa\in\bigl(\tfrac12,1\bigr),
    \qquad a=-\kappa(1-\kappa)\in\bigl(-\tfrac14,0\bigr).
\]
Let $\nu=\nu_\ell\to\infty$, set
$b_\nu=\nu^2-\frac14$ and
$\varepsilon_\nu=b_\nu^{-1/3}$, and denote the first
radial eigenvalue and a normalized eigenfunction by
$\lambda_\nu=\omega_{\nu,1}^2$ and $\phi_\nu$. 
By the Rayleigh--Ritz characterization of the first radial
eigenvalue, the boundary Hardy inequality yields
$\lambda_\nu\geq b_\nu$: indeed,
$\int(|v'|^2+a(1-r)^{-2}|v|^2)\,dr\geq0$
and $b_\nu/r^2\geq b_\nu$.
A normalized trial function supported where
$1-r\asymp\varepsilon_\nu$ gives
\[
    b_\nu\leq\lambda_\nu
    \leq b_\nu+C\varepsilon_\nu^{-2}.
\]
Consequently
$e_\nu:=\varepsilon_\nu^2(\lambda_\nu-b_\nu)$
remains bounded.

Rescale the boundary layer by
$F_\nu(s)=\varepsilon_\nu^{1/2}
\phi_\nu(1-\varepsilon_\nu s)$. Then
$\|F_\nu\|_{L^2(0,\varepsilon_\nu^{-1})}=1$ and
\[
    -F_\nu''+\frac{a}{s^2}F_\nu
    +\frac{(1-\varepsilon_\nu s)^{-2}-1}
           {\varepsilon_\nu}F_\nu
    =e_\nu F_\nu.
\]
On a fixed sufficiently small interval $0<s<s_0$,
the regular potential and $e_\nu$ are uniformly bounded with
respect to $\nu$.
The normalized principal-solution estimate
\eqref{eq:general-principal-asymptotics}, applied with
$\rho=\frac12-\kappa>0$, therefore gives
\[
    F_\nu(s)=C_\nu s^p(1+O(s^2)),
\]
with a uniform remainder. Taking $s_0$ smaller if
necessary, normalization of $F_\nu$ implies
$|C_\nu|\leq C$, since the corresponding principal
solution has $L^2(0,s_0)$ norm bounded below.
The unscaled endpoint coefficient and its Neumann
coupling consequently satisfy
\[
    c_{\nu,1}=\varepsilon_\nu^{-p-1/2}C_\nu,
    \qquad
    |\gamma_{\nu,1}|^2
    =(1-2\kappa)^2|c_{\nu,1}|^2.
\]
As $b_\nu=\varepsilon_\nu^{-3}$ and $0<\kappa<\frac12$, it follows that
\[
    \frac{|\gamma_{\nu,1}|^2}{\omega_{\nu,1}^2}
    \leq C_\kappa\varepsilon_\nu^{\,3-(2p+1)}
    =C_\kappa\varepsilon_\nu^{2\kappa}
    \longrightarrow0.
\]
For the single-mode wave
$w_\nu(t)=\cos(\omega_{\nu,1}t)\Phi_{\ell,m,1}$,
the energy is $\omega_{\nu,1}^2/2$, whereas its
observation on $(-T,T)$ is at most
$2T|\gamma_{\nu,1}|^2$. Thus no constant independent
of $\ell$ can make \eqref{eq:obs} hold in this range.
\end{remark}
\section{Proof of the uniform spectral lemma}

In this section, we prove \cref{prop:spectral-core}
in two parts: first the uniform density estimate,
then the uniform frequency gap.
We begin with the preliminaries needed for these
arguments, recalling Sturm oscillation theory and
the principal solutions of perturbed Bessel equations,
together with their endpoint asymptotics.
The endpoint estimates developed below also yield
the norming bound in \cref{lem:norming-lower1},
independently of existing observability results.

\subsection{Preliminaries}\label{sec:sturm-frobenius-friedrichs}
This subsection recalls the background used in the endpoint and
spectral gap arguments. 

\subsubsection{Sturm oscillation and eigenvalue counting}
For the reader's convenience, we recall the Sturm oscillation
and eigenvalue counting results for regular Sturm--Liouville
operators from \cite{TeschlODE2012}. We restrict attention to
the Dirichlet setting used in our counting argument. Consider
\[
    H_q=-\frac{\dd^2}{\dd x^2}+q(x)
    \quad\text{on }(\alpha,\beta),
    \qquad
    u(\alpha)=u(\beta)=0,
\]
where $q$ is real-valued and continuous on $[\alpha,\beta]$.
Its eigenvalues are simple and can be ordered as
\[
    \lambda_1<\lambda_2<\cdots,
    \qquad \lambda_k\longrightarrow+\infty.
\]
The eigenfunction corresponding to $\lambda_k$ has exactly
$k-1$ zeros in $(\alpha,\beta)$; see
\cite[Theorems~5.11 and~5.17]{TeschlODE2012}.

For each $E\in\mathbb R$, let $u_E$ be the unique solution of
\[
    -u_E''+qu_E=Eu_E,
    \qquad
    u_E(\alpha)=0,\quad u_E'(\alpha)=1.
\]
This solution satisfies the left Dirichlet condition, while
the right Dirichlet condition holds if and only if $E$ is
an eigenvalue of $H_q$. By
\cite[Theorem~5.18, equation~(5.93)]{TeschlODE2012},
\[
    \#\{k:\lambda_k<E\}
    =
    \#\{x\in(\alpha,\beta):u_E(x)=0\}.
\]
Since the eigenvalues are simple and
$u_E(\beta)=0$ if and only if $E$ is an eigenvalue, we also have
\[
    \#\{k:\lambda_k\leq E\}
    =
    \#\{x\in(\alpha,\beta]:u_E(x)=0\}.
\]
Thus including the right endpoint in the zero count corresponds
exactly to including the energy $E$ in the eigenvalue count.

The Pr\"ufer angle is a fundamental tool in Sturm oscillation theory. We specify below the phase convention used throughout
this paper; see also \cite[Section~5.5]{TeschlODE2012}.
A nonzero solution cannot satisfy $u(x)=u'(x)=0$ at an
interior point. Hence one can write
\[
    u'(x)=R(x)\cos\theta(x),
    \qquad
    u(x)=R(x)\sin\theta(x),
    \qquad R(x)>0,
\]
with a continuously chosen, unwrapped angle $\theta$.
Zeros of $u$ correspond to $\theta\in\pi\mathbb Z$.
Differentiating the argument of $(u',u)$ gives
\[
    \theta'(x)
    =
    \frac{u'(x)^2-u(x)u''(x)}{u'(x)^2+u(x)^2}
    =
    \cos^2\theta+(E-q)\sin^2\theta.
\]
In particular, $\theta'=1$ at a zero of $u$, so each zero
is crossed in the positive angular direction. Consequently, counting zeros reduces to measuring the increment
of the Pr\"ufer angle. More precisely, by
\cite[Lemma~5.14, equation~(5.88)]{TeschlODE2012}, on a regular
interval $(a,b)$,
\[
    \#\{x\in(a,b):u(x)=0\}
    =
    \left\lceil\frac{\theta(b)}{\pi}\right\rceil
    -
    \left\lfloor\frac{\theta(a)}{\pi}\right\rfloor
    -1.
\]
In particular,
\[
    \left|
    \#\{x\in(a,b):u(x)=0\}
    -
    \frac{\theta(b)-\theta(a)}{\pi}
    \right|
    \leq 1.
\]
Thus an estimate for the phase increment yields a zero-counting
estimate with an absolute error of at most one.

For singular endpoints, the boundary conditions must first
be specified through a self-adjoint realization.
In the counting argument of this paper, singular endpoints
are handled by approximation with regular Dirichlet problems
and convergence of the corresponding Friedrichs forms.

\subsubsection{Solutions of the perturbed inverse-square equation}\label{Principal solution}

Throughout this discussion, consider
\[
    \tau_\rho
    :=
    -\frac{\dd^2}{\dd x^2}
    +\frac{\rho^2-\frac14}{x^2}+W(x),
    \qquad 0<x<d,
\]
where $\rho\geq0$ and
$W\in C^1([0,d];\mathbb R)$ is independent of $E$.
By the sharp Hardy inequality, $\rho\geq0$ covers
precisely the inverse-square coefficients for which
the operator initially defined on $C_c^\infty(0,d)$
is bounded below.

The endpoint $0$ is limit-circle for $0\leq\rho<1$
and limit-point for $\rho\geq1$; it is regular when
$\rho=\frac12$.
We use this expression as a local model for the
Friedrichs radial operators $H_{\nu_\ell,\mu}$, with
\[
    (x,\rho)=(r,\nu_\ell)
    \quad\text{near }r=0,
    \qquad
    (x,\rho)=(1-r,\mu)
    \quad\text{near }r=1.
\]
At a limit-point endpoint, no additional boundary
condition is imposed.

The Friedrichs condition selects the principal branch,
with the faster-decaying behavior $x^{1/2+\rho}$ when
$\rho>0$ and the logarithm-free behavior $\sqrt{x}$
when $\rho=0$.
For $0\leq\rho<1$, this is equivalent to the endpoint
condition in \cite[equation~(2.29)]{KST2010} and
\cite[equation~(1.3)]{KostenkoTeschl2011},
with $l=\rho-\frac12$:
\[
    \lim_{x\downarrow0}
    x^{\rho-\frac12}
    \left[
        \left(\rho+\frac12\right)u(x)-xu'(x)
    \right]=0.
\]
This reduces to $u(0)=0$ when $\rho=\frac12$
and excludes the logarithmic branch when $\rho=0$.

\paragraph{The principal solution.}

By \cite[Lemma~3.2]{KostenkoTeschl2011}, with
$l=\rho-\frac12$, $q=W$, and $z=E$, the equation
$\tau_\rho u=Eu$ has a unique principal solution
$\phi_\rho$, normalized by
\[
    \lim_{x\downarrow0}x^{-p}\phi_\rho(x;E)=1,
    \qquad p:=\frac12+\rho.
\]
It is defined throughout $(0,d)$ and is entire in $E$.
Boundedness of $W$ ensures the hypotheses of the cited
lemma, including at $\rho=0$.

The bounded-potential specialization of the endpoint
construction gives
\begin{equation}\label{eq:general-principal-asymptotics}
    \phi_\rho(x;E)=x^p(1+O(x^2)),
    \qquad
    \partial_x\phi_\rho(x;E)
    =p x^{p-1}+O(x^{p+1}).
\end{equation}
These estimates are uniform for $\rho\geq0$, compact
complex energy sets, and potentials with a common
$L^\infty$ bound.
Indeed, the normalized Volterra equation for
$a_\rho=x^{-p}\phi_\rho$ gives
\[
    |a_\rho-1|+x|\partial_xa_\rho|
    \leq \frac{Cx^2}{1+\rho}
\]
on a common endpoint neighborhood.
Cauchy's estimate in the energy parameter consequently
yields, with the same uniformity,
\[
    \partial_E\phi_\rho=O(x^{p+2}),
    \qquad
    \partial_E\partial_x\phi_\rho=O(x^{p+1}).
\]

\paragraph{A companion solution and the Friedrichs condition.}

For $0<\rho<1$, put $p_\pm=\frac12\pm\rho$.
The companion solution from
\cite[Lemma~3.2]{KostenkoTeschl2011} can be uniquely
normalized by
\[
    \psi_\rho(x;E)-x^{p_-}=o(x^{p_+})
    \qquad(x\downarrow0).
\]
The reduction-of-order formula in its proof,
specialized to bounded $W$, and Cauchy's estimate give
\[
    \partial_x^j(\psi_\rho-x^{p_-})
    =O(x^{p_-+2-j}),
    \qquad
    \partial_E\partial_x^j\psi_\rho
    =O(x^{p_-+2-j}),
    \qquad j=0,1.
\]
Here $\psi_\rho$ is entire in $E$, and the estimates
are locally uniform in $E$ for fixed $\rho$.

The two solutions satisfy
\[
    \mathcal W(\phi_\rho,\psi_\rho)=-2\rho,
    \qquad
    \mathcal W(f,g):=fg'-f'g.
\]
Hence every solution has a unique representation
throughout $(0,d)$,
\[
    u(x;E)=A(E)\phi_\rho(x;E)+B(E)\psi_\rho(x;E),
\]
with coefficients independent of $x$.
Their endpoint asymptotics give
\[
    \lim_{x\downarrow0}
    x^{\rho-\frac12}
    \left[
        \left(\rho+\frac12\right)u(x;E)
        -x\partial_xu(x;E)
    \right]
    =2\rho B(E).
\]
Thus the Friedrichs condition recalled above is
equivalent to
\[
    B(E)=0.
\]
In particular, the solutions satisfying this condition
are precisely the scalar multiples of $\phi_\rho$.

\subsection{Proof of the uniform density estimate}

Following the action-integral approach to eigenvalue
counting in \cite[Section~6.4]{Zworski2012} and
\cite[Corollary~4.2]{Yafaev2011}, we use the action
integral to estimate the eigenfrequency counting
function and establish \eqref{eq:quant-count}.
To this end, set
\begin{equation}\label{eq:ab-def}
 a:=\mu^2-\frac14>0,
 \qquad b:=\nu^2-\frac14\geq0,
 \qquad q_{a,b}(r):=\frac b{r^2}+\frac a{(1-r)^2}.
\end{equation}
Let
\begin{equation}\label{eq:phase-volume}
 \Phi_{a,b}(\omega)
 :=\int_0^1\bigl(\omega^2-q_{a,b}(r)\bigr)_+^{1/2}\dd r
\end{equation}
denote the action and let $N_{a,b}(\omega)$ denote the number of positive eigenfrequencies not
exceeding $\omega$.  Our aim is to compare $N_{a,b}(\omega)$ with $\Phi_{a,b}(\omega)/\pi$, with an error bound uniform
in the angular parameter $b$.
Since $q_{a,b}$ is convex, we first establish a
uniform counting estimate for nonnegative convex
potentials.

\begin{lemma}[Convex-potential counting]\label{lem:convex-count}
Let $I$ be a bounded open interval and let $q\in C^2(I)$ be
non-negative and convex.  Let $H_q$ denote the Friedrichs
realization in $L^2(I)$ of the differential operator
\[
-\frac{\dd^2}{\dd x^2}+q(x),
\]
initially defined on $C_c^\infty(I)$.
At regular endpoints, this realization imposes the Dirichlet condition.

Denote the eigenvalues of $H_q$, arranged in increasing order, by
\[
0<\lambda_1(H_q)<\lambda_2(H_q)<\cdots,
\]
and define the frequency counting function by
\[
N_q(\omega)
:=\#\{k\geq1:\lambda_k(H_q)\leq\omega^2\},
\qquad \omega\geq0.
\]
Then
\begin{equation}\label{eq:convex-count}
 \left|N_q(\omega)-\frac1\pi
 \int(\omega^2-q(r))_+^{1/2}\dd r\right|\leq C_{\rm osc},
 \qquad \omega\geq0,
\end{equation}
where one may take the absolute constant $C_{\rm osc}=8$.
\end{lemma}
\begin{proof}
We first work on a regular interval $I=(\alpha,\beta)$, with
$q\in C^2([\alpha,\beta])$ and Dirichlet conditions, using
Dirichlet--Neumann bracketing as in
\cite[Section~9.1]{Pankrashkin2014}. Singular endpoints are treated
by exhaustion at the end. Fix $E=\omega^2$ and write
$\Phi_q(\omega)=\int_I(E-q)_+^{1/2}\,\dd r$.

\smallskip
\noindent\emph{\textbf{Step 1: The allowed interval and the phase equation.}}
The set $J=\{r\in I:q(r)<E\}$ is an interval by convexity:
$q((1-t)r_1+tr_2)\leq(1-t)q(r_1)+tq(r_2)<E$ for
$r_1,r_2\in J$ and $0<t<1$.
If $J=\varnothing$, then $\Phi_q(\omega)=0$, while the
Dirichlet Poincar\'e inequality and the Rayleigh quotient give
$\lambda_1(H_q)\geq E+\pi^2/|I|^2>E$, hence $N_q(\omega)=0$.
Assume henceforth that $J\neq\varnothing$ and set $k=\sqrt{E-q}$ on $J$.
For a nonzero real solution of $v''+k^2v=0$, write
$v'=R\cos\theta$, $kv=R\sin\theta$, where
$R=(v'^2+k^2v^2)^{1/2}>0$ and $\theta$ is continuous. Then
\[
 \theta'
 =\frac{v'(kv)'-kvv''}{v'^2+k^2v^2}
 =k+\frac{k'}{2k}\sin2\theta.
\]
With $\tau(r)=\int_{r_*}^r k(s)\,\dd s$ and $A=k'/(2k^2)$,
this becomes
\begin{equation}\label{eq:convex-pruefer}
 \theta_\tau=1+A\sin2\theta.
\end{equation}
Convexity also gives
\begin{equation}\label{eq:concave-wave-number}
 k''=-\frac{q''}{2k}-\frac{(q')^2}{4k^3}\leq0,
 \qquad
 A_\tau=\frac{k''}{2k^3}-\frac{(k')^2}{k^4}
       =\frac{k''}{2k^3}-4A^2\leq-4A^2.
\end{equation}
Thus $k$ is concave and $A$ is non-increasing. Accordingly,
\[
 J_{\rm l}=\{A>1/2\},\qquad
 J_{\rm m}=\{|A|\leq1/2\},\qquad
 J_{\rm r}=\{A<-1/2\}
\]
are intervals occurring in this order; endpoints of the pieces
do not affect the integrals, and empty or degenerate pieces are omitted.

\smallskip
\noindent\emph{\textbf{Step 2: The outer pieces.}}
On a nonempty $J_{\rm l}$, $k'>k^2>0$.
Let $K$ and $d$ be the limiting values of $k$ and $k'$ at its
right endpoint. Concavity gives $k'\geq d\geq K^2$, so changing
variables from $r$ to $k$ yields
\[
 \int_{J_{\rm l}}k\,\dd r
 \leq\frac1d\int_0^K s\,\dd s
 =\frac{K^2}{2d}\leq\frac12.
\]
The same argument after reflection applies to $J_{\rm r}$.
All endpoint calculations may first be made on compact subintervals
and then passed to the limit. Hence
\begin{equation}\label{eq:convex-outer-actions}
 \int_{J_{\rm l}}k\,\dd r\leq\frac12,\qquad
 \int_{J_{\rm r}}k\,\dd r\leq\frac12.
\end{equation}
If an outer interval has length $l$ and maximum wave number $K$,
concavity and nonnegativity place $k$ above its endpoint chord,
giving $lK/2\leq\int k\leq1/2$. Thus $lK\leq1$, and the
Dirichlet operator on that interval satisfies, in the form sense,
\[
 -\partial_r^2+q-E=-\partial_r^2-k^2
 \geq-\partial_r^2-K^2
 \geq\frac{\pi^2-1}{l^2}>0.
\]
It therefore contributes no eigenvalues at or below $E$.
Neither do the pieces outside $J$, where $q\geq E$ and the
Dirichlet kinetic energy is strictly positive.

\smallskip
\noindent\emph{\textbf{Step 3: The middle piece.}}
Suppose $J_{\rm m}=(\alpha_{\rm m},\beta_{\rm m})$ is nondegenerate.
Here $|(1/k)'|=2|A|\leq1$, so $k$ stays positive at its endpoints.
Let $v_{\rm m}$ solve $-v_{\rm m}''+qv_{\rm m}=Ev_{\rm m}$ with
$v_{\rm m}(\alpha_{\rm m})=0$, $v_{\rm m}'(\alpha_{\rm m})=1$,
and choose $\theta(\alpha_{\rm m})=0$.
Since $1/2\leq\theta_\tau\leq3/2$, the angle increases strictly,
and Sturm oscillation gives
\begin{equation}\label{eq:convex-middle-floor}
 N_{\rm m}(\omega)=
 \left\lfloor\frac{\theta(\beta_{\rm m})}{\pi}\right\rfloor,
 \qquad
 \left|N_{\rm m}(\omega)-\frac{\Delta\theta}{\pi}\right|\leq1.
\end{equation}
The floor includes a zero at $\beta_{\rm m}$ when $E$ is an eigenvalue.

Let $\tau_{\rm l},\tau_{\rm r}$ be the endpoint action coordinates.
Using $(\cos2\theta)_\tau=-2\sin2\theta(1+A\sin2\theta)$,
we integrate \eqref{eq:convex-pruefer} by parts to obtain
\begin{align*}
 \Delta\theta-\int_{J_{\rm m}}k\,\dd r
 &=\int_{\tau_{\rm l}}^{\tau_{\rm r}}A\sin2\theta\,\dd\tau\\
 &=-\frac12[A\cos2\theta]_{\tau_{\rm l}}^{\tau_{\rm r}}
   +\frac12\int_{\tau_{\rm l}}^{\tau_{\rm r}}
      A_\tau\cos2\theta\,\dd\tau\\
 &\quad-\int_{\tau_{\rm l}}^{\tau_{\rm r}}
      A^2\sin^22\theta\,\dd\tau.
\end{align*}
Writing $D=A(\tau_{\rm l})-A(\tau_{\rm r})\in[0,1]$,
monotonicity and \eqref{eq:concave-wave-number} give
\[
 \frac12|[A\cos2\theta]_{\tau_{\rm l}}^{\tau_{\rm r}}|
 \leq\frac12,\qquad
 \frac12\int_{\tau_{\rm l}}^{\tau_{\rm r}}|A_\tau|\,\dd\tau
 =\frac D2,\qquad
 \int_{\tau_{\rm l}}^{\tau_{\rm r}}A^2\,\dd\tau\leq\frac D4.
\]
Consequently $|\Delta\theta-\int_{J_{\rm m}}k\,\dd r|
\leq\frac12+\frac D2+\frac D4\leq\frac54$, and
\begin{equation}\label{eq:convex-middle-count}
 \left|N_{\rm m}(\omega)
       -\frac1\pi\int_{J_{\rm m}}k\,\dd r\right|
 \leq1+\frac{5}{4\pi}.
\end{equation}

\smallskip
\noindent\emph{\textbf{Step 4: Reassembly.}}
Impose Dirichlet conditions at the interior endpoints of $J$
and between its three pieces, introducing at most four cuts.
The resulting direct-sum operator $H_q^{\rm cut}$ has counting
function $N_q^{\rm cut}=N_{\rm m}$ by Step~2.
These cuts impose at most four independent constraints on the form
domain; hence the min--max principle
\cite[Theorem~4.10]{TeschlQM2009} gives
$N_{\rm m}\leq N_q\leq N_{\rm m}+4$.
Since $\Phi_q=\int_{J_{\rm l}}k+\int_{J_{\rm m}}k+\int_{J_{\rm r}}k$,
\eqref{eq:convex-outer-actions} and \eqref{eq:convex-middle-count} imply
\[
 \left|N_q(\omega)-\frac{\Phi_q(\omega)}{\pi}\right|
 \leq4+\left(1+\frac{5}{4\pi}\right)+\frac1\pi
 =5+\frac{9}{4\pi}<6.
\]
If $J_{\rm m}$ is empty or degenerate, set its count and action
equal to zero; Step~2 and the same cutting argument give the bound.

\smallskip
\noindent\emph{\textbf{Step 5: Singular endpoints.}}
For $q\in C^2(I)$ as in the lemma, choose nested intervals
$I_j\Subset I$ exhausting $I$, and let $H_{q,j}$ be their
Dirichlet realizations. The preceding estimate holds uniformly in $j$:
\[
 \left|N_{q,j}(\omega)
       -\frac1\pi\int_{I_j}(\omega^2-q)_+^{1/2}\,\dd r\right|\leq6.
\]
After zero extension, their form domains are increasing and their
union is dense in the Friedrichs form domain, since it contains
$C_c^\infty(I)$. The min--max principle therefore gives
$\lambda_k(H_{q,j})\downarrow\lambda_k(H_q)$ for every fixed $k$.
Thus $N_{q,j}(\omega)\to N_q(\omega)$ when $\omega^2\notin\sigma(H_q)$,
while monotone convergence gives convergence of the action integrals.
Passing to the limit proves the bound at these energies.
At an eigenvalue, approach from above through energies outside the
spectrum: the count is right-continuous, and the action is continuous
by dominated convergence on the bounded interval.
The same bound follows, so in particular $C_{\rm osc}=8$ is valid.
\end{proof}

Since
\[
    q_{a,b}''(r)=6b\,r^{-4}+6a(1-r)^{-4}>0,
\]
\cref{lem:convex-count} gives
\begin{equation}\label{eq:N-Phi}
    \left|N_{a,b}(\omega)-\pi^{-1}\Phi_{a,b}(\omega)\right|
    \leq C_{\rm osc},
\end{equation}
uniformly for $a>0$, $b\geq0$, and $\omega\geq0$.

For the exceptional sector $\nu=0$, the coefficient
$b=-1/4$ makes the potential nonconvex.
Insert a Dirichlet condition at $r=1/2$.
The two resulting operators are bounded perturbations
of Bessel operators with eigenvalues $(2j_{0,k})^2$
and $(2j_{\mu,k})^2$, respectively.
The min--max principle, the fixed-order asymptotics
$j_{\rho,k}=\pi k+O_\rho(1)$, and the counting error
of at most one from this cut yield
\[
    N_{0,\mu}(\omega)
    =N_{\mathrm L}(\omega)+N_{\mathrm R}(\omega)+O(1)
    =\frac{\omega}{\pi}+O_\mu(1).
\]
Consequently, absorbing a possible endpoint eigenvalue
into the constant,
\[
    \#\{k:\omega_{0,k}\in[s,s+R]\}
    \leq\frac{R}{\pi}+C_\mu,
    \qquad s\geq0,\quad R>0.
\]

Now, to derive the required eigenfrequency density estimate by \cref{eq:N-Phi} for general $\nu\geq1/2$, we need to control the increments of $\Phi_{a,b}$. Since this action is not readily evaluated explicitly, we compare it with the standard Bessel action, which admits an explicit formula.
The following lemma bounds the contribution of the
fixed right-endpoint inverse-square potential uniformly in $b$ and $\omega$, allowing us to deduce the density estimate from the Bessel action.

\begin{lemma}\label{lem:action-defect}
For $a>0$, $b\geq0$, and $\omega>0$, define the
Bessel action by
\begin{equation}\label{eq:Bessel-action}
    \Theta_b(\omega)
    :=\int_0^1
    \left(\omega^2-\frac{b}{r^2}\right)_+^{1/2}\dd r.
\end{equation}
Then
\begin{equation}\label{eq:action-defect}
 0\leq\Theta_b(\omega)-\Phi_{a,b}(\omega)
 \leq\pi\sqrt{2a}.
\end{equation}
The constant is independent of both $b$ and $\omega$.
\end{lemma}
\begin{proof}
The proof follows from a direct calculation using
a suitable change of variables. Since $a>0$, we have $\Phi_{a,b}\leq\Theta_b$.
If $\omega^2\leq b$, both actions vanish.
Otherwise, set
\[
    c:=\frac{\sqrt b}{\omega}\in[0,1),
    \qquad
    \delta:=\frac{\sqrt a}{\omega},
    \qquad
    p_c(r):=\sqrt{1-\frac{c^2}{r^2}}.
\]
Both integrands vanish for $r\leq c$, so factoring
out $\omega p_c(r)$ gives
\begin{equation}\label{eq:action-loss-before-change}
    \frac{\Theta_b(\omega)-\Phi_{a,b}(\omega)}{\omega}
    =
    \int_c^1 p_c(r)
    F\!\left(\frac{\delta}{(1-r)p_c(r)}\right)\dd r,
    \qquad
    F(z):=1-\sqrt{(1-z^2)_+}.
\end{equation}

To absorb $p_c(r)$ into the integration measure,
introduce the action coordinate
\begin{equation}\label{eq:action-coordinate}
    x(r):=\int_c^r p_c(s)\dd s,
    \qquad
    L:=x(1),
    \qquad
    g(x(r)):=(1-r)p_c(r).
\end{equation}
Since $x$ maps $(c,1)$ increasingly onto $(0,L)$
and $\dd x=p_c(r)\dd r$, we obtain
\begin{equation}\label{eq:defect-F}
    \frac{\Theta_b(\omega)-\Phi_{a,b}(\omega)}{\omega}
    =
    \int_0^L F\!\left(\frac{\delta}{g(x)}\right)\dd x.
\end{equation}

We next compare $g$ with the distance to the endpoints:
\begin{equation}\label{eq:g-distance}
    g(x)\geq 2^{-1/2}\min\{x,L-x\}.
\end{equation}
Indeed, $p_c$ is increasing. For $r\leq(1+c)/2$,
\[
    x\leq(r-c)p_c(r)\leq(1-r)p_c(r)=g(x).
\]
For $r\geq(1+c)/2$, we have
\[
    L-x\leq(1-r)p_c(1),
    \qquad
    \frac{p_c(r)^2}{p_c(1)^2}
    \geq\frac{1+3c}{(1+c)^3}\geq\frac12,
\]
which gives $g(x)\geq(L-x)/\sqrt2$ and proves
\eqref{eq:g-distance}.

Finally, monotonicity of $F$ and
\eqref{eq:defect-F}--\eqref{eq:g-distance} yield
\begin{align*}
    \frac{\Theta_b-\Phi_{a,b}}{\omega}
    &\leq
    2\int_0^{L/2}
    F\!\left(\frac{\sqrt2\delta}{x}\right)\dd x\\
    &\leq
    2\sqrt2\delta\int_0^\infty F(1/t)\dd t
    =\pi\sqrt2\delta.
\end{align*}
Here the substitution $s=1/t$ and integration by parts give
\[
    \int_0^\infty F(1/t)\dd t
    =
    1+\int_1^\infty
    \left(1-\sqrt{1-t^{-2}}\right)\dd t
    =\frac{\pi}{2}.
\]
Since $\omega\delta=\sqrt a$, this proves
\eqref{eq:action-defect}.
\end{proof}

Now, for $\omega>\sqrt b$, the elementary integral in
\eqref{eq:Bessel-action} is
\begin{equation}\label{eq:Theta-explicit}
 \Theta_b(\omega)
 =\sqrt{\omega^2-b}-\sqrt b\,
       \arccos\!\left(\frac{\sqrt b}{\omega}\right),
 \qquad
 \partial_\omega\Theta_b(\omega)
 =\sqrt{1-\frac b{\omega^2}}\leq1.
\end{equation}
Combining \eqref{eq:N-Phi} and \eqref{eq:action-defect} yields
\begin{equation}\label{eq:uniform-quantization-count}
 \left|N_{a,b}(\omega)-\frac1\pi\Theta_b(\omega)\right|
 \leq C_a
\end{equation}
with $C_a=C_{\rm osc}+\sqrt{2a}$.  Consequently, for
$0\leq s_1<s_2$,
\begin{equation}\label{eq:positive-local-count}
 \#\{k:s_1<\omega_{\nu,k}\leq s_2\}
 \leq\frac{s_2-s_1}{\pi}+2C_a.
\end{equation}
Combining this estimate with the separately treated $\nu=0$
sector above, reflecting the positive-frequency estimates, and
splitting intervals crossing the origin gives \eqref{eq:quant-count}.
In particular, \eqref{eq:uniform-density} follows.  Notice that the sharp coefficient
$1/\pi$ comes only from the inequality
$\partial_\omega\Theta_b\leq1$; all singular endpoint effects have been
confined to an $O_a(1)$ phase loss.
\subsection{Proof of the uniform frequency gap}
In this section, we proceed to prove the uniform frequency gap. When the right-endpoint inverse-square term is absent, the radial operators reduce to Bessel operators with a Dirichlet condition at $r=1$.
Their explicit eigenvalue formulas and the corresponding zero-spacing properties yield a frequency gap bounded below uniformly in the angular parameter; see \cite[Propositions~7.4(a) and~7.8(b)--(c)]{KomornikLoreti2005}. Since the Bessel operator itself is obtained from the
one-dimensional Dirichlet Laplacian by adding an
inverse-square potential at the left endpoint,
this suggests that uniform separation may persist
when a second inverse-square singularity is introduced
at the right endpoint. The explicit Bessel formulas are then no longer available, however, and we establish the required uniform gap by a different argument. 

The proof is inspired by the strategy of Hillairet and Marzuola
\cite[Sections~2--3 and~5]{HillairetMarzuola2023}.
We argue by contradiction and classify hypothetical
sequences of vanishing spectral gaps according to
their energy scales. Normalization and energy-dependent rescaling then reduce the analysis to uniform estimates for families of semiclassical operators and to the spectral properties of limiting model operators. Throughout this subsection, fix
\[
    \mu>\frac12,\qquad
    a:=\mu^2-\frac14>0,
    \qquad
    \mathcal V:=\{0\}\cup[1/2,\infty).
\]
Fix $\nu\in\mathcal V$ and $k\geq1$, and write
\[
    \omega:=\omega_{\nu,k},
    \qquad
    \omega':=\omega_{\nu,k+1}.
\]
Following the semiclassical approach described above,
we normalize the lower eigenvalue by setting
\[
    h:=\omega^{-1},
    \qquad
    \tau:=h\nu.
\]
The rescaled Friedrichs operator is
\[
    Q_{h,\tau}
    :=h^2H_{\nu,\mu}
    =
    -h^2\partial_r^2
    +\frac{\tau^2-h^2/4}{r^2}
    +\frac{h^2a}{(1-r)^2},
    \qquad 0<r<1,
\]
and its corresponding consecutive eigenvalues are
\[
    1,
    \qquad
    \left(\frac{\omega'}{\omega}\right)^2.
\]
As shown below, any hypothetical sequence of vanishing
frequency gaps must satisfy
$\omega\to\infty$ and $\omega'/\omega\to1$.
We are therefore led to study this semiclassical
operator family near energy one.

Following the approach of
\cite[Section~3.1]{HillairetMarzuola2023},
we next establish an eigenvalue gap estimate for
a semiclassical operator on the interval $(0,L)$,
which will be used to prove the uniform radial
frequency gap estimate. 

Fix $h>0$, and let $V\in L^1_{\mathrm{loc}}(0,L)$
be real-valued. Let $H_h$ be the self-adjoint operator
in $L^2(0,L)$ associated with the differential expression
\[
    \tau_h f:=-h^2f''+Vf
\]
and fixed separated self-adjoint boundary conditions
at the endpoints, with no boundary condition imposed
at a limit-point endpoint.
Assume that $H_h$ has discrete simple spectrum.
The potential and the boundary conditions are
independent of the energy parameter $E$.

Fix $x_0\in (0,L)$ and an energy interval $K$. For each $E\in K$, choose nonzero real solutions
$f_{\mathrm L}(\cdot;E)$ and $f_{\mathrm R}(\cdot;E)$
of
\[
    -h^2f''+Vf=Ef
    \quad\text{on }(0,L),
\]
satisfying the prescribed condition at the left
and right endpoints, respectively.
At a limit-point endpoint, this means square
integrability near that endpoint.
Each endpoint requirement determines its solution
up to a nonzero scalar factor.

Assume that these families are $C^1$ in $E$ with values
in $C^1_{\mathrm{loc}}(0,L)$, that
\[
    f_{\mathrm L}\in L^2(0,x_0),
    \qquad
    f_{\mathrm R}\in L^2(x_0,L),
\]
and that
\[
    \lim_{x\downarrow0}
    \mathcal W(f_{\mathrm L},\partial_Ef_{\mathrm L})(x)=0,
    \qquad
    \lim_{x\uparrow L}
    \mathcal W(f_{\mathrm R},\partial_Ef_{\mathrm R})(x)=0,
\]
where $\mathcal W(f,g):=fg'-f'g$. 

We now introduce the Pr\"ufer angles at the matching
point $x_0$. We use the scaled Cauchy coordinates
$(hf',f)$ as in \cite{HillairetMarzuola2023}, rather than $(h^2f',f)$ arising from the
convention $(pf',f)$ in
\cite[Section~5.5]{TeschlODE2012} with $p=h^2$.
This scaling is adapted to the semiclassical estimates
below. Define
\[
    Z_{\mathrm L}(E)
    :=h f_{\mathrm L}'(x_0;E)+if_{\mathrm L}(x_0;E),
    \qquad
    Z_{\mathrm R}(E)
    :=h f_{\mathrm R}'(x_0;E)+if_{\mathrm R}(x_0;E).
\]
Both functions are nonvanishing by uniqueness for the
interior initial-value problem. Choose continuous
arguments on $K$ and set
\[
    \Theta(E)
    :=\arg Z_{\mathrm L}(E)-\arg Z_{\mathrm R}(E).
\]
The Wronskian identity corresponding to
\cite[equation~(5.97)]{TeschlODE2012} (see also
\cite[equation~(2.6)]{KrugerTeschl2009}) is then
\[
    h\,\mathcal W(f_{\mathrm L},f_{\mathrm R})(x_0)
    =
    |Z_{\mathrm L}(E)|\,|Z_{\mathrm R}(E)|
    \sin\Theta(E).
\]
\begin{lemma}[Eigenvalue gaps]
\label{lem:interior-matching-phase}
Under the preceding assumptions, $\Theta$ is $C^1$ and
\begin{equation}\label{eq:interior-phase-identity}
    \Theta'(E)
    =
    \frac1h
    \left(
        \frac{\displaystyle\int_0^{x_0}f_{\mathrm L}(x;E)^2\,\dd x}
             {|Z_{\mathrm L}(E)|^2}
        +
        \frac{\displaystyle\int_{x_0}^Lf_{\mathrm R}(x;E)^2\,\dd x}
             {|Z_{\mathrm R}(E)|^2}
    \right)>0,
    \qquad E\in K.
\end{equation}
Moreover,
\[
    E\in\sigma(H_h)\cap K
    \quad\Longleftrightarrow\quad
    E\in K\ \text{and}\ \Theta(E)\in\pi\mathbb Z.
\]
If constants $C_{\mathrm L},C_{\mathrm R}>0$ satisfy
\begin{equation}\label{eq:matching-mass-assumptions}
    \int_0^{x_0}f_{\mathrm L}(x;E)^2\,\dd x
    \leq C_{\mathrm L}|Z_{\mathrm L}(E)|^2,
    \qquad
    \int_{x_0}^Lf_{\mathrm R}(x;E)^2\,\dd x
    \leq C_{\mathrm R}|Z_{\mathrm R}(E)|^2
\end{equation}
for every $E\in K$, then any two consecutive
eigenvalues $E_j<E_{j+1}$ of $H_h$ lying in $K$ satisfy
\begin{equation}\label{eq:interior-matching-gap}
    E_{j+1}-E_j
    \geq\frac{\pi h}{C_{\mathrm L}+C_{\mathrm R}}.
\end{equation}
\end{lemma}

\begin{proof}
We adapt the phase-variation argument of
\cite[Section~3.1]{HillairetMarzuola2023}
by comparing the phases of the left- and
right-admissible solutions at an interior matching point. Throughout this proof, a dot denotes differentiation
with respect to $E$, and a prime denotes differentiation with respect to $x$.
For either solution $f$, differentiating its equation gives
\[
    -h^2\dot f''+V\dot f=f+E\dot f,
\]
and hence
\[
    \mathcal W(f,\dot f)'=-h^{-2}f^2.
\]
Integration from the corresponding endpoint yields
\[
    \mathcal W(f_{\mathrm L},\dot f_{\mathrm L})(x_0)
    =-h^{-2}\int_0^{x_0}f_{\mathrm L}^2\,\dd x,
\]
and
\[
    \mathcal W(f_{\mathrm R},\dot f_{\mathrm R})(x_0)
    =h^{-2}\int_{x_0}^Lf_{\mathrm R}^2\,\dd x.
\]

A nonzero solution cannot have both its value and
derivative zero at $x_0$. Thus $Z_{\mathrm L}$ and
$Z_{\mathrm R}$ are nonvanishing $C^1$ functions.
For $Z=hf'(x_0)+if(x_0)$,
\[
    \operatorname{Im}(\overline Z\dot Z)
    =-h\,\mathcal W(f,\dot f)(x_0),
    \qquad
    \frac{\dd}{\dd E}\arg Z
    =
    \frac{\operatorname{Im}(\overline Z\dot Z)}{|Z|^2}.
\]
Therefore
\[
    \frac{\dd}{\dd E}\arg Z_{\mathrm L}
    =
    \frac{\int_0^{x_0}f_{\mathrm L}^2\,\dd x}
         {h|Z_{\mathrm L}|^2},
    \qquad
    \frac{\dd}{\dd E}\arg Z_{\mathrm R}
    =
    -\frac{\int_{x_0}^Lf_{\mathrm R}^2\,\dd x}
          {h|Z_{\mathrm R}|^2}.
\]
Subtracting proves \eqref{eq:interior-phase-identity}.

The energy $E$ is an eigenvalue precisely when the
left- and right-admissible solutions are proportional,
or equivalently,
\[
    \mathcal W(f_{\mathrm L},f_{\mathrm R})(x_0)=0.
\]
By the scaled version of the Wronskian identity
\cite[equation~(5.97)]{TeschlODE2012} established above,
and since $Z_{\mathrm L}(E)$ and $Z_{\mathrm R}(E)$
are nonzero, this is equivalent to
\[
    \sin\Theta(E)=0,
    \qquad\text{that is,}\qquad
    \Theta(E)\in\pi\mathbb Z.
\]

Since $\Theta$ is strictly increasing, consecutive
eigenvalues satisfy
\[
    \Theta(E_{j+1})-\Theta(E_j)=\pi.
\]
Indeed, a larger increase would give another energy
between them at which $\Theta$ belongs to $\pi\mathbb Z$.
Under \eqref{eq:matching-mass-assumptions},
\[
    0<\Theta'(E)
    \leq\frac{C_{\mathrm L}+C_{\mathrm R}}{h}.
\]
Integration between $E_j$ and $E_{j+1}$ proves
\eqref{eq:interior-matching-gap}.
\end{proof}

For the inverse-square equations considered below, we
choose the normalized principal solutions at the
respective endpoints. The principal-solution estimates
recalled in the preliminaries ensure both their $C^1$
dependence on $E$ and the
vanishing of the corresponding endpoint Wronskians,
including the critical logarithm-free case.

To verify the mass bounds required by
\cref{lem:interior-matching-phase}, we next establish
a uniform transfer estimate controlling the $L^2$ mass
of the principal solution by its semiclassically scaled Cauchy data at an interior point.
\begin{lemma}[Uniform transfer for the principal solution]
\label{lem:endpoint-transfer1}
Fix $\rho>1/2$, $d>0$, and a compact real interval $K$.
Put
\[
    a_\rho:=\rho^2-\frac14,
    \qquad p_\rho:=\frac12+\rho.
\]
Suppose that
\[
    W\in C^1([0,d];\mathbb R),\qquad
    \|W\|_{C^1([0,d])}\leq C_*,
    \qquad
    E-W(x)\geq\eta_*>0
\]
for every $E\in K$ and $x\in[0,d]$.
Consider
\begin{equation}\label{eq:endpoint-equation1}
    -h^2v''
    +\frac{h^2a_\rho}{x^2}v
    +Wv=Ev,
    \qquad 0<x<d.
\end{equation}
Assume that $v$ satisfies the Friedrichs condition at
zero, and write
\[
    v(x)=A\phi_{\rho,h}(x;E),
    \qquad
    \lim_{x\downarrow0}
    x^{-p_\rho}\phi_{\rho,h}(x;E)=1.
\]
There exist $c,C,L_0,h_0>0$, depending only on the
displayed fixed parameters and bounds, such that
$L_0h<d$ and
\begin{equation}\label{eq:endpoint-cauchy1}
    c h^{2p_\rho}|A|^2
    \leq |v(d)|^2+h^2|v'(d)|^2
    \leq C h^{2p_\rho}|A|^2
\end{equation}
for $0<h<h_0$.
Moreover,
\begin{align}
    \int_0^d|v|^2\,\dd x
    &\leq C h^{2p_\rho}|A|^2,
    \notag\\
    \int_0^{L_0h}|v|^2\,\dd x
    &\leq C h^{2p_\rho+1}|A|^2.
    \label{eq:endpoint-mass1}
\end{align}
In particular,
\begin{equation}\label{eq:principal-mass-cauchy}
    \int_0^d|v|^2\,\dd x
    \leq C\bigl(|v(d)|^2+h^2|v'(d)|^2\bigr).
\end{equation}
\end{lemma}

\begin{proof}
Choose $L_0\geq1$ such that
$a_\rho/L_0^2\leq\eta_*/2$, and take $h_0\leq1$
sufficiently small that $L_0h_0<d$.

We first estimate the principal solution on the scale
$x=hs$. Set
\[
    F_h(s):=h^{-p_\rho}\phi_{\rho,h}(hs;E).
\]
Then
\[
    -F_h''+\frac{a_\rho}{s^2}F_h
    +(W(hs)-E)F_h=0,
    \qquad
    \lim_{s\downarrow0}s^{-p_\rho}F_h(s)=1,
\]
and $v(hs)=Ah^{p_\rho}F_h(s)$.
Since $W(hs)-E$ is uniformly bounded on $[0,L_0]$,
the principal-solution estimates give
\[
    F_h(s)=s^{p_\rho}(1+O(s^2)),
    \qquad
    F_h'(s)=p_\rho s^{p_\rho-1}+O(s^{p_\rho+1})
\]
on a fixed sufficiently small interval, with uniform
remainder bounds. Thus we may choose a fixed
$s_*\in(0,L_0)$ and constants $c_1,C_1>0$ such that
\[
    c_1\leq |F_h(s_*)|^2+|F_h'(s_*)|^2\leq C_1.
\]

To propagate these bounds to $L_0$, write
\[
    Y_h(s):=
    \begin{pmatrix}F_h(s)\\ F_h'(s)\end{pmatrix},
    \qquad
    Y_h'(s)=B_h(s)Y_h(s),
    \qquad
    B_h(s):=
    \begin{pmatrix}
        0&1\\
        a_\rho s^{-2}+W(hs)-E&0
    \end{pmatrix}.
\]
On $[s_*,L_0]$, the matrix norm satisfies
$\|B_h(s)\|\leq M$ for a constant independent of
$h,E,W$ within the stated bounds.
Let $\Psi_h$ solve
\[
    \Psi_h'=B_h\Psi_h,
    \qquad \Psi_h(s_*)=I.
\]
Gronwall's inequality for the forward and backward
systems gives
\[
    \|\Psi_h(s)\|\leq e^{M(s-s_*)},
    \qquad
    \|\Psi_h(s)^{-1}\|\leq e^{M(s-s_*)}.
\]
Since $Y_h(s)=\Psi_h(s)Y_h(s_*)$, it follows that
\[
    e^{-M(s-s_*)}|Y_h(s_*)|
    \leq |Y_h(s)|
    \leq e^{M(s-s_*)}|Y_h(s_*)|.
\]
Together with the principal estimates on $(0,s_*)$,
this yields
\[
    c_2\leq |F_h(L_0)|^2+|F_h'(L_0)|^2\leq C_2,
    \qquad
    \int_0^{L_0}|F_h(s)|^2\,\dd s\leq C_2.
\]
Returning to $x$ and writing $x_h=L_0h$, we obtain
\begin{align*}
    c_2h^{2p_\rho}|A|^2
    &\leq |v(x_h)|^2+h^2|v'(x_h)|^2
    \leq C_2h^{2p_\rho}|A|^2,\\
    \int_0^{x_h}|v|^2\,\dd x
    &\leq C_2h^{2p_\rho+1}|A|^2.
\end{align*}

We next propagate the Cauchy-data bounds from $x_h$ to $d$.
Define
\[
    K_h(x):=E-W(x)-\frac{h^2a_\rho}{x^2},
    \qquad
    S_h(x):=h^2|v'(x)|^2+K_h(x)|v(x)|^2.
\]
By the choice of $L_0$, there is a fixed $K_*>0$ such that
\[
    \frac{\eta_*}{2}\leq K_h(x)\leq K_*,
    \qquad x_h\leq x\leq d.
\]
Consequently, with
$m_*:=\min\{1,\eta_*/2\}$ and $M_*:=\max\{1,K_*\}$,
\[
    m_*\bigl(|v|^2+h^2|v'|^2\bigr)
    \leq S_h
    \leq M_*\bigl(|v|^2+h^2|v'|^2\bigr).
\]
Differentiating and using $h^2v''+K_hv=0$ gives
\[
    S_h'=K_h'|v|^2,
    \qquad
    |S_h'|\leq\frac{2}{\eta_*}|K_h'|S_h.
\]
Moreover,
\[
    \int_{x_h}^d|K_h'(x)|\,\dd x
    \leq dC_*+2h^2a_\rho\int_{x_h}^d x^{-3}\,\dd x
    \leq dC_*+\frac{a_\rho}{L_0^2}.
\]
Hence, setting
\[
    C_0:=\frac{2}{\eta_*}
          \left(dC_*+\frac{a_\rho}{L_0^2}\right),
\]
Gronwall's inequality in both directions yields
\[
    e^{-C_0}S_h(x_h)
    \leq S_h(x)
    \leq e^{C_0}S_h(x_h),
    \qquad x_h\leq x\leq d.
\]
Combining these inequalities with the bounds at $x_h$,
we obtain uniform constants $c,C>0$ such that
\[
    ch^{2p_\rho}|A|^2
    \leq |v(x)|^2+h^2|v'(x)|^2
    \leq Ch^{2p_\rho}|A|^2,
    \qquad x_h\leq x\leq d.
\]
Taking $x=d$ proves \eqref{eq:endpoint-cauchy1}.
Integrating the upper bound and adding the estimate
on $(0,x_h)$ gives
\[
    \int_0^d|v|^2\,\dd x
    \leq C(d+h)h^{2p_\rho}|A|^2
    \leq C' h^{2p_\rho}|A|^2,
\]
which, together with the near-endpoint estimate,
proves \eqref{eq:endpoint-mass1}.
Finally, combining this mass bound with the lower
bound in \eqref{eq:endpoint-cauchy1} proves
\eqref{eq:principal-mass-cauchy}.
\end{proof}
We next record an elementary identity in which the
inverse-square contribution cancels. This identity
will be used to control the $L^2$ mass of a solution
by its interior Cauchy data.
\begin{lemma}
\label{lem:rellich11}
If a real solution $v$ satisfies
\[
    -h^2v''+\left(\frac{g}{x^2}+W(x)\right)v=Ev,
\]
then
\begin{equation}\label{eq:rellich-flux1}
    F(x)
    :=
    x\left\{
        h^2(v')^2+
        \left(E-\frac{g}{x^2}-W\right)v^2
    \right\}
    -h^2vv'
\end{equation}
satisfies
\begin{equation}\label{eq:rellich-identity11}
    F'(x)=(2E-2W-xW')v^2.
\end{equation}
\end{lemma}
\begin{proof}
Differentiate \eqref{eq:rellich-flux1} and substitute
$h^2v''=(g/x^2+W-E)v$.
The terms containing $vv'$ cancel, and the
inverse-square contribution vanishes because
\[
    2\frac{g}{x^2}
    +x\frac{\dd}{\dd x}\left(\frac{g}{x^2}\right)=0.
\]
\end{proof}
With these preparations, we now establish uniform
spectral gaps and lower bounds for the principal
coefficients in the transversal regime for the rescaled Friedrichs operator $h^2H_{\nu,\mu}$; see
\cref{Transversal} for an explanation of the terminology: transversal regime.
\begin{lemma}[Uniform spectral gaps and principal-coefficient bounds in the transversal regime]
\label{lem:transversal-family11}
Fix $a>0$ and $\tau_*\in[0,1)$.
For $h>0$ and $0\leq\tau\leq\tau_*$, let
\begin{equation}\label{eq:transversal-family1}
    Q_{h,\tau}
    :=
    -h^2\partial_r^2
    +\frac{\tau^2-h^2/4}{r^2}
    +\frac{h^2a}{(1-r)^2}
\end{equation}
be the Friedrichs realization on $(0,1)$.
There exist $h_0,c>0$ and a fixed compact interval $K$
containing $1$ in its interior such that consecutive
eigenvalues of $Q_{h,\tau}$ lying in $K$ have gap at
least $ch$ whenever $0<h<h_0$.
The constants are independent of $h$ and $\tau$.
Moreover, put $p=\frac12+\sqrt{a+1/4}>1$.
If $v$ is an $L^2(0,1)$-normalized eigenfunction
with eigenvalue in $K$, write its principal
endpoint expansion as
\[
    v(1-y)=A y^p(1+O(y^2)),
    \qquad y:=1-r\downarrow0.
\]
Then
\begin{equation}\label{eq:transversal-principal-lower}
    1\leq C h^{2p}|A|^2,
\end{equation}
where $C>0$ is independent of $h$, $\tau$, and $v$.
\end{lemma}

\begin{proof}
We estimate the masses of the left and right principal
solutions on their respective sides of a fixed interior
point, and then apply \cref{lem:interior-matching-phase}.

Choose $\tau_*<r_0<1$ and a sufficiently small
$\delta>0$ so that
\[
    K:=[1-\delta,1+\delta],
    \qquad
    E_{\min}:=1-\delta>\frac{\tau_*^2}{r_0^2}.
\]
For each $E\in K$, let $f_{\mathrm L}(\cdot;E)$ and
$f_{\mathrm R}(\cdot;E)$ be the normalized principal
solutions at $r=0$ and $r=1$, respectively.
No normalization on the full interval is imposed.

We first control the right principal solution by using \cref{lem:endpoint-transfer1}.
Set $x=1-r$ and
\[
    \widetilde f_{\mathrm R}(x;E)
    :=f_{\mathrm R}(1-x;E).
\]
Its equation on $(0,1-r_0)$ is
\[
    -h^2\widetilde f_{\mathrm R}''
    +\frac{h^2a}{x^2}\widetilde f_{\mathrm R}
    +W_{\mathrm R}(x)\widetilde f_{\mathrm R}
    =E\widetilde f_{\mathrm R},
    \qquad
    W_{\mathrm R}(x)
    :=\frac{\tau^2-h^2/4}{(1-x)^2}.
\]
The potentials $W_{\mathrm R}$ have a uniform $C^1$
bound on $[0,1-r_0]$, and
\[
    E-W_{\mathrm R}(x)
    \geq E_{\min}-\frac{\tau_*^2}{r_0^2}>0.
\]
Applying \cref{lem:endpoint-transfer1} with
$\rho=\sqrt{a+1/4}>1/2$ gives
\begin{equation}\label{eq:transversal-right-mass1}
    \int_{r_0}^1f_{\mathrm R}^2\,\dd r
    \leq C\left(
        |f_{\mathrm R}(r_0)|^2
        +h^2|f_{\mathrm R}'(r_0)|^2
    \right).
\end{equation}
We next use \cref{lem:rellich11} to control the left principal solution. Put
\[
    g:=\tau^2-\frac{h^2}{4},
    \qquad
    W(r):=\frac{h^2a}{(1-r)^2},
\]
and define
\[
    F(r)
    :=
    r\left[
        h^2\bigl(f_{\mathrm L}'(r)\bigr)^2
        +
        \left(
            E-\frac{g}{r^2}
            -\frac{h^2a}{(1-r)^2}
        \right)f_{\mathrm L}(r)^2
    \right]
    -h^2f_{\mathrm L}(r)f_{\mathrm L}'(r).
\]

The principal exponent at zero is
\[
    s=\frac12+\frac{\tau}{h},
    \qquad h^2s(s-1)=g.
\]
The normalized principal solution therefore satisfies
\[
    f_{\mathrm L}(r)=r^s+O(r^{s+2}),
    \qquad
    f_{\mathrm L}'(r)=sr^{s-1}+O(r^{s+1}),
    \qquad r\downarrow0.
\]
For $\tau=0$, this is the logarithm-free critical
Friedrichs solution. These asymptotics are used with
$h$, $\tau$, and $E$ fixed.

Since $E-W(r)$ is bounded near zero, substitution
into the definition of $F$ gives
\[
    F(r)
    =
    \bigl(h^2s^2-g-h^2s\bigr)r^{2s-1}
    +O(r^{2s+1}).
\]
The leading coefficient vanishes because
$g=h^2s(s-1)$. As $s\geq1/2$, it follows that
\[
    F(r)=O(r^{2s+1})\longrightarrow0
    \qquad\text{as }r\downarrow0.
\]

By \eqref{eq:rellich-identity11},
\[
    F'(r)=\bigl(2E-2W(r)-rW'(r)\bigr)
          f_{\mathrm L}(r)^2.
\]
Here
\[
    W'(r)=\frac{2h^2a}{(1-r)^3},
    \qquad
    2W(r)+rW'(r)=\frac{2h^2a}{(1-r)^3}.
\]
Decreasing $h_0$ if necessary, we may assume that
\[
    \frac{2h_0^2a}{(1-r_0)^3}\leq E_{\min}.
\]
Consequently, for $0<h<h_0$, $0<r\leq r_0$,
and $E\in K$,
\[
    2E-2W(r)-rW'(r)
    =
    2E-\frac{2h^2a}{(1-r)^3}
    \geq
    2E_{\min}-\frac{2h_0^2a}{(1-r_0)^3}
    \geq E_{\min}.
\]
Integration of \eqref{eq:rellich-identity11} therefore gives
\[
    E_{\min}\int_0^{r_0}f_{\mathrm L}^2\,\dd r
    \leq F(r_0).
\]
The coefficients in $F(r_0)$ are uniformly bounded, Viewed as a quadratic form in $f_{\mathrm L}(r_0)$
and $h f_{\mathrm L}'(r_0)$, the expression $F(r_0)$
has coefficients bounded uniformly for
\[
    |g|\leq\tau_*^2+\frac{h_0^2}{4},
    \qquad
    W(r_0)\leq\frac{h_0^2a}{(1-r_0)^2}.
\] and
\[
    h^2|f_{\mathrm L}f_{\mathrm L}'|
    \leq
    \frac h2\left(
        |f_{\mathrm L}|^2+h^2|f_{\mathrm L}'|^2
    \right).
\]
Consequently,
\begin{equation}\label{eq:transversal-left-mass}
    \int_0^{r_0}f_{\mathrm L}^2\,\dd r
    \leq C\left(
        |f_{\mathrm L}(r_0)|^2
        +h^2|f_{\mathrm L}'(r_0)|^2
    \right).
\end{equation}

The endpoint differentiability and Wronskian conditions
of \cref{lem:interior-matching-phase} hold by the
principal-solution properties stated above.
The two estimates
\eqref{eq:transversal-right-mass1} and
\eqref{eq:transversal-left-mass} verify the mass
hypotheses of \cref{lem:interior-matching-phase}, uniformly for $E\in K$.
Hence consecutive eigenvalues in $K$ are separated
by at least $ch$.

For a normalized eigenfunction $v$, both mass estimates
apply to the same global solution. Thus
\[
    1\leq C\bigl(|v(r_0)|^2+h^2|v'(r_0)|^2\bigr).
\]
The right-endpoint Cauchy-data estimate
\eqref{eq:endpoint-cauchy1}, with the coefficient $A$
of $v(1-y)$, bounds the right-hand side by
$Ch^{2p}|A|^2$, proving
\eqref{eq:transversal-principal-lower}.
\end{proof}
\begin{remark}[Transversal and glancing regimes.]\label{Transversal}
The terminology is motivated by the geometry of the
billiard flow in the unit ball. The relevant quantity
is the angular momentum relative to the total frequency,
rather than the angular parameter alone. For a mode
with frequency $\omega$, set
\[
    h=\omega^{-1},
    \qquad
    \tau=h\nu=\frac{\nu}{\omega}.
\]
Then $h^2H_{\nu,\mu}=Q_{h,\tau}$, and the corresponding
eigenvalue is $1$. The radial principal symbol in the
interior is
\[
    \mathfrak p_\tau(r,\xi_r)
    =\xi_r^2+\frac{\tau^2}{r^2}.
\]
At unit energy, its normal momentum at the geometric
boundary satisfies
\[
    |\xi_r|=\sqrt{1-\tau^2}.
\]
We therefore call the range $\tau\leq\tau_*<1$
the \emph{uniformly transversal regime}: the normal
momentum stays bounded away from zero. The limit
$\tau\to1$ is the \emph{glancing regime}, in which
the normal momentum vanishes and the radial turning
point approaches the boundary. More generally, at
energy $E$, the relevant ratio is $\tau/\sqrt E$.

This distinction is illustrated by the disk eigenfunctions
discussed in
\cite[Section~1.3]{AnantharamanLeautaudMacia2016}.
Different asymptotic relations between the angular
and radial indices produce different limiting spatial
mass distributions, including whispering-gallery
concentration on the boundary when the radial index
is fixed and the angular index tends to infinity.

For our singular operators, this geometric picture
motivates the terminology. The inverse-square boundary
term is not uniformly negligible near $r=1$, and its
effect is treated explicitly by the endpoint estimates
above.
\end{remark}

We are now in a position to prove the uniform radial frequency gap \cref{eq:quant-gap}.

\begin{lemma}[Uniform radial frequency gap]
\label{lem:uniform-radial-gap1}
For every fixed $\mu>1/2$, there exists $c_\mu>0$ such that
\begin{equation}\label{eq:radial-gap1}
    \omega_{\nu,k+1}-\omega_{\nu,k}\geq c_\mu,
    \qquad \nu\in\mathcal V,\quad k\geq1.
\end{equation}
\end{lemma}

\begin{proof}
Suppose that the assertion fails.
Choose $\nu_j\in\mathcal V$ and $k_j\geq1$ such that
\begin{equation}\label{eq:gap-contradiction-sequence1}
    0<
    \omega_{\nu_j,k_j+1}-\omega_{\nu_j,k_j}
    <\frac1j.
\end{equation}
Write
\[
    \omega_j:=\omega_{\nu_j,k_j},
    \qquad
    \omega_j':=\omega_{\nu_j,k_j+1}.
\]
These are consecutive frequencies of the same operator
\[
    H_{\nu_j,\mu}
    =
    -\partial_r^2
    +\frac{\nu_j^2-\frac14}{r^2}
    +\frac{a}{(1-r)^2},
\]
with Friedrichs conditions at both endpoints. We divide the argument into several steps.

\smallskip
\noindent\emph{\textbf{Step 1:} Exclusion of bounded frequencies.}

We first prove that
\begin{equation}\label{eq:gap-frequencies-escape1}
    \omega_j\longrightarrow\infty.
\end{equation}
Suppose instead that $\omega_j\leq M$ along a subsequence.

If $\nu_j=0$ along an infinite subsequence, the fixed
operator $H_{0,\mu}$ has compact resolvent and simple
spectrum. In particular,
\[
    \omega_{0,k}^2\longrightarrow+\infty
    \qquad\text{as }k\to\infty.
\]
Hence $\omega_{0,k_j}\leq M$ forces $k_j$ to remain
bounded, and a further subsequence has $k_j=k_*$.
The gap then equals the fixed positive number
\[
    \omega_{0,k_*+1}-\omega_{0,k_*},
\]
a contradiction.

It remains to consider $\nu_j\geq1/2$.
For these parameters, the quadratic forms have common
domain $H_0^1(0,1)$:
\[
    \mathfrak h_\nu[v]
    =
    \int_0^1
    \left(
        |v'|^2
        +\frac{\nu^2-\frac14}{r^2}|v|^2
        +\frac{a}{(1-r)^2}|v|^2
    \right)\dd r.
\]
The endpoint Hardy inequalities justify this domain.
By the min--max principle,
\begin{equation}\label{eq:gap-bounded-frequency-minmax1}
    \omega_{\nu,k}^2
    \geq \pi^2k^2+\nu^2-\frac14,
    \qquad \nu\geq\frac12.
\end{equation}
Thus bounded $\omega_j$ forces both $k_j$ and $\nu_j$
to remain bounded. Passing to a further subsequence,
\[
    k_j=k_*,
    \qquad \nu_j\to\nu_*\geq\frac12.
\]

For $v\in H_0^1(0,1)$, Hardy's inequality gives
\[
    |\mathfrak h_{\nu_j}[v]-\mathfrak h_{\nu_*}[v]|
    \leq
    4|\nu_j^2-\nu_*^2|
    \int_0^1|v'|^2\,\dd r
    \leq
    d_j\,\mathfrak h_{\nu_*}[v],
\]
where $d_j:=4|\nu_j^2-\nu_*^2|\to0$.
For sufficiently large $j$,
\[
    (1-d_j)\mathfrak h_{\nu_*}
    \leq\mathfrak h_{\nu_j}
    \leq(1+d_j)\mathfrak h_{\nu_*}.
\]
The min--max principle consequently implies
\[
    \omega_{\nu_j,k_*}\to\omega_{\nu_*,k_*},
    \qquad
    \omega_{\nu_j,k_*+1}\to\omega_{\nu_*,k_*+1}.
\]
Simplicity of the limiting eigenvalues contradicts
the vanishing gap.
This proves \eqref{eq:gap-frequencies-escape1}.
In particular,
\[
    \omega_j\to\infty,
    \qquad \frac{\omega_j'}{\omega_j}\to1.
\]

\smallskip
\noindent\emph{\textbf{Step 2:} Exclusion of the transversal regime.}

Suppose that, along a subsequence,
\[
    \frac{\nu_j}{\omega_j}\leq\tau_*<1.
\]
Set
\[
    \widehat h_j:=\omega_j^{-1},
    \qquad
    \tau_j:=\frac{\nu_j}{\omega_j}.
\]
Then
\[
    \widehat h_j^2H_{\nu_j,\mu}
    =
    -\widehat h_j^2\partial_r^2
    +\frac{\tau_j^2-\widehat h_j^2/4}{r^2}
    +\frac{\widehat h_j^2a}{(1-r)^2},
\]
which belongs to the family in
\cref{lem:transversal-family11}.
Its eigenvalues with indices $k_j$ and $k_j+1$ are
consecutive and equal to
\[
    1,\qquad
    \left(\frac{\omega_j'}{\omega_j}\right)^2.
\]
They eventually lie in the fixed interval provided
by that lemma. Therefore
\[
    \left(\frac{\omega_j'}{\omega_j}\right)^2-1
    \geq c\widehat h_j.
\]
However,
\[
    \frac{
        (\omega_j'/\omega_j)^2-1
    }{\widehat h_j}
    =
    (\omega_j'-\omega_j)
    \left(\frac{\omega_j'}{\omega_j}+1\right)
    \longrightarrow0,
\]
a contradiction.

By \eqref{eq:gap-bounded-frequency-minmax1},
$\omega_{\nu,k}>\nu$ for $\nu\geq1/2$.
For $\nu=0$, the ratio $\nu/\omega_{\nu,k}$ is zero.
Since every subsequence on which $\nu_j/\omega_j$
stays below a fixed number less than one has been
excluded, we obtain
\begin{equation}\label{eq:gap-glancing-ratio1}
    \frac{\nu_j}{\omega_j}\longrightarrow1,
    \qquad \nu_j\longrightarrow\infty.
\end{equation}

\smallskip
\noindent\emph{\textbf{Step 3:} Angular normalization and boundary scaling.}

We now treat the remaining regime
\eqref{eq:gap-glancing-ratio1}.
Define
\begin{equation}\label{eq:normalised-angular1}
    h_j:=\nu_j^{-1},
    \qquad
    d_j:=1-\frac{h_j^2}{4},
    \qquad
    h_j^2\omega_j^2=1+\epsilon_j.
\end{equation}
Then $h_j\to0$, $\epsilon_j>0$, and $\epsilon_j\to0$.

We first establish the lower bound
\begin{equation}\label{Eq:bottom-lower1}
    \epsilon_j\geq c_a h_j^{2/3}
\end{equation}
for all sufficiently large $j$.
For $r\geq1/2$, put $y=1-r$.
Using $(1-y)^{-2}\geq1+2y$, we find
\begin{align*}
    \frac{d_j}{r^2}
    +\frac{h_j^2a}{(1-r)^2}-1
    &\geq
    2y+\frac{h_j^2a}{y^2}-h_j^2\\
    &\geq
    3a^{1/3}h_j^{2/3}-h_j^2.
\end{align*}
The last inequality follows by minimizing
$2y+h_j^2a/y^2$ over $y>0$.
For $r\leq1/2$, the same potential minus one is at least
\[
    4d_j-1=3-h_j^2.
\]
Thus the potential of $h_j^2H_{\nu_j,\mu}$ is bounded
below by $1+c_a h_j^{2/3}$.
The kinetic form is nonnegative, so the same lower
bound holds for its eigenvalues, proving
\eqref{Eq:bottom-lower1}.

Set
\begin{equation}\label{Eq:glancing-rescaling1}
    x:=\frac{1-r}{\epsilon_j},
    \qquad
    L_j:=\epsilon_j^{-1},
    \qquad
    \eta_j:=h_j\epsilon_j^{-3/2}.
\end{equation}
The unitary transformation
\[
    (U_jf)(x)
    :=\epsilon_j^{1/2}f(1-\epsilon_jx)
\]
gives
\begin{equation}\label{Eq:glancing-family1}
    P_j
    :=
    \frac{U_j(h_j^2H_{\nu_j,\mu})U_j^{-1}-1}{\epsilon_j}
    =
    -\eta_j^2\partial_x^2
    +\frac{\eta_j^2a}{x^2}
    +V_j(x),
\end{equation}
where
\[
    V_j(x)
    :=
    \frac{d_j(1-\epsilon_jx)^{-2}-1}{\epsilon_j},
    \qquad 0<x<L_j.
\]
The Friedrichs conditions are preserved.
The two consecutive eigenvalues become
\begin{equation}\label{Eq:glancing-two-energies1}
    1,\qquad
    z_j'
    :=
    1+\frac{h_j^2(\omega_j'^2-\omega_j^2)}{\epsilon_j}.
\end{equation}

By \eqref{Eq:bottom-lower1},
\[
    0<\eta_j\leq c_a^{-3/2},
    \qquad
    \frac{h_j^2}{\epsilon_j}\longrightarrow0.
\]
Furthermore,
\begin{equation}\label{Eq:glancing-gap-ratio1}
    \frac{z_j'-1}{\eta_j}
    =
    h_j\sqrt{\epsilon_j}
    (\omega_j'+\omega_j)(\omega_j'-\omega_j)
    \longrightarrow0,
\end{equation}
because $h_j\omega_j\to1$, $h_j\omega_j'\to1$,
and $\omega_j'-\omega_j\to0$.
Since $\eta_j$ is bounded, this also implies $z_j'\to1$.

Passing to a subsequence, we may assume
$\eta_j\to\eta_\infty\geq0$.
We consider the two possible cases separately.

\smallskip
\noindent\emph{\textbf{Step 4:} The case $\eta_j\to0$.}
\smallskip
In this regime, $P_j$ has the local form considered
in \cref{lem:endpoint-transfer1} near $x=0$, with
kinetic parameter $\eta_j\to0$.
To apply \cref{lem:interior-matching-phase}, it remains
to establish the required $L^2$ mass bounds for the
left and right principal solutions in terms of their
Cauchy data at a fixed interior matching point.

Fix
\[
    K:=\left[\frac{19}{20},\frac{21}{20}\right],
    \qquad x_0:=\frac25.
\]
Since $h_j^2/\epsilon_j\to0$, the formula for $V_j$
gives
\[
    V_j\longrightarrow2x
    \quad\text{in }C^1([0,x_0]).
\]
In particular, for sufficiently large $j$,
\[
    \|V_j\|_{C^1([0,x_0])}\leq C,
    \qquad
    z-V_j(x)\geq\frac{3}{40},
    \qquad z\in K,\quad 0\leq x\leq x_0,
\]
because $\inf K-2x_0=3/20$.

For $z\in K$, let $f_{\mathrm L,j}(\cdot;z)$ and
$f_{\mathrm R,j}(\cdot;z)$ be the normalized principal
solutions at $x=0$ and $x=L_j$, respectively.
The preceding bounds verify the assumptions of
\cref{lem:endpoint-transfer1} on $(0,x_0)$, with
$\rho=\mu$, kinetic parameter $\eta_j$, and regular
potential $V_j$. Hence
\begin{equation}\label{Eq:glancing-near-mass1}
    \int_0^{x_0}f_{\mathrm L,j}^2\,\dd x
    \leq C\left(
        |f_{\mathrm L,j}(x_0;z)|^2
        +\eta_j^2|f_{\mathrm L,j}'(x_0;z)|^2
    \right).
\end{equation}

For the right principal solution, we use the same
auxiliary-function argument as in the proof of
\cref{lem:transversal-family11}.
Here
\begin{align}
    2V_j+xV_j'-2z
    &=
    \frac{2d_j(1-\epsilon_jx)^{-3}-2}{\epsilon_j}
    -2z \notag\\
    &\geq
    6d_jx-\frac{h_j^2}{2\epsilon_j}-2z
    \geq c_1>0,
    \qquad x_0\leq x<L_j,\quad z\in K,
    \label{Eq:glancing-rellich-sign1}
\end{align}
uniformly for sufficiently large $j$.
The last inequality follows from
$d_j\to1$, $h_j^2/\epsilon_j\to0$, and
$6x_0-2\sup K=3/10$.

Let $F_j$ be the auxiliary function
\eqref{eq:rellich-flux1} with
\[
    h=\eta_j,\qquad g=\eta_j^2a,\qquad
    W=V_j,\qquad E=z,\qquad v=f_{\mathrm R,j}.
\]
Since $x=L_j$ corresponds to $r=0$, the principal
exponent there is $s_j=\frac12+\nu_j>1$ for large $j$.
Its endpoint asymptotics imply that $F_j(x)\to0$
as $x\uparrow L_j$ and that $F_j'$ is integrable there.
Thus
\[
    F_j(x_0)
    =
    \int_{x_0}^{L_j}
        (2V_j+xV_j'-2z)f_{\mathrm R,j}^2\,\dd x
    \geq
    c_1\int_{x_0}^{L_j}f_{\mathrm R,j}^2\,\dd x.
\]
The coefficients of $F_j(x_0)$, expressed in terms
of $f_{\mathrm R,j}(x_0;z)$ and
$\eta_jf_{\mathrm R,j}'(x_0;z)$, are uniformly bounded.
The same Cauchy-data estimate as before therefore gives
\begin{equation}\label{Eq:glancing-far-mass}
    \int_{x_0}^{L_j}f_{\mathrm R,j}^2\,\dd x
    \leq C\left(
        |f_{\mathrm R,j}(x_0;z)|^2
        +\eta_j^2|f_{\mathrm R,j}'(x_0;z)|^2
    \right).
\end{equation}

The principal-solution properties verify the remaining
assumptions of \cref{lem:interior-matching-phase}.
Together with \eqref{Eq:glancing-near-mass1} and
\eqref{Eq:glancing-far-mass}, that lemma yields a gap
of at least $c\eta_j$ between consecutive eigenvalues
in $K$.

Since $1$ and $z_j'$ are consecutive eigenvalues
lying in $K$ for sufficiently large $j$, we obtain
\[
    z_j'-1\geq c\eta_j,
\]
contradicting \eqref{Eq:glancing-gap-ratio1}.

\smallskip
\noindent\emph{\textbf{Step 5:} The case $\eta_j\to\eta_\infty>0$.}

One could establish compact spectral convergence of
$P_j$ to the limiting operator $P_\infty$ within the
framework of \cite{KuwaeShioya2003} and obtain a
contradiction from the stability of spectral
multiplicities in \cite[Theorem~2.6]{KuwaeShioya2003}.
Here we give a more direct argument based on
compactness of the corresponding eigenfunctions to show that two consecutive
eigenvalues approaching one would produce two
independent eigenfunctions at the same limiting energy.

Choose real eigenfunctions $v_j,w_j$ of $P_j$ at
eigenvalues $1,z_j'$, respectively, normalized by
\[
    \|v_j\|_{L^2(0,L_j)}
    =\|w_j\|_{L^2(0,L_j)}=1,
    \qquad
    \langle v_j,w_j\rangle_{L^2(0,L_j)}=0.
\]
For sufficiently large $j$, the inverse-square
coefficients at both endpoints of $P_j$ are positive.
Its Friedrichs form domain is therefore $H_0^1(0,L_j)$,
by the endpoint Hardy inequalities. 

Extend $v_j,w_j$ by zero to $(0,\infty)$, so that
functions defined on the varying intervals $(0,L_j)$
can be compared in a common space.
Since $v_j,w_j\in H_0^1(0,L_j)$, their zero extensions
belong to $H_0^1(0,\infty)$ and preserve both their
$L^2$ norms and the $L^2$ norms of their derivatives.

Using $(1-\epsilon_jx)^{-2}\geq1+2\epsilon_jx$, we have
\begin{equation}\label{Eq:glancing-confinement1}
    V_j(x)
    \geq2d_jx-\beta_j
    \geq x-\beta_j,
    \qquad
    \beta_j:=\frac{h_j^2}{4\epsilon_j}\longrightarrow0.
\end{equation}
Pairing $P_jv_j=v_j$ with $v_j$ in $L^2(0,L_j)$
and using the quadratic form associated with the
Friedrichs realization, together with
$\|v_j\|_{L^2(0,L_j)}=1$, gives
\[
    \eta_j^2\int_0^{L_j}|v_j'|^2\,\dd x
    +\eta_j^2a\int_0^{L_j}\frac{|v_j|^2}{x^2}\,\dd x
    +\int_0^{L_j}V_j|v_j|^2\,\dd x
    =1.
\]
Since $\eta_j$ is bounded above and below by positive
constants, \eqref{Eq:glancing-confinement1} implies
\[
    \int_0^\infty
    \left(
        |v_j'|^2+\frac{a}{x^2}|v_j|^2+x|v_j|^2
    \right)\dd x\leq C.
\]
The same estimate holds for $w_j$, since $z_j'\to1$.

By the compact embedding
$H^1(0,R)\hookrightarrow L^2(0,R)$ for every finite
$R>0$; see \cite[Section~5.7, Theorem~1 and the
subsequent remark]{Evans2010}, a diagonal subsequence
argument yields
\[
    v_j\to v,\qquad w_j\to w
    \quad\text{strongly in }L^2(0,R)
\]
for every finite $R>0$, and weakly in
$H_0^1(0,\infty)$.
Moreover,
\[
    \int_R^\infty|v_j|^2\,\dd x\leq\frac CR,
    \qquad
    \int_R^\infty|w_j|^2\,\dd x\leq\frac CR.
\]
These uniform tail estimates imply strong convergence
on the full half-line:
\[
    v_j\to v,\qquad w_j\to w
    \quad\text{in }L^2(0,\infty).
\]
Hence
\begin{equation}\label{Eq:glancing-limit-orthogonality1}
    \|v\|_{L^2}=\|w\|_{L^2}=1,
    \qquad \langle v,w\rangle_{L^2}=0.
\end{equation}

We identify the limiting operator as
\[
    P_\infty
    :=
    -\eta_\infty^2\partial_x^2
    +\frac{\eta_\infty^2a}{x^2}+2x
\]
with its Friedrichs realization on $(0,\infty)$.
Its form domain is
\[
    \mathcal Q_\infty
    :=
    \left\{
        f\in H_0^1(0,\infty):
        \int_0^\infty
        \left(\frac{a}{x^2}+x\right)|f|^2\,\dd x<\infty
    \right\}.
\]
The uniform bounds and lower semicontinuity imply
$v,w\in\mathcal Q_\infty$.

For $\varphi\in C_c^\infty(0,\infty)$, the eigenvalue
equation for $v_j$ gives, for sufficiently large $j$,
\[
    \eta_j^2\int v_j'\varphi'\,\dd x
    +\eta_j^2a\int x^{-2}v_j\varphi\,\dd x
    +\int V_jv_j\varphi\,\dd x
    =\int v_j\varphi\,\dd x.
\]
On the support of $\varphi$, $V_j\to2x$ uniformly.
Passing to the limit yields
\[
    \eta_\infty^2\int v'\varphi'\,\dd x
    +\eta_\infty^2a\int x^{-2}v\varphi\,\dd x
    +2\int xv\varphi\,\dd x
    =\int v\varphi\,\dd x.
\]
The same identity holds for $w$, because $z_j'\to1$.
Since $C_c^\infty(0,\infty)$ is a core for the
Friedrichs form, these identities extend to all
test functions in $\mathcal Q_\infty$.
The representation theorem for closed forms therefore gives
\begin{equation}\label{Eq:airy-bessel-limit1}
    v,w\in D(P_\infty),
    \qquad P_\infty v=v,\qquad P_\infty w=w.
\end{equation}

Finally, the Friedrichs condition at zero selects a
one-dimensional solution space for $P_\infty f=f$:
every such solution is a multiple of the principal
solution
\[
    f(x)=x^{1/2+\mu}(1+O(x^2)).
\]
Thus two Friedrichs eigenfunctions at the same energy
are proportional near zero, and uniqueness for the
interior ordinary differential equation extends this
proportionality to the whole half-line.
The eigenvalue $1$ therefore has multiplicity at most one.
This contradicts
\eqref{Eq:glancing-limit-orthogonality1} and
\eqref{Eq:airy-bessel-limit1}.

Both possibilities for $\eta_\infty$ have been excluded.
Hence no sequence of consecutive frequencies can have
gap tending to zero, and \eqref{eq:radial-gap1} follows.
\end{proof}

For completeness, the signed sets also have a uniform gap
across zero. By \eqref{eq:gap-bounded-frequency-minmax1},
$\omega_{\nu,1}\geq\pi$ when $\nu\geq1/2$.
For $\nu=0$, comparison with the critical Bessel
operator gives $\omega_{0,1}\geq j_{0,1}$.
Thus, with $\omega_*:=\min\{\pi,j_{0,1}\}>0$,
\[
  \operatorname{gap}(\Lambda_\nu)
  \geq\min\{c_\mu,2\omega_*\}>0,
  \qquad \nu\in\mathcal V.
\]
Together with \eqref{eq:quant-count}, this proves
\cref{prop:spectral-core}.

\appendix

\section{The observation domain and the lower time obstruction}
\label{sec:observation-domain}

In this section, we justify the boundary trace construction used above and
characterize its maximal $L^2$ observation domain. No bounded observation map on the entire energy space is assumed. For a general energy solution, the weighted Neumann trace is initially defined as a distribution. We characterize the initial states for which this trace belongs to $L^2$
on the observation interval.
Throughout this section, put
\[
    p=\mu+\frac12=1-\kappa, \qquad\alpha=\frac{n-1}{2},
    \qquad
    \mathcal V=\{0\}\cup[1/2,\infty).
\]
Recall that the radial eigenfunctions are normalized in
$L^2(0,1)$ and satisfy
\[
    \phi_{\nu,k}(1-y)
    =
    c_{\nu,k}y^p(1+O(y^2)),
    \qquad
    \gamma_{\nu,k}=-2\mu c_{\nu,k}.
\]
The arguments below apply to every fixed $\kappa<0$, including $\kappa\leq-1/2$. Before constructing the weak Neumann trace and
characterizing the maximal observation domain, we establish polynomial bounds for the boundary coefficients and the weighted interior traces.

\begin{lemma}[Polynomial bounds for the radial traces]
\label{lem:norming-upper}
\label{lem:extended-polynomial-trace}
Fix $\mu>1/2$. Uniformly for $\nu\in\mathcal V$ and $k\geq1$,
\begin{equation}\label{eq:norming-upper}
    |\gamma_{\nu,k}|^2
    \leq C_\mu\omega_{\nu,k}^{2\mu+2}.
\end{equation}
Moreover, writing $y=1-r$ and
\[
    M_{\nu,k}(r)
    :=y^{2\kappa}\frac{d}{dr}
       \left(y^{-\kappa}r^{-\alpha}\phi_{\nu,k}(r)\right),
\]
we have
\begin{equation}\label{eq:extended-modal-trace-bound}
    \sup_{1/2\leq r<1}|M_{\nu,k}(r)|
    \leq C_{n,\mu}\omega_{\nu,k}^{\mu+1},
\end{equation}
and
\begin{equation}\label{eq:modal-trace-limit}
    \lim_{r\uparrow1}M_{\nu,k}(r)=\gamma_{\nu,k}.
\end{equation}
\end{lemma}
\begin{proof}
Abbreviate $\omega=\omega_{\nu,k}$, $b=\nu^2-\frac14$,
$c=c_{\nu,k}$, and $\phi=\phi_{\nu,k}$.
The frequency lower bounds give (see \cref{eq:gap-bounded-frequency-minmax1})
\[
    -\frac14\leq b\leq\omega^2,
    \qquad
    \omega\geq\omega_*:=\min\{\pi,j_{0,1}\}>0,
    \qquad |b|\leq C\omega^2.
\]
Under $s=\omega(1-r)$, the equation becomes
\[
    -f''+\frac{a}{s^2}f+W_{\nu,\omega}(s)f=f,
    \qquad
    W_{\nu,\omega}(s)
    :=\frac{b/\omega^2}{(1-s/\omega)^2}.
\]
For fixed $d_*<\omega_*/2$, the potentials
$W_{\nu,\omega}$ are uniformly bounded on $[0,d_*]$.
The normalized principal solution $F=F_{\nu,k}$ therefore
satisfies, uniformly in $\nu,k$,
\[
    F(s)=s^p(1+O_\mu(s^2)),
    \qquad
    F'(s)=ps^{p-1}+O_\mu(s^{p+1}).
\]
Choose a fixed $s_0\in(0,d_*]$ such that
$|F(s)|\geq s^p/2$ for $0<s\leq s_0$.
By principal-solution uniqueness,
$\phi(1-s/\omega)=c\omega^{-p}F(s)$.
Normalization then gives
\[
    1\geq
    \int_{1-s_0/\omega}^{1}|\phi(r)|^2\,\dd r
    =
    |c|^2\omega^{-2p-1}\int_0^{s_0}|F(s)|^2\,\dd s
    \geq
    \frac{s_0^{2p+1}}{4(2p+1)}
    |c|^2\omega^{-2p-1}.
\]
Thus $|c|\leq C_\mu\omega^{p+1/2}
=C_\mu\omega^{\mu+1}$, proving
\eqref{eq:norming-upper} since $\gamma_{\nu,k}=-2\mu c$.

To estimate the interior trace, expand
\[
    M_{\nu,k}(r)
    =
    r^{-\alpha}
    \left(
        y^\kappa\phi'
        +\kappa y^{\kappa-1}\phi
        -\frac{\alpha}{r}y^\kappa\phi
    \right).
\]
If $s=\omega y\leq s_0$, the rescaled representation
and $p+\kappa=1$ yield
\[
    M_{\nu,k}(r)
    =
    cr^{-\alpha}
    \left(
        -s^\kappa F'(s)
        +\kappa s^{\kappa-1}F(s)
        -\frac{\alpha}{r\omega}s^\kappa F(s)
    \right).
\]
The principal estimates imply
$|M_{\nu,k}(r)|\leq C_{n,\mu}|c|$ and, for each fixed mode,
\[
    \lim_{r\uparrow1}M_{\nu,k}(r)
    =(\kappa-p)c=-2\mu c=\gamma_{\nu,k}.
\]
This proves \eqref{eq:modal-trace-limit} and the required
bound in the endpoint region.

For $\omega y\geq s_0$ and $r\geq1/2$, fix
$0<\delta<\min\{s_0/4,\omega_*/8\}$ and set
$I_r=(r-\delta/\omega,r+\delta/\omega)$.
Then $I_r\subset(1/4,1)$ and
$1-t\geq(s_0-\delta)/\omega$ for $t\in I_r$, so
\[
    \sup_{t\in I_r}
    \left|
        \frac{b}{t^2}+\frac{a}{(1-t)^2}-\omega^2
    \right|
    \leq C_\mu\omega^2.
\]
Rescaling to a fixed interval and applying the interior
ODE estimate gives
\[
    |\phi(r)|^2+\omega^{-2}|\phi'(r)|^2
    \leq C_\mu\omega\int_{I_r}|\phi(t)|^2\,\dd t
    \leq C_\mu\omega.
\]
Since $\kappa<0$ and $y\geq s_0/\omega$,
\[
    y^\kappa\leq C_\mu\omega^{-\kappa},
    \qquad
    y^{\kappa-1}\leq C_\mu\omega^{1-\kappa}.
\]
Substituting these bounds into the expression for
$M_{\nu,k}$, and using $\omega\geq\omega_*$, yields
\[
    |M_{\nu,k}(r)|
    \leq C_{n,\mu}\omega^{3/2-\kappa}
    =C_{n,\mu}\omega^{\mu+1}.
\]
Together with the endpoint-region estimate, this proves
\eqref{eq:extended-modal-trace-bound}.
\end{proof}
\subsection{The distributional Neumann trace}

We now construct the weak trace for every energy
solution, without assuming that it belongs to $L^2$.
For a finite spectral-sum solution $v$, define
\[
    (\N_{\kappa,r}^{\,0}v)(t,\vartheta)
    :=
    (1-r)^{2\kappa}\partial_r
    \bigl((1-r)^{-\kappa}v(t,r,\vartheta)\bigr),
    \qquad \frac12<r<1,
\]
where $v$ is written in polar coordinates.

\begin{proposition}[Distributional Neumann trace]
\label{prop:extended-weak-trace}
For each $r\in(1/2,1)$, the map
\[
    (v(0),\partial_tv(0))
    \longmapsto \N_{\kappa,r}^{\,0}v,
\]
initially defined on finite spectral initial states,
extends uniquely to a continuous linear map from
$\mathcal H_\kappa$ into
$\mathcal D'(\mathbb R\times\Sph^{n-1})$.
For an energy solution $u$, denote the value of this
extension at $(u(0),\partial_tu(0))$ by $\N_{\kappa,r}u$.
If $u$ is $C^1$ in an interior collar, this extended trace
agrees on every sphere with the classical weighted radial derivative.

The limit
\begin{equation}\label{eq:weak-neumann-trace}
    \N_\kappa u
    :=
    \lim_{r\uparrow1}\N_{\kappa,r}u
\end{equation}
exists in
$\mathcal D'(\mathbb R\times\Sph^{n-1})$
and is given by
\begin{equation}\label{eq:extended-trace-series}
    \N_\kappa u
    =
    \sum_{\ell,m,k}\gamma_{\nu_\ell,k}
    \left(
        a^+_{\ell,m,k}e^{i\omega_{\nu_\ell,k}t}
        +
        a^-_{\ell,m,k}e^{-i\omega_{\nu_\ell,k}t}
    \right)Y_{\ell,m},
\end{equation}
where $a^\pm_{\ell,m,k}$ are the modal coefficients
of $u$ defined by \eqref{eq:apm-from-initial-data},
and the series converges in distributions.

The map from initial data to the boundary trace is
continuous from $\mathcal H_\kappa$ into
$\mathcal D'(J\times\Sph^{n-1})$
for every bounded open time interval $J$.
Moreover, for every $\rho\in C_c^\infty(\mathbb R)$,
\[
    \N_\kappa(\rho*_t u)
    =
    \rho*_t(\N_\kappa u)
    \quad\text{in }
    \mathcal D'(\mathbb R\times\Sph^{n-1}).
\]
\end{proposition}
\begin{proof}
We first establish a uniform estimate for the classical
interior traces of finite spectral-sum solutions.
We then construct their continuous extensions, prove
the existence of the boundary limit, and verify the
commutation property.

\smallskip
\noindent\emph{Step 1: A uniform estimate for finite spectral sums.}

Let $v$ be a finite spectral-sum solution, with modal
coefficients $b^\pm_{\ell,m,k}$. By the definition of
$M_{\nu,k}(r)$,
\[
    \N_{\kappa,r}^{\,0}v
    =
    \sum_{\ell,m,k}M_{\nu_\ell,k}(r)
    \left(
        b^+_{\ell,m,k}e^{i\omega_{\nu_\ell,k}t}
        +
        b^-_{\ell,m,k}e^{-i\omega_{\nu_\ell,k}t}
    \right)Y_{\ell,m}.
\]
All sums involving $v$ are finite.

Fix a bounded open time interval $J$ and
$\psi\in C_c^\infty(J\times\Sph^{n-1})$.
Write
\[
    \psi_{\ell,m}(t)
    :=
    \int_{\Sph^{n-1}}
    \psi(t,\vartheta)\overline{Y_{\ell,m}(\vartheta)}
    \,\dd\vartheta.
\]
We write distributional pairings with complex conjugation
in the test function. Angular orthogonality gives
\[
    \langle\N_{\kappa,r}^{\,0}v,\psi\rangle
    =
    \sum_{\ell,m,k,\pm}
    M_{\nu_\ell,k}(r)b^\pm_{\ell,m,k}
    \int_{\mathbb R}
    e^{\pm i\omega_{\nu_\ell,k}t}
    \overline{\psi_{\ell,m}(t)}\,\dd t.
\]
For every integer $N\geq0$, integration by parts in
time and Cauchy--Schwarz yield
\[
    \left|
        \int_{\mathbb R}
        e^{\pm i\omega t}
        \overline{\psi_{\ell,m}(t)}\,\dd t
    \right|
    \leq
    |J|^{1/2}\omega^{-N}
    \|\partial_t^N\psi_{\ell,m}\|_{L^2(J)}.
\]

We shall also use the uniform frequency bound
\[
    \omega_{\nu,k}\geq c_\mu k,
    \qquad \nu\in\mathcal V,\quad k\geq1.
\]
For $\nu\geq1/2$, this follows from
\eqref{eq:gap-bounded-frequency-minmax1}.
For $\nu=0$, comparison with the critical Bessel
operator gives $\omega_{0,k}\geq j_{0,k}$,
and the fixed-order Bessel zero asymptotics imply
the required bound.

Choose an integer $N>\mu+1/2$. Since $2\mu-2N<-1$,
\[
    \sup_{\nu\in\mathcal V}
    \sum_{k\geq1}\omega_{\nu,k}^{2\mu-2N}
    \leq
    C_{\mu,N}\sum_{k\geq1}k^{2\mu-2N}
    <\infty.
\]
Using \eqref{eq:extended-modal-trace-bound} and
Cauchy--Schwarz in the modal indices, we obtain
\begin{align*}
    |\langle\N_{\kappa,r}^{\,0}v,\psi\rangle|
    &\leq
    C_{J,n,\mu}
    \left[
        \sum_{\ell,m,k,\pm}
        \omega_{\nu_\ell,k}^2
        |b^\pm_{\ell,m,k}|^2
    \right]^{1/2}\\
    &\quad\times
    \left[
        \sum_{\ell,m,k}
        \omega_{\nu_\ell,k}^{2\mu-2N}
        \|\partial_t^N\psi_{\ell,m}\|_{L^2(J)}^2
    \right]^{1/2}.
\end{align*}
The factor arising from the two signs is absorbed
into the constant. Moreover,
\begin{align*}
    &\sum_{\ell,m,k}
    \omega_{\nu_\ell,k}^{2\mu-2N}
    \|\partial_t^N\psi_{\ell,m}\|_{L^2(J)}^2\\
    &\qquad\leq
    C_{\mu,N}
    \sum_{\ell,m}
    \|\partial_t^N\psi_{\ell,m}\|_{L^2(J)}^2\\
    &\qquad=
    C_{\mu,N}
    \|\partial_t^N\psi\|_{L^2(J\times\Sph^{n-1})}^2,
\end{align*}
where the last equality is Parseval's identity
on the sphere. Thus
\begin{equation}\label{eq:interior-trace-test-bound}
    |\langle\N_{\kappa,r}^{\,0}v,\psi\rangle|
    \leq
    C_{J,n,\mu,N}
    \E_\kappa[v](0)^{1/2}
    \|\partial_t^N\psi\|_{L^2(J\times\Sph^{n-1})},
\end{equation}
uniformly for $1/2<r<1$.

\smallskip
\noindent\emph{Step 2: Continuous extension to energy solutions.}

Let $u$ be an arbitrary energy solution, with coefficients
$a^\pm_{\ell,m,k}$, and let $u_j$ be its finite spectral
truncations over increasing finite index sets $F_j$
exhausting all modes.
Applying \eqref{eq:interior-trace-test-bound} to
$u_j-u_i$ gives
\[
    \sup_{1/2<r<1}
    \left|
        \left\langle
            \N_{\kappa,r}^{\,0}u_j
            -\N_{\kappa,r}^{\,0}u_i,\psi
        \right\rangle
    \right|
    \leq
    C_{J,n,\mu,N}
    \E_\kappa[u_j-u_i](0)^{1/2}
    \|\partial_t^N\psi\|_{L^2(J\times\Sph^{n-1})}.
\]
Since $u$ has finite energy,
$\E_\kappa[u_j-u_i](0)\to0$ as $i,j\to\infty$.
For each fixed $r$, define
\[
    \langle\N_{\kappa,r}u,\psi\rangle
    :=
    \lim_{j\to\infty}
    \langle\N_{\kappa,r}^{\,0}u_j,\psi\rangle.
\]
The convergence is uniform in $r$ after pairing with
any fixed test function.

Passing to the limit in
\eqref{eq:interior-trace-test-bound}, we obtain
\begin{equation}\label{eq:extended-interior-trace-test-bound}
    |\langle\N_{\kappa,r}u,\psi\rangle|
    \leq
    C_{J,n,\mu,N}
    \E_\kappa[u](0)^{1/2}
    \|\partial_t^N\psi\|_{L^2(J\times\Sph^{n-1})},
\end{equation}
uniformly for $1/2<r<1$.
This bound ensures continuity in the test function,
so $\N_{\kappa,r}u$ is a distribution.

The same estimate applied to differences shows that
the construction is independent of the choice of
finite spectral approximations and is continuous
with respect to the initial data in $\mathcal H_\kappa$.
Linearity follows from the construction.
Since finite spectral initial states are dense in
$\mathcal H_\kappa$, this continuous extension is unique.
For finite spectral solutions, it agrees with
$\N_{\kappa,r}^{\,0}$; hence we may use
$\N_{\kappa,r}$ for both.

We also identify this extension with the usual interior
derivative whenever that derivative can be restricted to
a sphere. Fix a compact radial collar
$[r_-,r_+]\subset(1/2,1)$ and a bounded time interval $J$.
Finite spectral truncations converge to $u$ in
$C(J;H_0^1(\B))$. On this collar, the weights and their
derivatives are bounded, so their weighted radial
derivatives converge in
$L^2(J\times(r_-,r_+)\times\Sph^{n-1})$ to the
distributional radial derivative of $u$ (the radial
Jacobian is comparable to one). For every test function
$\psi(t,\vartheta)$, the corresponding pairings are
continuous functions of $r$ for finite sums and converge
uniformly on the collar by
\eqref{eq:interior-trace-test-bound}. Thus the function
$r\mapsto\langle\N_{\kappa,r}u,\psi\rangle$ is the
continuous representative of the pairing with the
bulk weighted derivative, which initially holds for
almost every $r$. If $u$ is $C^1$ there, the classical
restriction is also continuous in $r$ and the two agree
on every interior sphere. In particular, vanishing of
$u$ in an open collar implies
$\N_{\kappa,r}u=0$ there.

\smallskip
\noindent\emph{Step 3: The boundary limit and its spectral representation.}

For each truncation $u_j$, set
\[
    g_j
    :=
    \sum_{(\ell,m,k)\in F_j}
    \gamma_{\nu_\ell,k}
    \left(
        a^+_{\ell,m,k}e^{i\omega_{\nu_\ell,k}t}
        +
        a^-_{\ell,m,k}e^{-i\omega_{\nu_\ell,k}t}
    \right)Y_{\ell,m}.
\]
Since $M_{\nu,k}(r)\to\gamma_{\nu,k}$ for each fixed
mode, the finite sums satisfy
\[
    \N_{\kappa,r}u_j\longrightarrow g_j
    \quad\text{in }
    \mathcal D'(\mathbb R\times\Sph^{n-1})
    \quad\text{as }r\uparrow1.
\]
Taking this limit in the estimates from Step~1 gives
\[
    |\langle g_j-g_i,\psi\rangle|
    \leq
    C_{J,n,\mu,N}
    \E_\kappa[u_j-u_i](0)^{1/2}
    \|\partial_t^N\psi\|_{L^2(J\times\Sph^{n-1})}.
\]
Thus $\langle g_j,\psi\rangle$ is a Cauchy sequence
for every test function. Define
\[
    \langle g,\psi\rangle
    :=
    \lim_{j\to\infty}\langle g_j,\psi\rangle.
\]
The corresponding uniform bound passes to the limit:
\begin{equation}\label{eq:boundary-trace-test-bound}
    |\langle g,\psi\rangle|
    \leq
    C_{J,n,\mu,N}
    \E_\kappa[u](0)^{1/2}
    \|\partial_t^N\psi\|_{L^2(J\times\Sph^{n-1})}.
\end{equation}
Hence $g$ is a distribution.
Since $g_j$ are the finite partial sums of the series
in \eqref{eq:extended-trace-series}, that series
represents $g$ in the distributional sense.

We now identify $g$ as the boundary limit of
$\N_{\kappa,r}u$. For every test function $\psi$,
\begin{align*}
    |\langle\N_{\kappa,r}u-g,\psi\rangle|
    &\leq
    |\langle\N_{\kappa,r}(u-u_j),\psi\rangle|\\
    &\quad+
    |\langle\N_{\kappa,r}u_j-g_j,\psi\rangle|
    +
    |\langle g_j-g,\psi\rangle|.
\end{align*}
By the interior estimate and the corresponding
estimate for the boundary series, the first and
third terms satisfy
\begin{align*}
    &|\langle\N_{\kappa,r}(u-u_j),\psi\rangle|
    +|\langle g_j-g,\psi\rangle|\\
    &\qquad\leq
    2C_{J,n,\mu,N}
    \E_\kappa[u-u_j](0)^{1/2}
    \|\partial_t^N\psi\|_{L^2(J\times\Sph^{n-1})},
\end{align*}
uniformly for $1/2<r<1$.
For fixed $j$, the middle term tends to zero as
$r\uparrow1$. Consequently,
\[
    \limsup_{r\uparrow1}
    |\langle\N_{\kappa,r}u-g,\psi\rangle|
    \leq
    2C_{J,n,\mu,N}
    \E_\kappa[u-u_j](0)^{1/2}
    \|\partial_t^N\psi\|_{L^2(J\times\Sph^{n-1})}.
\]
Letting $j\to\infty$ proves
\[
    \N_{\kappa,r}u\longrightarrow g
    \quad\text{in }
    \mathcal D'(\mathbb R\times\Sph^{n-1})
    \quad\text{as }r\uparrow1.
\]
Therefore $\N_\kappa u=g$, proving
\eqref{eq:weak-neumann-trace} and
\eqref{eq:extended-trace-series}.

The bound \eqref{eq:boundary-trace-test-bound},
applied to differences of solutions, also proves
the asserted continuity of the boundary trace
with respect to the energy initial data.

\smallskip
\noindent\emph{Step 4: Commutation with time convolution.}

Let $\rho\in C_c^\infty(\mathbb R)$.
For every finite spectral truncation $u_j$, direct
convolution of its finitely many time exponentials gives
\[
    \N_\kappa(\rho*_t u_j)
    =
    \rho*_t(\N_\kappa u_j).
\]
Indeed, convolution multiplies the coefficient of
$e^{\pm i\omega t}$ by $\widehat\rho(\pm\omega)$,
and taking the boundary trace multiplies the same
coefficient by $\gamma_{\nu,k}$.

Time convolution is continuous in the energy topology.
More precisely, energy conservation and the triangle
inequality give
\[
    \left\|
        \left(
            (\rho*_t(u_j-u))(0),
            \partial_t(\rho*_t(u_j-u))(0)
        \right)
    \right\|_{\mathcal H_\kappa}
    \leq
    \|\rho\|_{L^1(\mathbb R)}
    \left\|
        \bigl(u_j(0)-u(0),
              \partial_tu_j(0)-\partial_tu(0)\bigr)
    \right\|_{\mathcal H_\kappa}.
\]
Thus the left-hand side of the finite-sum identity
converges in distributions to
$\N_\kappa(\rho*_t u)$.

On the right-hand side,
$\N_\kappa u_j\to\N_\kappa u$ in distributions,
and convolution with a fixed smooth compactly
supported function is continuous in distributions.
Explicitly, with $\check\rho(t):=\rho(-t)$, our
pairing convention gives
\[
    \langle\rho*_t G,\psi\rangle
    =
    \langle G,\overline{\check\rho}*_t\psi\rangle.
\]
The transformed test function is again smooth
and compactly supported.
Passing to the limit in the finite-sum identity
therefore yields
\[
    \N_\kappa(\rho*_t u)
    =
    \rho*_t(\N_\kappa u)
    \quad\text{in }
    \mathcal D'(\mathbb R\times\Sph^{n-1}).
\]
\end{proof}

Whenever the weighted interior traces converge strongly
in $L^2$ on a boundary time cylinder, their limit agrees
with the trace in \cref{prop:extended-weak-trace}.
Thus this construction is consistent with the classical
weighted trace whenever the classical boundary passage
is justified.

\subsection{The closed and maximal observation domain}

Recall the definitions of $Q$, $\mathcal D_{\rm obs}$,
and $\mathcal U_{\rm obs}$ from the discussion of
spectral representations.
The lower norming estimate and
\eqref{eq:modal-energy} give
\[
    \E_\kappa[u](0)\leq C_\mu Q(u_0,u_1).
\]
Thus $Q^{1/2}$ controls the energy norm of the initial
state.

The upper estimate \eqref{eq:norming-upper}, together
with \eqref{eq:apm-from-initial-data}, yields
\begin{align*}
    Q(u_0,u_1)
    &\leq
    C_\mu\sum_{\ell,m,k}
    \omega_{\nu_\ell,k}^{2\mu+2}
    \bigl(
        |a^+_{\ell,m,k}|^2+
        |a^-_{\ell,m,k}|^2
    \bigr)\\
    &=
    \frac{C_\mu}{2}
    \sum_{\ell,m,k}
    \left(
        \omega_{\nu_\ell,k}^{2\mu+2}|U_{\ell,m,k}|^2
        +
        \omega_{\nu_\ell,k}^{2\mu}|V_{\ell,m,k}|^2
    \right).
\end{align*}
Consequently,
\[
    Q(u_0,u_1)
    \leq C_\mu
    \left(
        \|A_\kappa^{(\mu+1)/2}u_0\|_{L^2(\B)}^2
        +
        \|A_\kappa^{\mu/2}u_1\|_{L^2(\B)}^2
    \right),
\]
and hence
\begin{equation}\label{eq:observation-smooth-domain}
    D\bigl(A_\kappa^{(\mu+1)/2}\bigr)
    \times D\bigl(A_\kappa^{\mu/2}\bigr)
    \subset\mathcal D_{\rm obs}.
\end{equation}
For $1/2<\mu\leq1$, this also implies
\[
    D(A_\kappa)\times D(A_\kappa^{1/2})
    \subset\mathcal D_{\rm obs}.
\]
These are sufficient regularity conditions; no optimality
of the displayed power domains is asserted.

Let $\mathcal I$ denote the set of all admissible indices
$(\ell,m,k)$, and define
\[
    \mathcal H_{\kappa,\mathrm{fin}}
    :=
    \operatorname{span}
    \left\{
        (\Phi_{\ell,m,k},0),\ (0,\Phi_{\ell,m,k}):
        (\ell,m,k)\in\mathcal I
    \right\}.
\]
Here $\operatorname{span}$ means finite linear combinations.
Define the finite-sum boundary observation operator by
\[
    \mathcal O_T^0:\mathcal H_{\kappa,\mathrm{fin}}
    \longrightarrow L^2(\Gamma_T),
    \qquad
    \mathcal O_T^0(u_0,u_1)
    :=\N_\kappa u\big|_{\Gamma_T},
\]
where $u$ is the solution with initial data $(u_0,u_1)$
and $\Gamma_T:=(-T,T)\times\Sph^{n-1}$.
Explicitly,
\[
    \bigl[\mathcal O_T^0(u_0,u_1)\bigr](t,\vartheta)
    =
    \sum_{\ell,m,k}
    \gamma_{\nu_\ell,k}
    \left(
        a^+_{\ell,m,k}e^{i\omega_{\nu_\ell,k}t}
        +
        a^-_{\ell,m,k}e^{-i\omega_{\nu_\ell,k}t}
    \right)Y_{\ell,m}(\vartheta),
\]
where $a^\pm_{\ell,m,k}$ are defined in
\eqref{eq:apm-from-initial-data}, and the sum has only
finitely many nonzero terms.
\begin{proposition}[Closed and maximal $L^2$ observation domain]
\label{prop:observation-domain}
Let $T>1$. There exist constants $c_{T,\mu},B_{T,\mu}>0$
such that, for every $(u_0,u_1)\in\mathcal D_{\rm obs}$,
the distributional Neumann trace constructed in
\cref{prop:extended-weak-trace} belongs to $L^2(\Gamma_T)$
and satisfies
\begin{equation}\label{eq:observation-graph-equivalence}
    c_{T,\mu}Q(u_0,u_1)
    \leq
    \|\N_\kappa u\|_{L^2(\Gamma_T)}^2
    \leq
    B_{T,\mu}Q(u_0,u_1).
\end{equation}

The operator $\mathcal O_T^0$ is closable from
$\mathcal H_\kappa$ into $L^2(\Gamma_T)$.
Its closure $\mathcal O_T$ is densely defined, with
\[
    D(\mathcal O_T)=\mathcal D_{\rm obs},
    \qquad
    \mathcal O_T(u_0,u_1)
    =
    \N_\kappa u\big|_{\Gamma_T}.
\]
Finite spectral sums form a core for its graph norm.

Moreover, for every energy solution,
\[
    \N_\kappa u\big|_{\Gamma_T}\in L^2(\Gamma_T)
    \quad\Longleftrightarrow\quad
    (u_0,u_1)\in\mathcal D_{\rm obs}.
\]
In particular, \eqref{eq:obs} holds for every energy
solution whose weak Neumann trace belongs to
$L^2(\Gamma_T)$.
\end{proposition}

\begin{proof}
We first establish the two-sided estimate
\eqref{eq:observation-graph-equivalence} and identify
the resulting $L^2$ observation with the weak trace.

The common separation bound for the signed frequencies
and the upper estimate in Ingham's theorem give a uniform
upper synthesis estimate on every bounded time interval.
If necessary, apply Ingham's theorem on a larger interval
satisfying its length condition and then restrict the
integral to $(-T,T)$.
For $T>1$, \cref{prop:uniform-ingham} also gives the
uniform lower estimate.

Applying these estimates in each spherical harmonic
channel and using angular orthogonality, we obtain
\[
    c_{T,\mu}Q(z)
    \leq
    \|\mathcal O_T^0z\|_{L^2(\Gamma_T)}^2
    \leq
    B_{T,\mu}Q(z)
\]
for every finite spectral initial state $z$.
Now let $z=(u_0,u_1)\in\mathcal D_{\rm obs}$.
Since $Q(z)<\infty$, the upper estimate applied to
differences of finite truncations shows that the series
\[
    \sum_{\ell,m,k}\gamma_{\nu_\ell,k}
    \left(
        a^+_{\ell,m,k}e^{i\omega_{\nu_\ell,k}t}
        +
        a^-_{\ell,m,k}e^{-i\omega_{\nu_\ell,k}t}
    \right)Y_{\ell,m}
\]
converges in $L^2(\Gamma_T)$.
Passing to the limit in the finite-sum estimates gives
\eqref{eq:observation-graph-equivalence}.
By \cref{prop:extended-weak-trace}, this $L^2$ limit,
viewed as a distribution, coincides with
$\N_\kappa u|_{\Gamma_T}$
(since $L^2$ convergence implies distributional convergence
and the distributional limit is unique).
We therefore define
\[
    \mathcal O_Tz:=\N_\kappa u|_{\Gamma_T},
    \qquad z\in\mathcal D_{\rm obs}.
\]
This operator extends $\mathcal O_T^0$.

We next prove that $\mathcal O_T$ is closed.
Suppose that $z_j\in\mathcal D_{\rm obs}$,
\[
    z_j\longrightarrow z
    \quad\text{in }\mathcal H_\kappa,
    \qquad
    \mathcal O_Tz_j\longrightarrow g
    \quad\text{in }L^2(\Gamma_T).
\]
Let $a_j$ denote the corresponding pairs of modal
coefficient sequences, and let $a$ denote the pair
associated with $z$.
The lower estimate in \eqref{eq:observation-graph-equivalence} applied to differences gives
\[
    Q(z_j-z_i)
    \leq c_{T,\mu}^{-1}
    \|\mathcal O_Tz_j-\mathcal O_Tz_i\|_{L^2(\Gamma_T)}^2.
\]
Thus the weighted coefficient vectors $\gamma a_j$
are Cauchy in $\ell^2$.
Energy convergence implies coordinatewise convergence
$a_j\to a$, so their weighted $\ell^2$ limit is
necessarily $\gamma a$.
Consequently,
\[
    Q(z)<\infty,
    \qquad
    Q(z_j-z)\longrightarrow0.
\]
The upper estimate now yields
$\mathcal O_Tz_j\to\mathcal O_Tz$ in $L^2(\Gamma_T)$.
Hence $\mathcal O_Tz=g$, proving closedness.

We now identify $\mathcal O_T$ as the closure of
$\mathcal O_T^0$ and verify the graph-core assertion.
For every $z\in\mathcal D_{\rm obs}$, its finite spectral
truncations $z^{(N)}$ satisfy
\[
    z^{(N)}\longrightarrow z
    \quad\text{in }\mathcal H_\kappa,
    \qquad
    Q(z^{(N)}-z)\longrightarrow0.
\]
By the upper estimate,
\[
    \mathcal O_T^0z^{(N)}
    \longrightarrow\mathcal O_Tz
    \quad\text{in }L^2(\Gamma_T).
\]
Thus finite spectral sums are dense in
$D(\mathcal O_T)$ for the graph norm
\[
    \|z\|_{\mathrm{graph}}
    :=
    \left(
        \|z\|_{\mathcal H_\kappa}^2
        +
        \|\mathcal O_Tz\|_{L^2(\Gamma_T)}^2
    \right)^{1/2}.
\]
Since $\mathcal O_T$ is a closed extension of
$\mathcal O_T^0$, this proves that $\mathcal O_T^0$
is closable and that its closure is exactly
$\mathcal O_T$.
Its domain is dense in $\mathcal H_\kappa$ because it
contains all finite spectral initial states.

It remains to prove maximality: every energy solution
whose weak trace belongs to $L^2(\Gamma_T)$ must have
initial data in $\mathcal D_{\rm obs}$.
The reverse implication was established above.

Let $u$ be an energy solution with
\[
    g:=\N_\kappa u|_{\Gamma_T}\in L^2(\Gamma_T).
\]
Choose $1<T_1<T$, and let
\[
    \rho_\varepsilon(t)
    =
    \varepsilon^{-1}\rho(t/\varepsilon),
\]
where $\rho\in C_c^\infty(-1,1)$ is even, nonnegative,
and has integral one.
For $0<\varepsilon<T-T_1$, set
\[
    u_\varepsilon=\rho_\varepsilon*_t u.
\]
Its modal coefficients are
\[
    a^\pm_{\ell,m,k}
    \widehat\rho(\pm\varepsilon\omega_{\nu_\ell,k}),
    \qquad
    \widehat\rho(\xi)
    :=
    \int_{\mathbb R}\rho(t)e^{-i\xi t}\,\dd t.
\]
The polynomial bound \eqref{eq:norming-upper} gives
\begin{align*}
    Q\bigl(u_\varepsilon(0),
           \partial_tu_\varepsilon(0)\bigr)
    &\leq
    C_\mu
    \sup_{s\geq\omega_*}
    \left(s^{2\mu}|\widehat\rho(\varepsilon s)|^2\right)
    \E_\kappa[u](0)
    <\infty,
\end{align*}
because $\widehat\rho$ decays faster than every power.
Thus the initial data of $u_\varepsilon$ belong to
$\mathcal D_{\rm obs}$.

By \cref{prop:extended-weak-trace}, the weak trace
commutes with time convolution. Therefore
\[
    \N_\kappa u_\varepsilon
    =
    \rho_\varepsilon*g
    \quad\text{on }\Gamma_{T_1}.
\]
Since $\varepsilon<T-T_1$, only values of $g$ in
$\Gamma_T$ enter this convolution.
Extending $g$ by zero outside $(-T,T)$ and applying
Young's inequality, we obtain
\[
    \|\N_\kappa u_\varepsilon\|_{L^2(\Gamma_{T_1})}
    \leq \|g\|_{L^2(\Gamma_T)}.
\]
The lower estimate in
\eqref{eq:observation-graph-equivalence},
applied to $u_\varepsilon$ on $(-T_1,T_1)$, gives
\[
    Q\bigl(u_\varepsilon(0),
           \partial_tu_\varepsilon(0)\bigr)
    \leq
    c_{T_1,\mu}^{-1}
    \|\N_\kappa u_\varepsilon\|_{L^2(\Gamma_{T_1})}^2
    \leq
    c_{T_1,\mu}^{-1}\|g\|_{L^2(\Gamma_T)}^2.
\]
Choose a sequence $\varepsilon_j\downarrow0$ with
$\varepsilon_j<T-T_1$.
Since $\widehat\rho$ is continuous and
$\widehat\rho(0)=1$, for each fixed modal index we have
\[
    |\widehat\rho(
        \pm\varepsilon_j\omega_{\nu_\ell,k})|^2
    \longrightarrow1.
\]
Applying Fatou's lemma with respect to counting measure
on the modal index set, we obtain
\begin{align*}
    Q(u_0,u_1)
    &=
    \sum_{\ell,m,k,\pm}
    \liminf_{j\to\infty}
    \left[
        |\gamma_{\nu_\ell,k}a^\pm_{\ell,m,k}|^2
        |\widehat\rho(
            \pm\varepsilon_j\omega_{\nu_\ell,k})|^2
    \right]\\
    &\leq
    \liminf_{j\to\infty}
    \sum_{\ell,m,k,\pm}
    |\gamma_{\nu_\ell,k}a^\pm_{\ell,m,k}|^2
    |\widehat\rho(
        \pm\varepsilon_j\omega_{\nu_\ell,k})|^2\\
    &\leq
    c_{T_1,\mu}^{-1}\|g\|_{L^2(\Gamma_T)}^2
    <\infty.
\end{align*}
Thus $(u_0,u_1)\in\mathcal D_{\rm obs}$.
The trace identification proved above then gives
$\mathcal O_T(u_0,u_1)=g$.
This proves the asserted characterization.

\end{proof}
By \cref{lem:norming-lower1}, the identity
$\gamma_{\nu,k}=-2\mu c_{\nu,k}$, and
\eqref{eq:modal-energy},
\[
    \|z\|_{\mathcal H_\kappa}^2\leq C_\mu Q(z).
\]
Combining this with
\eqref{eq:observation-graph-equivalence}, we obtain
\[
    \|z\|_{\mathcal H_\kappa}^2
    +\|\mathcal O_Tz\|_{L^2(\Gamma_T)}^2
    \asymp Q(z),
    \qquad z\in\mathcal D_{\rm obs},\quad T>1.
\]
Thus $Q^{1/2}$ is equivalent to the graph norm of
$\mathcal O_T$.

\subsection{The lower time obstruction}\label{prop:short-time-failure}
With the polynomial trace bounds and the characterization
of the observation domain established, we can now prove
the failure of observability for $T<1$.
\begin{proposition}[Failure of observability for $T<1$]
\label{prop:lower-time-obstruction}
Fix $\kappa<0$ and $0<T<1$.
There is a nonzero energy solution
$u\in\mathcal U_{\rm obs}$ such that
\[
    \N_\kappa u=0
    \quad\text{on }(-T,T)\times\Sph^{n-1}.
\]
Consequently, no constant $C>0$ can make
\[
    \E_\kappa[u](0)
    \leq
    C\int_{-T}^T\int_{\Sph^{n-1}}
       |\N_\kappa u|^2\,d\vartheta\,dt
\]
hold for all $u\in\mathcal U_{\rm obs}$.
In fact, no such estimate holds uniformly over finite
spectral-sum solutions.
\end{proposition}

\begin{proof}
Choose $0<\rho<1-T$ and a nonzero function
\[
    f\in
    C_c^\infty\bigl(\{x:\rho/2<|x|<\rho\}\bigr).
\]
Let $u$ be the Friedrichs energy solution with
\[
    u(0)=0,\qquad \partial_tu(0)=f.
\]
The potential is smooth on a neighborhood of
$\operatorname{supp}f$.
Repeated application of its differential expression
therefore gives compactly supported smooth functions
with support in the same interior region.
Hence
\[
    f\in\bigcap_{m\geq1}D(A_\kappa^m).
\]
By \eqref{eq:observation-smooth-domain},
$(0,f)\in\mathcal D_{\rm obs}$, so
$u\in\mathcal U_{\rm obs}$.
By the finite propagation property for wave equations
with bounded real potentials
\cite[Lemma~2.2(a)]{Remling2007}, applied also to the
initial-velocity propagator through
$\sin(t\sqrt H)/\sqrt H=\int_0^t\cos(s\sqrt H)\,ds$,
we have
\[
    \operatorname{supp}u(t,\cdot)
    \subset
    \{x:|x|\leq\rho+|t|\},
    \qquad |t|\leq T.
\]
For completeness, choose
\[
    \rho+T<R<1
\]
and extend $a/(1-|x|)^2$ from $B(0,R)$ to a bounded
potential on $\mathbb R^n$.
The whole-space solution with the same initial data
has the displayed support property.
For $|t|\leq T$, it is supported strictly inside
$B(0,R)$ and therefore also solves the original
Friedrichs problem on $\B$.
Uniqueness identifies it with $u$.

Since $\rho+T<1$, the solution vanishes in a fixed
neighborhood of the boundary throughout $[-T,T]$.
Thus its weighted Neumann trace is zero on $\Gamma_T$.
On the other hand,
\[
    \E_\kappa[u](0)
    =
    \frac12\|f\|_{L^2(\B)}^2>0.
\]
The observability inequality therefore fails.

To obtain failure already for finite spectral sums,
let $u_j$ be finite spectral truncations of $u$. Then
$\E_\kappa[u_j](0)\to\E_\kappa[u](0)>0$ and
$Q(u_j-u)\to0$. The upper synthesis estimate on the
bounded interval $(-T,T)$ (which holds also for $T<1$)
and $\N_\kappa u=0$ there give
\[
    \|\N_\kappa u_j\|_{L^2(\Gamma_T)}^2
    \leq B_{T,\mu}Q(u_j-u)\longrightarrow0.
\]
This rules out an inequality with one constant valid
for all finite spectral-sum solutions.
\end{proof}

Together with the estimate for $T>1$, this shows that
the half-timespan threshold is $1$, corresponding to
a total observation length of $2$.
The critical case $T=1$ is not decided by these arguments.

\subsection{An independent lower bound for the observation coefficients}\label{Section6}

In this appendix, we use
\cref{lem:endpoint-transfer1,lem:rellich11}
and the compactness argument in the proof of
\cref{lem:uniform-radial-gap1}
to establish the lower norming estimate.
Only the Friedrichs principal solution is needed,
including when $\mu\geq1$ and the nonprincipal
solution is not locally square integrable. Fix $\mu>1/2$, put
\[
   p=\mu+\frac12=1-\kappa>1,
   \qquad a=\mu^2-\frac14,
   \qquad \mathcal V=\{0\}\cup[1/2,\infty).
\]

\begin{lemma}[Uniform lower norming estimate]
\label{lem:norming-lower1}
For every fixed $\mu>1/2$, there exists $c_\mu>0$ such that
\begin{equation}\label{eq:independent-norming-lower}
    |c_{\nu,k}|^2\geq c_\mu\omega_{\nu,k}^2,
    \qquad \nu\in\mathcal V,\quad k\geq1.
\end{equation}
\end{lemma}

\begin{proof}
Suppose otherwise. Choose $\nu_j\in\mathcal V$ and
$k_j\geq1$, and write
\[
    (\omega_j,\phi_j,c_j)
    :=
    (\omega_{\nu_j,k_j},\phi_{\nu_j,k_j},c_{\nu_j,k_j}),
    \qquad
    \frac{|c_j|^2}{\omega_j^2}\longrightarrow0.
\]
Here $\|\phi_j\|_{L^2(0,1)}=1$ and
$\phi_j(1-y)=c_jy^p(1+O(y^2))$.

As in \cref{lem:uniform-radial-gap1}, bounded frequencies
restrict $\nu_j\geq1/2$ to a compact set and $k_j$ to
finitely many values. On this set, $|c_{\nu,k}|$ is
continuous and nonzero: it is recovered by comparison
with the normalized principal solution at a fixed
interior point near $r=1$.
The isolated sector $\nu=0$ likewise contributes only
finitely many modes at bounded frequencies.
Thus $\omega_j\to\infty$.

We use the same division into regimes as in the gap
proof. If $\nu_j/\omega_j\leq\tau_*<1$ along a subsequence,
\eqref{eq:transversal-principal-lower}, with
$\widehat h_j=\omega_j^{-1}$, gives
\[
    \frac{|c_j|^2}{\omega_j^2}
    \geq C^{-1}\omega_j^{2p-2}\longrightarrow\infty,
\]
a contradiction.

It remains to consider $\nu_j/\omega_j\to1$.
Use the glancing rescaling from
\cref{lem:uniform-radial-gap1}:
\[
    h_j=\nu_j^{-1},\qquad
    \epsilon_j=h_j^2\omega_j^2-1,\qquad
    \eta_j=h_j\epsilon_j^{-3/2},\qquad
    G_j(x)=\sqrt{\epsilon_j}\phi_j(1-\epsilon_jx).
\]
Then $\epsilon_j>0$, $\epsilon_j\to0$, and $\eta_j$
is bounded. Moreover, $G_j$ is a normalized eigenfunction
of the rescaled operator $P_j$ at energy $1$.
Its principal coefficient satisfies
\begin{equation}\label{eq:rescaled-principal-coefficient}
    A_j=\epsilon_j^{p+1/2}c_j,
    \qquad
    \frac{|c_j|^2}{\omega_j^2}
    =
    \frac{\eta_j^2|A_j|^2}{1+\epsilon_j}
    \epsilon_j^{2-2p}.
\end{equation}
Passing to a subsequence, assume
$\eta_j\to\eta_\infty\geq0$.

If $\eta_j\to0$, the same argument as in Step~4
of the proof of \cref{lem:uniform-radial-gap1},
combining the two mass estimates with
\eqref{eq:endpoint-cauchy1}, gives
\[
    1\leq C\eta_j^{2p}|A_j|^2.
\]
Consequently,
\[
    \frac{|c_j|^2}{\omega_j^2}
    \geq
    C^{-1}
    \frac{(\eta_j\epsilon_j)^{2-2p}}{1+\epsilon_j}
    \longrightarrow\infty,
\]
since $p>1$.

If $\eta_\infty>0$, Step~5 of the proof of
\cref{lem:uniform-radial-gap1} gives, after zero extension
and passage to a subsequence,
\[
    G_j\longrightarrow G_\infty
    \quad\text{in }L^2(0,\infty)
    \text{ and }C^1_{\mathrm{loc}}(0,\infty),
\]
where $G_\infty$ is a normalized Friedrichs eigenfunction
at energy $1$ of
\[
    -\eta_\infty^2\partial_x^2
    +\frac{\eta_\infty^2a}{x^2}+2x.
\]
Let $f_j,f_\infty$ be the corresponding principal
solutions normalized by $x^{-p}f(x)\to1$.
Since the regular coefficients
$\eta_j^{-2}(V_j-1)$ converge uniformly near zero,
continuous dependence of the principal solutions gives
$f_j(x_*)\to f_\infty(x_*)\neq0$ at a sufficiently
small fixed $x_*>0$. Hence
\[
    A_j=\frac{G_j(x_*)}{f_j(x_*)}
    \longrightarrow
    \frac{G_\infty(x_*)}{f_\infty(x_*)}
    =A_\infty\neq0.
\]
Indeed, a nonzero Friedrichs solution cannot have
zero principal coefficient.
Since $p>1$, \eqref{eq:rescaled-principal-coefficient}
again yields
\[
    \frac{|c_j|^2}{\omega_j^2}
    =
    \frac{\eta_j^2|A_j|^2}{1+\epsilon_j}
    \epsilon_j^{2-2p}
    \longrightarrow\infty.
\]
Both cases contradict the assumed vanishing ratio.
\end{proof}

This proof uses only the preceding endpoint estimates
and compactness arguments, without invoking any
previously established boundary observability inequality.
\section{Energy equivalence and comparison of the solution classes}
\label{app:energy-and-solution-classes}

We retain the notation and the Friedrichs realization
$A_\kappa$ introduced in the main text.
Throughout this appendix, we assume $n\geq3$ and
$-\frac12<\kappa<0$, and compare our energy and
observation class with those of \cite{ESV2021}
for the unperturbed equation $X=V=0$.
\subsection{Equivalence of the energies}\label{app:energy-equivalence}

Recall the energy defined in \cref{eq:energy}:
\begin{equation}\label{eq:app-form-energy}
    \mathcal E_\kappa[u](t)
    =
    \frac12\int_B
    \left(
        |\partial_tu|^2
        +|\nabla u|^2
        +\frac{a}{y^2}|u|^2
    \right)\,dx.
\end{equation}
Let
\[
    D_ru
    =
    y^\kappa\partial_r(y^{-\kappa}u)
    =
    \partial_ru+\frac{\kappa}{y}u,
    \qquad
    \nabla_{\mathrm{tan}}u
    =
    r^{-1}\nabla_{\mathbb S^{n-1}}u.
\]
The twisted energy appearing in \cite[equation~(1.3)]{ESV2021}
is
\begin{equation}\label{eq:app-twisted-energy}
    \mathcal E_{\mathrm{tw}}[u](t)
    =
    \int_B
    \left(
        |\partial_tu|^2
        +|D_ru|^2
        +|\nabla_{\mathrm{tan}}u|^2
    \right)\,dx.
\end{equation}
The energy used in
\cite[equation~(2.13) and Theorem~5.1]{ESV2021} is
\begin{equation}\label{eq:app-E1}
    E_{ESV}[u](t)
    =
    \mathcal E_{\mathrm{tw}}[u](t)
    +\|u(t)\|_{L^2(B)}^2.
\end{equation}

\begin{proposition}[Energy comparison]
\label{prop:app-energy-comparison}
For every $(v,w)\in H_0^1(B)\times L^2(B)$,
\begin{equation}\label{eq:app-energy-identity}
    \mathcal E_{\mathrm{tw}}[v,w]
    =
    2\mathcal E_\kappa[v,w]
    -\kappa(n-1)
    \int_B\frac{|v|^2}{ry}\,dx.
\end{equation}
Consequently, there exists $C_{n,\kappa}>0$ such that
\begin{equation}\label{eq:app-energy-equivalence}
    2\mathcal E_\kappa[v,w]
    \leq
    \mathcal E_{\mathrm{tw}}[v,w]
    \leq
    E_{ESV}[v,w]
    \leq
    C_{n,\kappa}\mathcal E_\kappa[v,w].
\end{equation}
In particular, these energies define equivalent norms on
the same Friedrichs energy space.
\end{proposition}

\begin{proof}
Initially let $v\in C_c^\infty(B)$. Expanding the twisted
radial derivative gives
\[
    |D_rv|^2
    =
    |\partial_rv|^2
    +\frac{\kappa^2}{y^2}|v|^2
    +\frac{\kappa}{y}\partial_r|v|^2.
\]
Integration by parts in polar coordinates yields
\[
    \int_B\frac{\kappa}{y}\partial_r|v|^2\,dx
    =
    -\kappa
    \int_B
    \left(
        \frac1{y^2}+\frac{n-1}{ry}
    \right)|v|^2\,dx.
\]
There is no contribution at $r=0$, since $n\geq3$.
Using
\[
    |\nabla v|^2
    =
    |\partial_rv|^2+|\nabla_{\mathrm{tan}}v|^2,
    \qquad
    a=\kappa^2-\kappa,
\]
we obtain \eqref{eq:app-energy-identity}.

The boundary and interior Hardy inequalities give
\[
    \int_B\frac{|v|^2}{y^2}\,dx
    \leq4\int_B|\nabla v|^2\,dx,
    \qquad
    \int_B\frac{|v|^2}{r^2}\,dx
    \leq
    \frac4{(n-2)^2}\int_B|\nabla v|^2\,dx.
\]
Since
\[
    \frac1{ry}
    \leq\frac12\left(\frac1{r^2}+\frac1{y^2}\right),
\]
the additional integral in
\eqref{eq:app-energy-identity} is continuous on $H_0^1(B)$.
Thus the identity extends by density to that space.

Finally, $-\kappa(n-1)>0$, which gives the lower bounds
in \eqref{eq:app-energy-equivalence}. The Hardy inequalities
and the Poincare inequality give the upper bound.
\end{proof}

Thus the use of $\mathcal E_\kappa$ in the spectral argument
does not change the energy topology of the problem.
Once the Neumann traces are identified, an estimate
\[
    \mathcal E_\kappa[u](0)
    \leq C_T\|\mathcal N_\kappa u\|_{L^2(\Gamma_T)}^2
\]
immediately implies
\begin{equation}\label{eq:app-ESV-energy-observation}
    E_{ESV}[u](0)
    \leq C_{T,n,\kappa}
    \|\mathcal N_\kappa u\|_{L^2(\Gamma_T)}^2.
\end{equation}
This is the energy used in
\cite[Theorem~5.1]{ESV2021}.

\subsection{Boundary admissibility and the spectral observation domain}

Write $\beta=1-2\kappa$ and
$\mathcal N_{\kappa,r}u
:=y^{2\kappa}\partial_r(y^{-\kappa}u)$ for a sufficiently regular function. In our notation, the boundary-admissibility conditions
of \cite[Definition~2.2]{ESV2021} require
$u\in C^1(I_T\times B)$ and some $g\in L^2(\Gamma_T)$
such that, strongly in $L^2(\Gamma_T)$ as $r\uparrow1$,
\begin{align}
    \mathcal N_{\kappa,r}u
        &\longrightarrow g,
        \label{eq:app-BA-Neumann}\\
    \beta y^{\kappa-1}u
        &\longrightarrow-g,
        \label{eq:app-BA-principal}\\
    y^\kappa\partial_tu
        &\longrightarrow0.
        \label{eq:app-BA-time}
\end{align}
Then $\mathcal N_\kappa u=g$, and
\eqref{eq:app-BA-principal} also gives
\[
    y^{-\kappa}u
    =y^\beta(y^{\kappa-1}u)\longrightarrow0,
\]
so $\mathcal D_\kappa u=0$.

Let $\mathscr A_T$ denote the class of solutions of the
unperturbed equation on $I_T$ with finite twisted energy
in the sense of \cite{ESV2021}, satisfying the
boundary-admissibility conditions above.

By the energy comparison in
\cref{prop:app-energy-comparison}, together with
\cite[Lemmas~5.2--5.3]{ESV2021}, these solutions satisfy
\[
    u\in L^2_{\mathrm{loc}}(I_T;H_0^1(B)),
    \qquad
    \partial_tu\in L^2_{\mathrm{loc}}(I_T;L^2(B)).
\]
Expanding the weak equation in the eigenbasis of
$A_\kappa$ gives the scalar wave equations.
Local integrability of the energy then implies finite
conserved spectral energy and identifies each
$u\in\mathscr A_T$ with a Friedrichs energy solution.
This identification requires no large-time assumption.

Finally, let $\mathscr S_T$ denote the restrictions to
$I_T$ of Friedrichs energy solutions with
\[
    \bigl(u(0),\partial_tu(0)\bigr)\in\mathcal D_{\rm obs}.
\]

\begin{proposition}[Inclusion and identification of the traces]
\label{prop:app-inclusion}
For every $T>1$,
\[
    \mathscr A_T\subset\mathscr S_T.
\]
Moreover, for every $u\in\mathscr A_T$, the weak Neumann
trace constructed in \cref{prop:extended-weak-trace}
coincides with the strong boundary limit $g$ in
\eqref{eq:app-BA-Neumann}:
\[
    \mathcal N_\kappa u=g
    \quad\text{in }L^2(\Gamma_T).
\]
\end{proposition}

\begin{proof}
Let $u\in\mathscr A_T$, and let $g\in L^2(\Gamma_T)$
be the strong boundary limit in
\eqref{eq:app-BA-Neumann}.
Fix an angular channel $(\ell,m)$, put $\nu=\nu_\ell$,
and set
\[
    U_{\ell,m}(t,r)
    :=
    \int_{\Sph^{n-1}}
        u(t,r,\vartheta)\overline{Y_{\ell,m}(\vartheta)}
        \,\dd\vartheta,
    \qquad
    v(t,r):=r^\alpha U_{\ell,m}(t,r).
\]
Then $\partial_t^2v+H_{\nu,\mu}v=0$. Let
\[
    g_{\ell,m}(t)
    :=
    \int_{\Sph^{n-1}}
        g(t,\vartheta)\overline{Y_{\ell,m}(\vartheta)}
        \,\dd\vartheta.
\]
Our aim is to compute $g_{\ell,m}$ and show that it
coincides with the corresponding spherical harmonic
coefficient of the weak Neumann trace constructed
in \cref{prop:extended-weak-trace}.

Since $\|Y_{\ell,m}\|_{L^2(\Sph^{n-1})}=1$,
angular projection is a bounded map from
$L^2(\Gamma_T)$ to $L^2(I_T)$.
Moreover, the interior $C^1$ regularity of $u$
allows radial differentiation under the spherical
integral. Hence \eqref{eq:app-BA-Neumann} implies
\[
    \mathcal N_{\kappa,r}U_{\ell,m}
    \longrightarrow g_{\ell,m}
    \quad\text{strongly in }L^2(I_T).
\]

Choose a smooth cutoff $\chi$ vanishing near $r=0$
and equal to one near $r=1$, and set
\[
    \eta(r):=\chi(r)(1-r)^\kappa.
\]
Since $\kappa>-1/2$,
\[
    \eta,\ H_{\nu,\mu}\eta\in L^2(0,1),
\]
where $H_{\nu,\mu}\eta$ denotes the formal differential
expression applied to $\eta$.
Indeed, the most singular terms cancel because
$a=\kappa(\kappa-1)$.

The radial boundary expression satisfies
\[
    y^{2\kappa}\partial_r(y^{-\kappa}v)
    =
    r^\alpha\mathcal N_{\kappa,r}U_{\ell,m}
    +\alpha r^{\alpha-1}y^\kappa U_{\ell,m}.
\]
By \eqref{eq:app-BA-principal},
\[
    y^\kappa U_{\ell,m}
    =
    y\bigl(y^{\kappa-1}U_{\ell,m}\bigr)
    \longrightarrow0
    \quad\text{in }L^2(I_T).
\]
Thus the radial boundary expression converges strongly
in $L^2(I_T)$ to the angular coefficient $g_{\ell,m}$ of $g$.

Green's identity on truncated radial intervals,
tested against a compactly supported smooth function
of time, therefore yields
\begin{equation}\label{eq:app-Green-trace}
    g_{\ell,m}
    =
    \partial_t^2\langle v,\eta\rangle
    +\langle v,H_{\nu,\mu}\eta\rangle
    \quad\text{in }\mathcal D'(I_T).
\end{equation}
The volume integrals converge because
$v\in C(I_T;L^2(0,1))$ and
$\eta,H_{\nu,\mu}\eta\in L^2(0,1)$.

Applying the same identity to a radial eigenfunction gives
\begin{equation}\label{eq:app-gamma-Green}
    \gamma_{\nu,k}
    =
    \langle\phi_{\nu,k},H_{\nu,\mu}\eta\rangle
    -\omega_{\nu,k}^2\langle\phi_{\nu,k},\eta\rangle.
\end{equation}
Substituting the spectral expansion of $v$ into
\eqref{eq:app-Green-trace} consequently gives
\begin{equation}\label{eq:app-distributional-trace}
    g_{\ell,m}
    =
    \sum_{k\geq1}\gamma_{\nu,k}
    \left(
        a_{\ell,m,k}^+e^{i\omega_{\nu,k}t}
        +a_{\ell,m,k}^-e^{-i\omega_{\nu,k}t}
    \right)
    \quad\text{in }\mathcal D'(I_T).
\end{equation}
This passage follows first for finite truncations
and then by convergence in $C_{\mathrm{loc}}(I_T;L^2(0,1))$,
with time derivatives interpreted distributionally.

By \cref{prop:extended-weak-trace}, the series in
\eqref{eq:app-distributional-trace} is precisely the
$(\ell,m)$ angular coefficient of the weak trace
$\mathcal N_\kappa u$.
Completeness of the spherical harmonics therefore gives
\[
    \mathcal N_\kappa u=g
    \quad\text{in }\mathcal D'(\Gamma_T).
\]
Since $g\in L^2(\Gamma_T)$, the maximal-domain
characterization in \cref{prop:observation-domain} yields
\[
    \bigl(u(0),\partial_tu(0)\bigr)\in\mathcal D_{\rm obs}.
\]
Hence $u\in\mathscr S_T$, and the weak Neumann trace
coincides with the strong boundary limit in
$L^2(\Gamma_T)$.
\end{proof}
\begin{remark}[Strictness of the inclusion]
\label{rem:app-strict-inclusion}
In fact, $\mathscr A_T\subsetneq\mathscr S_T$ for every
$T>1$. To see the mechanism, choose distinct angular
degrees $\ell_j\to\infty$ sufficiently rapidly that
\[
    \sum_j\varepsilon_j<\infty,
    \qquad
    \varepsilon_j
    :=\left(\nu_{\ell_j}^2-\frac14\right)^{-1/3}.
\]
Let $\Phi_j$ be normalized real eigenfunctions
corresponding to the first radial mode in these
angular sectors, with frequencies $\omega_j$ and
boundary coefficients $\gamma_j$. Consider
\[
    u(t)
    :=
    \sum_j
    \frac{\sqrt{\varepsilon_j}}{\gamma_j}
    \sin(\omega_jt)\Phi_j.
\]
The lower norming estimate ensures finite energy, and
\[
    Q[u]=\frac12\sum_j\varepsilon_j<\infty,
\]
so $u\in\mathscr S_T$.

The first radial modes concentrate in boundary layers
of thickness $\varepsilon_j$. Their rescaled profiles
converge to the ground state of
$-\partial_x^2+a/x^2+2x$ on $(0,\infty)$.
Evaluating the weighted time derivative at
$r_j=1-\varepsilon_j$, this scaling and angular
orthogonality give
\[
    \liminf_{j\to\infty}
    \bigl\|
        \varepsilon_j^\kappa
        \partial_tu(\,\cdot\,,r_j\,\cdot\,)
    \bigr\|_{L^2(\Gamma_T)}^2>0.
\]
Thus the boundary-admissibility condition
\eqref{eq:app-BA-time} fails.
The construction can moreover be verified to have
interior $C^1$ regularity and to satisfy the strong
Neumann and principal boundary limits
\eqref{eq:app-BA-Neumann}--\eqref{eq:app-BA-principal}.
We omit the detailed estimates for brevity.
\end{remark}

\noindent \textbf{Authors' addresses:}

\medskip
\noindent Longben Wei, School of Mathematical Sciences, Guizhou Normal University, Guiyang 550025, P.R. China. \\
Email address: \texttt{longbenwei51@gmail.com}

\begin{thebibliography}{9}
\bibitem{AnantharamanLeautaudMacia2016}
N. Anantharaman, M. L\'eautaud and F. Maci\`a,
Wigner measures and observability for the
Schr\"odinger equation on the disk,
Invent. Math. \textbf{206} (2016), 485--599.

\bibitem{Bachelot2011}
A. Bachelot,
The Klein--Gordon equation in the anti-de Sitter cosmology,
J. Math. Pures Appl. (9) \textbf{96} (2011), 527--554.

\bibitem{BKL2002}
C. Baiocchi, V. Komornik, and P. Loreti,
\emph{Ingham--Beurling type theorems with weakened gap conditions},
Acta Math. Hungar. \textbf{97} (2002), no.~1--2, 55--95.
\bibitem{BarceloRuizVega2006}
J.~A. Barcel\'o, A. Ruiz, and L. Vega,
Some dispersive estimates for Schr\"odinger equations
with repulsive potentials,
J. Funct. Anal. \textbf{236} (2006), 1--24.

\bibitem{Beurling1989}
A.~Beurling:
Interpolation for an interval on $\mathbb R^1$,
in: \emph{The Collected Works of Arne Beurling},
Vol.~2: \emph{Harmonic Analysis},
Birkh\"auser, Boston, 1989.

\bibitem{BiccariZuazua2016}
U. Biccari and E. Zuazua,
Null controllability for a heat equation with a singular
inverse-square potential involving the distance to the
boundary function,
J. Differential Equations \textbf{261} (2016),
no.~5, 2809--2853.

\bibitem{BurqEtAl2004}
N. Burq, F. Planchon, J.~G. Stalker, and
A.~S. Tahvildar-Zadeh,
Strichartz estimates for the wave and Schr\"odinger
equations with potentials of critical decay,
Indiana Univ. Math. J. \textbf{53} (2004), 1665--1680.

\bibitem{Cazacu2012}
C. Cazacu,
Schr\"odinger operators with boundary singularities:
Hardy inequality, Pohozaev identity and controllability
results,
J. Funct. Anal. \textbf{263} (2012), no.~12, 3741--3782.

\bibitem{Cazacu2014}
C. Cazacu,
Controllability of the heat equation with an inverse-square
potential localized on the boundary,
SIAM J. Control Optim. \textbf{52} (2014), no.~4, 2055--2089.

\bibitem{EGNT2013}
J.~Eckhardt, F.~Gesztesy, R.~Nichols, and G.~Teschl,
Weyl--Titchmarsh theory for Sturm--Liouville operators
with distributional potentials,
Opuscula Math. \textbf{33} (2013), 467--563.

\bibitem{ESV2021}
A. Enciso, A. Shao, and B. Vergara,
\emph{Carleman estimates with sharp weights and boundary observability
for wave operators with critically singular potentials},
J. Eur. Math. Soc. \textbf{23} (2021), 3459--3495.

\bibitem{Ervedoza2008}
S. Ervedoza,
Control and stabilization properties for a singular
heat equation with an inverse-square potential,
Comm. Partial Differential Equations
\textbf{33} (2008), no.~11, 1996--2019.

\bibitem{Evans2010}
L.~C. Evans,
\emph{Partial Differential Equations},
2nd ed., Graduate Studies in Mathematics, vol.~19,
American Mathematical Society, Providence, RI, 2010.

\bibitem{GesztesyPangStanfillBessel}
F.~Gesztesy, M.~M.~H.~Pang, J.~Stanfill,
On domain properties of Bessel-type operators,
\emph{Discrete Contin. Dyn. Syst. Ser. S}
\textbf{17} (2024), 1911--1946.

\bibitem{Gueye2014}
M.~Gueye, Exact boundary controllability of 1-D parabolic and
hyperbolic degenerate equations, SIAM J. Control Optim. (4)\textbf{52}(2014), 2037--2054.

\bibitem{HaakJamingWangWang2026}
B.~H.~Haak, P.~Jaming, M.~Wang, and Y.~Wang,
Curved Ingham inequalities for dispersive trigonometric sums,
arXiv:2603.15193 (2026).
\url{https://arxiv.org/abs/2603.15193}.

\bibitem{Haraux1989}
A.~Haraux,
S\'eries lacunaires et contr\^ole semi-interne des
vibrations d'une plaque rectangulaire,
J. Math. Pures Appl. (9) \textbf{68} (1989), 457--465.

\bibitem{HillairetMarzuola2023}
L. Hillairet and J. L. Marzuola,
Eigenvalue spacing for 1D singular Schr\"odinger operators,
Asymptot. Anal. \textbf{133} (2023), 267--289.

\bibitem{Ingham1936}
A.~E. Ingham:
Some trigonometrical inequalities with applications to the
theory of series,
Math. Z. \textbf{41} (1936), 367--379.

\bibitem{Jaffard1988}
S.~Jaffard,
Contr\^ole interne exact des vibrations d'une plaque
carr\'ee,
C. R. Acad. Sci. Paris S\'er. I Math.
\textbf{307} (1988), 759--762.

\bibitem{Jaffard1990}
S.~Jaffard,
Contr\^ole interne exact des vibrations d'une plaque
rectangulaire,
Port. Math. \textbf{47} (1990), 423--429.

\bibitem{JenaShao2026}
V. K. Jena and A. Shao,
\emph{Observability for wave equations with critically singular potentials},
\href{https://arxiv.org/abs/2608.28132}
{arXiv:2608.28132} (2026).

\bibitem{Kato1995}
T.~Kato:
\emph{Perturbation Theory for Linear Operators}.
Classics in Mathematics, Springer-Verlag, Berlin, 1995.
Reprint of the 1980 edition.

\bibitem{Komornik1992}
V.~Komornik,
On the exact internal controllability of a Petrowsky
system,
J. Math. Pures Appl. (9) \textbf{71} (1992), 331--342.

\bibitem{Komornik2008}
V.~Komornik: Explicit multidimensional Ingham--Beurling type estimates,
arXiv:0808.2255 (2008).

\bibitem{KomornikLoreti2005}
V.~Komornik and P.~Loreti,
\emph{Fourier Series in Control Theory},
Springer Monographs in Mathematics,
Springer, New York, 2005.

\bibitem{KST2010}
A.~Kostenko, A.~Sakhnovich, G.~Teschl:
Inverse eigenvalue problems for perturbed spherical
Schr\"odinger operators,
Inverse Problems \textbf{26} (2010), 105013, 14~pp.

\bibitem{KostenkoTeschl2011}
A.~Kostenko, G.~Teschl:
On the singular Weyl--Titchmarsh function of perturbed
spherical Schr\"odinger operators,
J. Differential Equations \textbf{250} (2011), 3701--3739.

\bibitem{KrugerTeschl2009}
H. Kr\"uger and G. Teschl,
Relative oscillation theory, weighted zeros of the
Wronskian, and the spectral shift function,
Commun. Math. Phys. \textbf{287} (2009), 613--640.

\bibitem{KuwaeShioya2003}
K.~Kuwae and T.~Shioya,
Convergence of spectral structures: a functional analytic
theory and its applications to spectral geometry,
\textit{Comm. Anal. Geom.} \textbf{11} (2003), no.~4,
599--673.

\bibitem{Remling2007}
C. Remling:
Finite propagation speed and kernel estimates for
Schr\"odinger operators.
Proc. Amer. Math. Soc. \textbf{135} (2007), 3329--3340.

\bibitem{Taeufer2023}
M.~T\"aufer,
Controllability of the Schr\"odinger equation on
unbounded domains without geometric control condition,
ESAIM Control Optim. Calc. Var.
\textbf{29} (2023), Paper No.~59.

\bibitem{TeschlQM2009}
G.~Teschl:
\emph{Mathematical Methods in Quantum Mechanics:
With Applications to Schr\"odinger Operators},
Graduate Studies in Mathematics, Vol.~99,
American Mathematical Society, Providence, RI, 2009.

\bibitem{TeschlODE2012}
G.~Teschl:
\emph{Ordinary Differential Equations and Dynamical Systems},
Graduate Studies in Mathematics, Vol.~140,
American Mathematical Society, Providence, RI, 2012.

\bibitem{TucsnakWeiss2009}
M.~Tucsnak and G.~Weiss,
\emph{Observation and Control for Operator Semigroups},
Birkh\"auser Advanced Texts: Basler Lehrb\"ucher,
Birkh\"auser, Basel, 2009.

\bibitem{VancostenobleZuazua2008}
J. Vancostenoble and E. Zuazua,
Null controllability for the heat equation with singular
inverse-square potentials,
J. Funct. Anal. \textbf{254} (2008), no.~7, 1864--1902.

\bibitem{VancostenobleZuazua2009}
J. Vancostenoble and E. Zuazua,
Hardy inequalities, observability, and control for the
wave and Schr\"odinger equations with singular potentials,
SIAM J. Math. Anal. \textbf{41} (2009), no.~4, 1508--1532.

\bibitem{Warnick2013}
C.~M. Warnick,
The massive wave equation in asymptotically AdS spacetimes,
Comm. Math. Phys. \textbf{321} (2013), 85--111.

\bibitem{WeiDuanXu2026}
Wei, L., Duan, Z., Xu, H.:
Uncertainty principle and geometric condition for the
observability of Schr\"odinger equations.
J. Fourier Anal. Appl. \textbf{32}, Article 28 (2026).

\bibitem{Yafaev2011}
D.~R. Yafaev,
The semiclassical limit of eigenfunctions of the
Schr\"odinger equation and the Bohr--Sommerfeld
quantization condition, revisited,
\textit{St. Petersburg Math. J.}
\textbf{22} (2011), no.~6, 1051--1067.

\bibitem{Zuazua2024}
E.~Zuazua:
\emph{Exact Controllability and Stabilization of the
Wave Equation},
UNITEXT, vol.~162, Springer, Cham, 2024.

\bibitem{Zworski2012}
M.~Zworski,
\textit{Semiclassical Analysis},
Graduate Studies in Mathematics, vol.~138,
American Mathematical Society, Providence, RI, 2012.

\end{thebibliography}
\end{document}